\documentclass{amsart}
\calclayout
\usepackage{amssymb,amsmath,amsfonts,epsfig,latexsym,tikz, amsrefs}
\usepackage{tikz-cd}
\usepackage{hyperref}
\usepackage{enumerate}
\usepackage{mathtools}
\usepackage{verbatim}
\usepackage{cleveref}
\usepackage[shortlabels]{enumitem}
\usepackage{subcaption}
\usepackage{ mathrsfs }

\usetikzlibrary{positioning}
\usetikzlibrary{matrix}
\usetikzlibrary{decorations}
\usetikzlibrary{decorations.pathreplacing, decorations.pathmorphing, angles,quotes}

\newtheorem{theorem}{Theorem}[section]
\newtheorem{definition}[theorem]{Definition}
\newtheorem{proposition}[theorem]{Proposition}
\newtheorem{lemma}[theorem]{Lemma}
\newtheorem{corollary}[theorem]{Corollary}

\newtheorem{question}[theorem]{Question}

\def\Q{\mathbb{Q}} 
\def\R{\mathbb{R}}

\def\Z{\mathbb{Z}}
\def\K{\mathcal{K}}

\def\1{\mathbf{1}}

\def\O{\mathcal{O}}
\def\cS{\mathcal{S}}

\def\<{\langle}
\def\>{\rangle}

\DeclareMathOperator{\id}{id}

\DeclareMathOperator{\Int}{Int}

\DeclareMathOperator{\rank}{rank}
\DeclareMathOperator{\Pyr}{Pyr}
\DeclareMathOperator{\Bipyr}{Bipyr}
\DeclareMathOperator{\Prism}{Prism}

\DeclareMathOperator{\JoinIdealLW}{JoinAdm}
\DeclareMathOperator{\JoinIdealCW}{CWJoinAdm}
\DeclareMathOperator{\Joincat}{\textbf{JoinAdm}}
\DeclareMathOperator{\JoinCWcat}{\textbf{CWJoinAdm}}
\DeclareMathOperator{\SFS}{Subdiv}
\DeclareMathOperator{\SFSCW}{CWSubdiv}
\DeclareMathOperator{\SFSCWcat}{\textbf{CWSubdiv}}
\DeclareMathOperator{\SFScat}{\textbf{Subdiv}}

\DeclareMathOperator{\EulPos}{\textbf{EulPos}}
\DeclareMathOperator{\CWPos}{\textbf{CWPos}}

\DeclareMathOperator{\Arr}{Arr}
\DeclareMathOperator{\Cyl}{Cyl}
\DeclareMathOperator{\CYL}{CYL}
\DeclareMathOperator{\MAP}{MAP}
\DeclareMathOperator{\CYLcat}{\textbf{CYL}}
\DeclareMathOperator{\MAPcat}{\textbf{MAP}}

\DeclareMathOperator{\Obj}{Obj}
\DeclareMathOperator{\Mor}{Mor}

\DeclareMathOperator{\face}{\mathcal{F}}
\DeclareMathOperator{\UH}{UH}

\newcommand{\Conv}[1]{\operatorname{Conv}\left\{{#1}\right\}}

\theoremstyle{definition}
\newtheorem{remark}[theorem]{Remark}
\newtheorem{example}[theorem]{Example}

\begin{document}
	
	\title[]{Subdivisions of lower Eulerian posets}
	\date{\today}
	\author{Alan Stapledon}

\address{Sydney Mathematics Research Institute, L4.42, Quadrangle A14, University of Sydney, NSW 2006, Australia}
	\email{astapldn@gmail.com}
	\begin{abstract}
		There is a  natural notion  of a  subdivision of a lower Eulerian poset called a strong formal subdivision, which abstracts the notion of a polyhedral subdivision of a polytope, or a proper, surjective morphism of fans. We show that there is a canonical  
		bijection between strong formal subdivisions and triples consisting of a lower Eulerian poset, a corresponding rank function, and a non-minimal element such that the join with any other element exists. 
		The bijection uses the non-Hausdorff mapping cylinder construction introduced by Barmak and Minian. 
		A corresponding bijection for $CW$-posets is given, as well as  an application to computing the $cd$-index of an Eulerian poset. 
		A companion paper explores applications to Kazhdan-Lusztig-Stanley theory.

	\end{abstract}
	
	\maketitle
	
	\vspace{-20 pt}

	\section{Introduction}
	
	A finite ranked poset is \emph{lower Eulerian} if contains a unique minimal element, and  every interval between distinct elements contains as many elements of even rank as odd rank.
	 A prototypical example of a lower Eulerian poset is the face lattice  $\face(P)$
	of a polytope $P$, 
	or, more generally, the face poset $\face(\cS)$ of a polyhedral subdivision $\cS$ of a polytope, i.e., the poset of all faces of $\cS$, including the empty face, under inclusion. Another example  is the face poset $\face(\Sigma)$ of a fan $\Sigma$ in a vector space $V$, i.e., the poset of all cones in $\Sigma$ under inclusion. 
	A \emph{formal subdivision} between lower Eulerian posets was 
	introduced by Stanley in \cite[Definition~7.4]{Stanley92} as a way to abstract the notion of subdivisions of polytopes, or refinements of fans. 
		The notion of a \emph{strong formal subdivision} 
	of lower Eulerian posets was introduced in \cite{KatzStapledon16}*{Definition~3.17} and further developed in \cite{DKTPosetSubdivisions}, replacing the notion of a formal subdivision, and also generalizing the notion of subdivisions of Gorenstein*-posets introduced in   \cite{EKDecompositionTheoremCDIndex}*{Definition~2.6}.  See Definition~\ref{def:sfs} for the definition.
	Let $\SFS$  be the set of strong formal subdivisions $\sigma: X \to Y$ between  lower Eulerian posets $X$ and $Y$  with rank functions $\rho_X$ and $\rho_Y$ respectively (see Definition~\ref{def:bijections}). 	
	For example, given a polyhedral subdivision $\cS$ of a polytope $P$, the corresponding strong formal subdivision is the 
	function
	$\sigma: \face(\cS) \to \face(P)$, where for every face $F$ in $\cS$,  $\sigma(F)$ is the smallest face of $P$ containing $F$. As another example, consider a linear map $\phi: V' \to V$ inducing a proper, surjective morphism between fans $\Sigma'$ and $\Sigma$ in real vector spaces $V'$ and $V$ respectively, e.g. $\phi$ is the identity and $\Sigma'$ is a refinement of $\Sigma$. Then the corresponding strong formal subdivision is the 
	function
	$\sigma: \face(\Sigma') \to \face(\Sigma)$, where for every cone $C'$ in $\Sigma'$,  $\sigma(C')$ is the smallest cone of $\Sigma$ containing $\phi(C')$. By using the standard technique of 
	placing a polytope at height $1$ in a higher dimensional space and taking cones generated by faces of either $\cS$ or $P$, we may view the first example as a special case of the second example (see Section~\ref{ss:fans}).

	Let $\Gamma$ be a lower Eulerian poset 
		with 
		unique minimial element $\hat{0}_\Gamma$.  
		We say that an element $q$ of $\Gamma$ is \emph{join-admissible} if for any element $z$ in $\Gamma$, 
		the join (or  least upper bound) $z \vee q$ exists (see Definition~\ref{def:joinadmissible}). 
		For example, $\hat{0}_\Gamma$ is join-admissible. 
	Let $\JoinIdealLW^\circ$ be the set of triples $(\Gamma, \rho_\Gamma, q)$,
	where $\Gamma$ is a lower Eulerian poset  with rank function $\rho_\Gamma$  and 
	$q$ is a join-admissible element of $\Gamma$ with $q \neq \hat{0}_\Gamma$ (see Definition~\ref{def:bijections}). The main result of this paper is the following.

\begin{theorem}(see Theorem~\ref{thm:mainsimplified})\label{thm:intromainsimplified}
	There is a canonical bijection between $\SFS$ and $\JoinIdealLW^\circ$. 
\end{theorem}

    The \emph{non-Hausdorff mapping cylinder} of an order-preserving map of posets (see Definition~\ref{def:nonHausdorffmc}) was introduced by Barmak and Minian
\cite{BMSimpleHomotopy}*{Definition~3.6} with applications in 
\cite{BarmakQuillensTheoremA,BarmakAlgebraicTopologyFinite,FMCylinder}.
With the notation above,
under the bijection of Theorem~\ref{thm:intromainsimplified}, $\Gamma$ is   the non-Hausdorff mapping cylinder  associated to a strong formal subdivision $\sigma$. Conversely, $\sigma$ is obtained by setting 
 $Y = \{ z \in \Gamma : q \le z \}$,  $X = \Gamma \smallsetminus Y$, and 
$\sigma(x) = x \vee q$ for all $x \in X$. See Definition~\ref{def:bijections} for details.

More generally, $\SFS$ and $\JoinIdealLW^\circ$ are naturally objects of corresponding categories $\SFScat$ and  $\Joincat^\circ$
respectively, and we have the following result.

\begin{theorem}(see Theorem~\ref{thm:main})\label{thm:intromain}
	There is a canonical isomorphism of categories between $\SFScat$ and $\Joincat^\circ$. 
\end{theorem}	

	Section~\ref{sec:examples} contains numerous examples illustrating the above theorems. This includes new methods for constructing strong formal subdivisions (see Section~\ref{ss:constructingsfs}), and descriptions of all strong formal subdivisions to the one and two element lower Eulerian posets (see Example~\ref{ex:B0}, Example~\ref{ex:B1}, Corollary~\ref{cor:boundaryEulerian}, Corollary~\ref{cor:nearEulerian}). 
	We now describe an example that is  explained in detail in Section~\ref{ss:facelatticepolytope}. Let $Q$ be a polytope. 
	Then  all elements of the face lattice $\face(Q)$ are join-admissible. 
	Consider the rank function $\rho_{\face(Q)}(G) = \dim G + 1$ for $G \in \face(Q)$. 
	Let $\Gamma = \face(Q)$, $\rho_\Gamma = \rho_{\face(Q)}$,  and let $q = F$ be a nonempty face of $Q$. 
	Then there is a corresponding projective, surjective (and, in particular, proper) morphism between fans $\Sigma'$ and $\Sigma$, where $\Sigma$ is a pointed cone, such that the corresponding strong formal subdivision $\sigma: \face(\Sigma') \to \face(\Sigma)$
	corresponds to $(\Gamma, \rho_\Gamma, q)$ under Theorem~\ref{thm:intromainsimplified}. 
	Moreover, every strong formal subdivision induced by a projective, surjective morphism between a fan and a pointed cone appears in this way. For example, if $F = Q$, then, after possible translation, we may assume that the origin lies in the interior of $Q$, and then $\Sigma'$ is the fan over the faces of $Q$ and $\Sigma = \{ 0 \}$ (see Example~\ref{ex:polytopeB0}). As a special case, if $\cS$ is a regular polyhedral subdivision of a polytope $P$, then there is a  polytope $Q$ and vertex $v$ of $Q$ such that the corresponding strong formal subdivision $\sigma: \face(\cS) \to \face(P)$ described earlier
	corresponds to $(\Gamma, \rho_\Gamma, q)$ under Theorem~\ref{thm:intromainsimplified}. 
	 For example, if $\cS$ is the trivial subdivision of $P$, then $Q$ is the pyramid over $P$ with apex $q$, and $\sigma$ is the identity function (see Example~\ref{ex:polytopeidentity}).

	With Theorem~\ref{thm:intromainsimplified} in hand and with the notation above, one wants to relate invariants associated with $\Gamma$ to invariants associated to $\sigma$. We describe this for the $cd$-index.
	A poset is \emph{Eulerian} if it is lower Eulerian and contains a unique maximal element. Let $B$ be an Eulerian poset such that $|B| > 1$. 
	The \emph{$cd$-index} $\Phi(B) = \Phi(B;c,d)$ is a polynomial in non-commuting variables $c$ and $d$ that encodes the number of chains of elements of $B$ with a prescribed sequence of ranks. See Section~\ref{ss:cdindexbackground} for the definition and \cite{StanleyFlagfVectors,BayerCDIndexSurvey} for surveys. 
	In \cite[p.485]{StanleyFlagfVectors}, Stanley also introduced the 
	notion of a near-Eulerian poset (see Definition~\ref{def:nearEulerian}). 
	For example, an Eulerian poset $B$ with $|B| > 1$ is near-Eulerian. 
	The \emph{local $cd$-index} $\ell^{\Phi}(B) = \ell^{\Phi}(B;c,d)$ of a near-Eulerian poset $B$ (see \eqref{eq:localcdindex}) was 
	 introduced by Ehrenborg and Karu  in a special case in order to prove a decomposition theorem for the $cd$-index of a Gorenstein*-poset \cite{EKDecompositionTheoremCDIndex}*{Theorem~5.6}. This decomposition theorem was later generalized in  \cite{DKTPosetSubdivisions}*{Theorem~1.1}, where the authors introduced the general definition of $\ell^{\Phi}(B)$. For example, if $B$ is Eulerian with $|B| > 1$, then $\ell^{\Phi}(B) = 0$. In general, 	the local $cd$-index may be considered a measure of how far $B$ is from being Eulerian.
	 
	 
	 Let $\Gamma$ be an Eulerian poset with $|\Gamma| > 1$, with unique minimal element $\hat{0}_\Gamma$, and with unique maximal element $\hat{1}_\Gamma$. Let $q$ be a join-admissible element of $\Gamma$ with 
	 $q \notin \{ \hat{0}_\Gamma, \hat{1}_\Gamma\}$. Consider the corresponding strong formal subdivision $\sigma: X \to Y$ determined by Theorem~\ref{thm:intromainsimplified}. Then $Y$ is Eulerian with unique minimal element $\hat{0}_Y = q$, and  unique maximal element $\hat{1}_Y = \hat{1}_\Gamma$. 
	 Given $z \le z'$ in a poset $B$, let $[z,z'] = \{ z'' \in B : z \le z'' \le z' \}$ denote the corresponding interval. 
	 Given an element $y$ in $Y$,  the interval $[y,\hat{1}_Y]$ is an Eulerian poset. 	We let
	 $X_{\le y} =  \sigma^{-1}([\hat{0}_Y,y])  = \{ x \in X : \sigma(x) \le y \}$ and 
	 $X_{< y} = \{ x \in X : \sigma(x) < y \}$.
	 For example,  $X_{\le \hat{1}_Y} = X$ and we write $\partial X = X_{< \hat{1}_Y}$. 
	Given a poset $B$, let $\overline{B}$ be obtained from $B$ by adjoining a unique maximal element. 	 
	 Then  $\{ X_{\le y} : y \neq \hat{0}_Y \}$ are near-Eulerian posets, $\overline{X_{\le \hat{0}_Y}} = [\hat{0}_\Gamma,q]$ is an Eulerian poset, and 
	$\{ \overline{X_{< y}} :  y \neq \hat{0}_Y \}$ are Eulerian posets \cite{DKTPosetSubdivisions}*{Proposition~2.9, Proposition~4.4, Remark~4.5}.

	\begin{proposition}(see Proposition~\ref{prop:cdsubdivision})\label{prop:introcdsubdivision}
		With the notation above, 
			\begin{equation*}
			\Phi(\Gamma) = \ell^{\Phi}(X) + \frac{1}{2} \left[ \Phi(\overline{\partial X}) c + 
			\Phi(\overline{X_{ \le \hat{0}_Y}}) c \Phi(Y) +  \sum_{\hat{0}_Y < y < \hat{1}_Y} \left(\ell^{\Phi}(X_{\le y})c + \Phi(\overline{X_{< y}}) d \right)\Phi([y,\hat{1}_Y]) \right].
		\end{equation*}	
	\end{proposition}

	When $\Gamma$ is a Gorenstein* poset, all terms in the above equation have nonnegative integer coefficients (see Example~\ref{ex:Gorensteinstarnonnegative}). 
In Example~\ref{ex:identitycdindex} and Example~\ref{ex:cdbipyramid} we recover formulas of Ehrenborg and Readdy for the $cd$-indices of the pyramid, bipyramid, and prism of an Eulerian poset $B$ with $|B| > 1$	\cite{ERCoproductscdindex}. 

	A second set of invariants associated to $\Gamma$ comes from Kazhdan-Lusztig-Stanley theory,  which has been developed, for example, in \cite{Stanley92,DyerHecke,BrentiTwistedIncidence,BrentiPKernels,KatzStapledon16,ProudfootAGofKLSpolynomials,FMVChowFunctions}. 
	See \cite{BPIntersectionCohomology} for a recent survey on the connections with intersection cohomology.
	On the other hand,
	one may consider local $h$-polynomials associated to a strong formal subdivision $\sigma$. These appear in connection with the decomposition theorem applied to toric varieties  \cite{Stanley92,KatzStapledon16,deCataldoMiglioriniMustata18,KaruRelativeHardLefschetz}, as well as in connection with the action of monodromy on the cohomology of the Milnor fiber of a nondegenerate hypersurface singularity \cite{Stapledon17,LPSLocalMotivic,SaitoMixed,STMonodromies}. We explore the connection between these two sets of invariants in detail in a companion paper \cite{StapledonKLSTheory}.

	We now consider examples of lower Eulerian posets coming from topology. 
	A finite poset is a \emph{$CW$-poset} if it is the poset of closed cells,  including the empty cell, of a regular  $CW$-complex under inclusion (see Definition~\ref{def:regularCWcomplex} and Definition~\ref{def:CWposet}). 
	See \cite{BjornerPosets} for an alternative characterization of $CW$-posets. 	The notion of a \emph{$CW$-regular subdivision} of $CW$-posets was introduced by Stanley  in \cite{Stanley92}. We generalize this to the notion of a \emph{strong $CW$-regular subdivision} of $CW$-posets (see Definition~\ref{def:CWsubdivision}). 
	A strong $CW$-regular subdivision is a strong formal subdivision (see Remark~\ref{rem:CWregularisformal}). 
	The \emph{rank}  of a strong formal subdivision $\sigma: X \to Y$ is $\rank(\sigma) = \rank(X) - \rank(Y)$ (see Definition~\ref{def:stronglysurjective}), and Stanley's $CW$-regular subdivisions are precisely the strong 
	$CW$-regular subdivisions of rank $0$. 
	
	Let $\SFSCW \subset \SFS$  be the set of strong $CW$-regular subdivisions $\sigma: X \to Y$ between  $CW$-posets $X$ and $Y$  with rank functions $\rho_X$ and $\rho_Y$ respectively, and let $\JoinIdealCW^\circ \subset \JoinIdealLW^\circ$ be the set of triples $(\Gamma, \rho_\Gamma, q)$,
	where $\Gamma$ is a  $CW$-poset  with rank function $\rho_\Gamma$  and $q$ is a join-admissible element of $\Gamma$ such that $q \neq \hat{0}_\Gamma$ and $\Gamma_{\ge q}$ is a $CW$-poset (see Definition~\ref{def:CWbijections}). We show that the canonical bijection in Theorem~\ref{thm:intromainsimplified} restricts to give the following result. 
	
	\begin{theorem}(see Theorem~\ref{thm:mainsimplifiedCW})\label{thm:intromainsimplifiedCW}
		There is a canonical bijection between $\SFSCW$ and $\JoinIdealCW^\circ$. 
	\end{theorem}
	
	Moreover,
		$\SFSCW$ and $\JoinIdealCW^\circ$ are naturally objects of corresponding categories $\SFSCWcat$ and  $\JoinCWcat^\circ$
	respectively, and 
	 the canonical isomorphism in Theorem~\ref{thm:intromain} restricts to give the following result. 
	
	\begin{theorem}(see Theorem~\ref{thm:mainCW})\label{thm:intromainCW}
		There is a canonical isomorphism of categories between $\SFSCWcat$ and $\JoinCWcat^\circ$. 
	\end{theorem}
	
	For example,  consider   a polyhedral subdivision $\cS$ of a polytope $P$, with corresponding strong formal subdivision $\sigma: \face(\cS) \to \face(P)$, and lower Eulerian poset $\Gamma$ corresponding to $\sigma$ under Theorem~\ref{thm:intromainsimplified}. Then $\sigma$ is a strong $CW$-regular subdivision of rank $0$ (see Example~\ref{ex:CWsubdivision}). 
		We have seen that if $\cS$ is a regular subdivision, then $\Gamma$ is the face lattice of a polytope.
	In general, Theorem~\ref{thm:intromainsimplifiedCW} implies that  $\Gamma$ is a $CW$-poset. See Example~\ref{ex:CWsubdivisionmappingcylinder} for an explicit description. 

 	More generally, consider a linear map $\phi: V' \to V$ inducing a proper, surjective morphism between fans $\Sigma'$ and $\Sigma$ in real vector spaces $V'$ and $V$ respectively, with  corresponding strong formal subdivision 
	$\sigma: \face(\Sigma') \to \face(\Sigma)$, and lower Eulerian poset $\Gamma$ corresponding to $\sigma$ under Theorem~\ref{thm:intromainsimplified}. Then $\sigma$ is a strong $CW$-regular subdivision of rank equal to $\dim \ker(\phi)$ (see Lemma~\ref{lem:CWproper}). We have seen that if $\Sigma$ is a pointed cone and the morphism is projective, then $\Gamma$ is the face lattice of a polytope. If we only assume that the morphism is projective, then $\Gamma$ is the face poset of a fan (see Example~\ref{ex:CWpropermappingcylinder}).  	In general, Theorem~\ref{thm:intromainsimplifiedCW} implies that  $\Gamma$ is a $CW$-poset. See Example~\ref{ex:CWpropermappingcylinder} for an explicit description of a corresponding regular $CW$-complex. 
	
	We mention the following corollary of the proof of Theorem~\ref{thm:intromainsimplifiedCW}. Let $\sigma: X \to Y$ be a strong $CW$-regular subdivision with $CW$-poset $\Gamma$ corresponding to $\sigma$ under Theorem~\ref{thm:intromainsimplifiedCW}. 
	Let $|\K|$ denote the underlying topological space associated with a regular $CW$-complex $\K$, and let $r = \rank(\sigma)$.  If $\K_X$, $\K_Y$, $\K_\Gamma$ are regular $CW$-complexes with face posets $X$, $Y$, $\Gamma$ respectively, then 
	there is a homeomorphism between $|\K_\Gamma|$ and the topological join of $|\K_Y|$ and an 
	$r$-dimensional simplex, which restricts to a homeomorphism between 
	$|\K_X|$ and the topological join of $|\K_Y|$ with the boundary of an $r$-dimensional simplex (see Lemma~\ref{lem:CWmappingcylinder}).
	For example, when $r = 0$, $|\K_\Gamma|$ is homeomorphic to the cone of $|\K_X| \cong |\K_Y|$.

	
	
	Finally, we mention that strong formal subdivisions are  defined more generally for locally Eulerian posets, i.e., posets on which every interval is Eulerian but may not have a unique minimal element \cite{KatzStapledon16}*{Definition~3.17}. We outline how Theorem~\ref{thm:intromainsimplified} and Theorem~\ref{thm:intromain} extend to this case at the end of Section~\ref{sec:generalizations}. 

	We briefly outline the contents of the paper. In Section~\ref{sec:background}, we recall background material to be used throughout the paper. In Section~\ref{sec:subdivisions} we further develop the theory of strong formal subdivisions and state our main results,  Theorem~\ref{thm:intromainsimplified} and Theorem~\ref{thm:intromain}.  In Section~\ref{sec:examples}, we present a number of examples illustrating the main results. In Section~\ref{sec:proof}, we prove our main results. In Section~\ref{sec:cdindex}, we apply our main results to $cd$-indices and prove Proposition~\ref{prop:introcdsubdivision}. 
	In Section~\ref{sec:CWsubdivisions}, we define strong $CW$-regular subdivisions and state and prove Theorem~\ref{thm:intromainsimplifiedCW} and Theorem~\ref{thm:intromainCW}.
	In Section~\ref{sec:generalizations}, we provide some open questions and discuss possible generalizations.  
				
				\emph{Notation and conventions}: All posets are finite posets, and all vector spaces are finite-dimensional. 	All cones in this paper will be polyhedral cones, i.e., the intersection of finitely many half-spaces in a vector space.
				All regular $CW$-complexes are finite. For convenience, we consider the single element poset to be a $CW$-poset equal to the face poset of the empty regular $CW$-poset. 
				If $Z$ and $Z'$ are homeomorphic topological spaces, then we write $Z \cong Z'$. 
				 We write $|S|$ for the cardinality of a finite set $S$. 
				
				%
				%

\section{Background}\label{sec:background}

In this section we provide background material on posets, polytopes, and fans that will be used throughout the paper.

\subsection{Background on posets}\label{ss:posets}

We first recall some background on posets and refer the reader to  \cite[Chapter~3]{StanleyEnumerative} and \cite{DKTPosetSubdivisions}*{Section~2}
for more details. All posets in this paper will be finite posets, and we often identify isomorphic posets when it is convenient. The empty set is considered to be a poset. 

Let $B$ be a  poset. An interval of $B$ is a closed interval unless specified otherwise. Given elements $z \le z'$ in $B$, we write 
$[z,z'] = \{ z'' \in B : z \le z'' \le z' \}$, $[z,z') = \{ z'' \in B : z \le z'' < z' \}$, 
and $(z,z') = \{ z'' \in B : z < z'' < z' \}$. 
We say that $z'$ \emph{covers} $z$ if $z < z'$ and $z \le z'' \le z'$ implies that $z'' \in \{ z,z'\}$.  
If $B$ contains a unique minimal element, we will denote it $\hat{0}_B$, or simply $\hat{0}$. 
Similarly, if $B$ contains a unique maximal element, we will denote it $\hat{1}_B$, or simply $\hat{1}$. 
We let $\overline{B} = B \cup \{\hat{1}\}$ be the poset obtained from $B$ by adjoining a  maximal element $\hat{1}$. 
A chain $z_0 < z_1 < \cdots < z_s$ in $B$ has length $s$. 
A maximal chain is a chain that is not a proper subchain of another chain in $B$. 


A \emph{lower order ideal} of a poset $B$ is a subset $I \subset B$ such that if $z \le z'$ in $B$ and $z' \in I$, then $z \in I$. Similarly, an \emph{upper order ideal} of $B$ is a subset $I \subset B$ such that if $z \le z'$ in $B$ and $z \in I$, then $z' \in I$.
Given a subset $S$ of $B$,  the lower order ideal generated by $S$ is $\{ z \in B : z \le z' \textrm{ for some } z' \in S\}$, and the  upper order ideal generated by $S$ is $\{ z \in B : z' \le z \textrm{ for some } z' \in S\}$. For any $z \in B$, denote the lower order ideal   generated by $z$ by $B_{\le z} = \{ z' \in B : z' \le z \}$, and denote the upper order ideal generated by $z$ by $B_{\ge z} = \{ z' \in B : z \le z' \}$.  Also, let $B_{>z} = \{ z' \in B : z < z' \}$.

A function $\sigma: X \to Y$ between posets $X$ and $Y$ is \emph{order-preserving} if $x \le x'$ in $X$ implies that $\sigma(x) \le \sigma(x')$ in $Y$. 
In that case, given an element $y$ in $Y$, we write 
$X_{\le y}$ and $X_{< y}$ to denote the lower order ideals $\sigma^{-1}(Y_{\le y})$ and $\sigma^{-1}(Y_{< y})$ respectively. 
Throughout, we let $\id_B: B \to B$ be the identity function on $B$.  


\begin{definition}
	Let $B$ be a nonempty poset. A function $\rho_B: B \to \Z$ is a \emph{rank function} for $B$ if $\rho_B(z') = \rho_B(z) + 1$ whenever $z'$ covers $z$. For any $z \le z'$ in $B$, we write $\rho_B(z,z') := \rho_B(z') - \rho_B(z)$. 
	If a rank function for $B$ exists, then we call $B$ a \emph{ranked poset}. The \emph{rank} $\rank(B)$ of $B$ is the length of the longest maximal chain in $B$. 
	%
\end{definition}

Note that if $\rho_B: B \to \Z$ is a rank function,  then 
then every maximal chain of an interval $[z,z']$ of $B$ has length
$\rho_B(z,z')$. 
For any integer $s$, we define a shifted rank function 
$\rho_B[s] : B \to \Z$ by $\rho_B[s](z) = \rho_B(z) + s$ for all $z \in B$. 
Observe that if $B$ contains a unique minimal element $\hat{0}_B$, then any two rank functions $\rho_B$ and $\rho_B'$ for $B$ satisfy $\rho_B' = \rho_B[s]$, where $s = \rho_B'(\hat{0}_B) - \rho_B(\hat{0}_B) \in \Z$.   In this case, the \emph{natural rank function} is the unique rank function with $\rho_B(\hat{0}_B) = 0$. 
Suppose that $X$ and $Y$ are ranked posets with rank functions $\rho_X$ and $\rho_Y$ respectively. A function $\sigma: X \to Y$ is \emph{rank-increasing} if $\rho_X(x) \le \rho_Y(\sigma(x))$ for all $x$ in $X$.  


\begin{definition}\label{def:locallyEulerian}
	Let $B$ be a ranked poset with rank function $\rho_B: B \to \Z$. 
	Then $B$ is \emph{locally Eulerian} if for any $z < z'$ in $B$,  $\sum_{z \le z'' \le z'} (-1)^{\rho_B(z'')} = 0$, i.e., the interval $[z,z']$ contains as many elements of even rank as odd rank.
	%
	A locally Eulerian poset is 
	\emph{lower Eulerian} if it contains 
	a unique minimal element. 
	A lower Eulerian poset is \emph{Eulerian} if it contains 
	a unique maximal element. 
\end{definition}

Since the restriction of any two rank functions for $B$ to an interval agree up to a shift, it follows that 
Definition~\ref{def:locallyEulerian} is independent of the choice of rank function $\rho_B$. 
For example, 
if $I$ is a nonempty lower order ideal of a lower Eulerian poset, 
then $I$ is lower Eulerian.  
If $B$ is locally Eulerian and $z \in B$, then  $B_{\ge z}$ is lower Eulerian.

\begin{example}\label{ex:boolean}
	For a nonnegative integer $n$, let $B_n$ be the poset of all (possibly empty) subsets of $[n] := \{ 1,\ldots, n\}$. 
	Then $B_n$ is Eulerian  of rank $n$ and is called the \emph{Boolean algebra} on $n$ elements. 	
	The natural rank function $\rho_{B_n}: B_n \to \Z$ is given by
	$\rho_{B_n}(J) = |J|$  for all $J \subset [n]$. For example, $B_0$ is the one element poset, and $B_1 = \{ \hat{0}, \hat{1} \}$. 
\end{example}

If $B$ is an Eulerian poset of positive rank, then the boundary of $B$ is $\partial B = B \smallsetminus \{ \hat{1} \}$. We recall the following generalization introduced by Stanley. 

\begin{definition}\cite[p.485]{StanleyFlagfVectors}\label{def:nearEulerian}
	A nonempty poset $B$ is \emph{near-Eulerian} if  there exists an Eulerian poset of positive rank $\tilde{\Sigma} B$ and a maximal element $\hat{z} \in \partial (\tilde{\Sigma} B)$  such that $B = \partial (\tilde{\Sigma} B) \smallsetminus \{ \hat{z} \}$. The \emph{boundary} $\partial B$ is the lower order ideal of $B$ generated by all elements $z \in B$ 
	such that $B_{\ge z} = B_1$, i.e., precisely one element of $B$ is strictly greater than $z$. We call $\tilde{\Sigma} B$ the \emph{semisuspension} of $B$. 
\end{definition}

With the notation of Definition~\ref{def:nearEulerian}, let $\rho_{\tilde{\Sigma} B}: \tilde{\Sigma} B \to \Z$ be a rank function. Since every element $z \in \tilde{\Sigma} B$ with $\rho_{\tilde{\Sigma} B}(z,\hat{1}_{\tilde{\Sigma} B}) = 2$ is covered by precisely two elements, 
it follows that   $B$ is locally Eulerian with $\rank(B) = \rank(\tilde{\Sigma} B) - 1$, and 
$\partial B = \{ z \in \tilde{\Sigma} : z < \hat{z} \}$. In particular, $\tilde{\Sigma} B$ is uniquely determined from $B$ by first adjoining an element that is greater than all elements in $\partial B$, and then adjoining a maximal element.

For example, if $B$ is Eulerian with positive rank, then $B$ is near-Eulerian. Observe that a near-Eulerian poset has an even number of elements, while the boundary of an Eulerian poset of positive rank has an odd number of elements. In particular, the one element poset $B_0 = \partial B_1$ is Eulerian, but not near-Eulerian.

Let $B$ be a poset and consider two elements $z$ and $z'$ in $B$. 
If $B_{\ge z} \cap B_{\ge z'}$ is nonempty and contains a unique minimal element, then this element is the \emph{join} (or least upper bound) of $z$ and $z'$ and is denoted $z \vee z'$.  
In this case, we say that the join  
$z \vee z'$ exists. We need the following well-known property.

\begin{remark}\label{rem:joinorderpreserving}
	If $z \le z'$,  and if 
	$z \vee q$ and $z' \vee q$ exist for some $q \in B$, then $z \vee q \le z' \vee q$. To see this, observe that $B_{\ge z'} \subset B_{\ge z}$ and hence $z' \vee q \in B_{\ge z'} \cap B_{\ge q} \subset B_{\ge z} \cap B_{\ge q} = B_{\ge z \vee q}$. 
\end{remark}

We will use the following notation. 

\begin{definition}\label{def:joinadmissible}
	An element $q$ of a poset $B$ is \emph{join-admissible} if 
	$z \vee q$ exists for all $z \in B$.	 
\end{definition}

Similarly, one can define the meet (or greatest lower bound) $z \wedge z'$, when it exists. If $z \vee z'$ and $z \wedge z'$ exist for any two elements $z,z'$ of $B$, then $B$ is a \emph{lattice}. For example, the Boolean algebra $B_n$ 
is a lattice. 
By \cite{StanleyEnumerative}*{Proposition~3.3.1}, $B$ is a lattice if and only if $B$ contains a unique minimal element and $z \vee z'$  exists for any two elements $z,z'$ in $B$. 
As a nonexample, given any Eulerian poset of rank at least $2$, the semisuspension 
$\tilde{\Sigma} B$ is not a lattice. Indeed, if $z$ and $z'$ are distinct maximal elements of $\partial B$, then $z \vee z'$ does not exist in $\tilde{\Sigma} B$.

We next consider several operations on posets. If $B$ is a poset, then the \emph{dual poset} $B^*$ is the poset with the same elements as $B$ and with all orderings reversed, i.e., $z \le z'$ in $B$ if and only if $z' \le z$ in $B^*$. Then $B$ is Eulerian if and only if $B^*$ is Eulerian. For example, $B_n^* = B_n$. 

If $B$ and $B'$ are posets, then the \emph{direct product} is the poset $B \times B' = \{ (z,z') : z \in B, z' \in B' \}$ with $(z_1,z_1') \le (z_2,z_2')$ in $B \times B'$ if and only if $z_1 \le z_2$ in $B$ and $z_1' \le z_2'$ in $B'$. If  $\sigma: X \to Y$ and $\sigma': X' \to Y'$ are  functions between posets $X$ and $Y$, and $X'$ and $Y'$ respectively, then let 
$\sigma \times \sigma': X \times X' \to Y \times Y'$ be defined by $(\sigma \times \sigma')(x,x') = (\sigma(x),\sigma(x'))$ for all $(x,x') \in X \times X'$. 
If $B$ and $B'$ are ranked posets with rank functions $\rho_B$ and $\rho_{B'}$ respectively, then $B \times B'$ is a ranked poset with rank function
$\rho_{B \times B'}(z,z') = \rho_B(z) + \rho_{B'}(z')$ for all $z \in B$ and $z' \in B'$. In this case, 
 $\rank(B \times B') = \rank(B) + \rank(B')$. 
The \emph{pyramid} of $B$ is  $\Pyr(B) = B \times B_1$.  
For example, $B_n \times B_{n'} = B_{n + n'}$ and $\Pyr(B_n) = B_{n + 1}$. 
If $B$ and $B'$ are locally Eulerian (respectively lower Eulerian or Eulerian), then $B \times B'$ is locally Eulerian (respectively lower Eulerian or Eulerian).
Also,  $(B \times B')^* = B^* \times (B')^*$, and, in particular,  $\Pyr(B)^* = \Pyr(B^*)$. 
For example, if $B$ is an Eulerian poset of positive rank, then $\Pyr(\partial B)$ is near-Eulerian with semisuspension $\Pyr(B)$ and boundary $\partial B \times \{ \hat{0} \}$. 


Let $B$ and $B'$ be Eulerian posets of positive rank with rank functions $\rho_B$ and $\rho_{B'}$ respectively. The \emph{dual diamond product} $B \diamond^\ast B'$ is the Eulerian poset of positive rank uniquely determined by the condition that $\partial(B \diamond^\ast B') = \partial B \times \partial B'$, with rank function $\rho_{B \diamond^\ast B'}$ restricting to 
$\rho_{\partial B \times \partial B'}$. 
%
The \emph{bipyramid} of $B$ is $\Bipyr(B) = B \diamond^\ast B_2$.
The \emph{diamond product} is the Eulerian poset $B \diamond B' = (B^* \diamond^\ast B^*)^*$, and the \emph{prism} of $B$ is $\Prism(B) = B \diamond B_2 = \Bipyr(B^*)^*$. Equivalently, $B \diamond B'$ is the Eulerian poset of positive rank uniquely determined by the condition that 
$(B \diamond B') \smallsetminus \{ \hat{0}_{B \diamond B'}\} = (B \smallsetminus \{ \hat{0}_{B}\}) \times (B' \smallsetminus \{ \hat{0}_{B'}\})$. It has a rank function $\rho_{B \diamond B'}$ that restricts to 
$\rho_{(B \smallsetminus \{ \hat{0}_{B}\}) \times (B' \smallsetminus \{ \hat{0}_{B'}\})}$. 


If $B$ is an Eulerian poset of positive rank and $B'$ is a lower Eulerian poset, then the \emph{star product} $B \ast B'$ is the disjoint union of $\partial B$ 
and $B' \smallsetminus \{ \hat{0}_{B'} \}$ with $z \le z'$ in 
$B \ast B'$  if and only if one of the following conditions hold:
\begin{enumerate}
	\item $z,z' \in \partial B$ and $z \le z'$ in $B$, or,
	\item $z,z' \in B' \smallsetminus \{ \hat{0}_{B'} \}$ and $z \le z'$ in $B'$, or,
	\item $z \in \partial B$ and $z' \in B' \smallsetminus \{ \hat{0}_{B'} \}$.
\end{enumerate}
One may verify that $B \ast B'$ is lower Eulerian. In particular,  $B \ast B'$ is Eulerian of positive rank if and only if $B'$ is Eulerian of positive rank.  Moreover, if $\rho_B$ and $\rho_{B'}$ are the natural rank functions of $B$ and $B'$ respectively, then the natural rank function $\rho_{B \ast B'}$ of  $B \ast B'$ is given by
\[
\rho_{B \ast B'}(z) = \begin{cases}
	\rho_B(z) &\textrm{ if } z \in \partial B, \\
	\rho_{B'}[\rank(B) - 1](z) &\textrm{ if } z \in B' \smallsetminus \{ \hat{0}_{B'} \}. 
\end{cases}
\]
This construction first appeared in \cite{StanleyFlagfVectors}. For example,  $B \ast B_0 = \partial B$, $B_1 \ast B' = B'$, $B \ast B_1 = B$,  and $B \ast B_2 = \tilde{\Sigma} B$.

Finally, we briefly discuss some aspects of the topology of posets. 
The \emph{order complex} 
$\O(B)$ of a poset $B$
is the simplicial complex with vertices given by the elements of $B$ and faces given by chains of elements of $B$. If $x < x'$ in $B$, then we write $\O(x,x') = \O((x,x'))$.  
Let $\tilde{H}_i(Z;\Q)$ denote the $i$th reduced homology group over $\Q$ of a topological space $Z$.

\begin{definition}\label{def:Gorensteinstar}
	Let $B$ be an Eulerian poset of positive rank with rank function $\rho_B$.
	Then $B$ is \emph{Gorenstein*} if for any $x < x'$ in $B$ with $\rho_B(x,x') \ge 2$, 
	\[
	\tilde{H}_i(\O(x,x');\Q) = \begin{cases}
		k &\textrm{ if } i = \rho_B(x,x') - 2, \\
		0 &\textrm{ otherwise. }
	\end{cases}
	\]
\end{definition}

Equivalently, $B$ is Gorenstein* if and only if $B$ is Eulerian of positive rank and Cohen-Macaulay over $\Q$ \cite{StanleySurveyEulerian}. Equivalently, $B$ is Gorenstein* if it contains unique minimal and maximal elements, and  $\O(\hat{0}_B,\hat{1}_B)$ is a homology sphere over $\Q$ of dimension $\rank(B) - 2$ \cite{EKDecompositionTheoremCDIndex}*{Definition~2.1}. 
For example, the face lattice of a polytope is Gorenstein* (c.f. Example~\ref{ex:CWEulerian}). 
By Definition~\ref{def:Gorensteinstar}, any interval of positive rank in a 
Gorenstein* poset is a Gorenstein* poset.

Let $B$ be a near-Eulerian poset. Then $B$ is \emph{near-Gorenstein*} if the semisuspension $\tilde{\Sigma} B$ is Gorenstein*  \cite{EKDecompositionTheoremCDIndex}*{Lemma~2.3}. 
In this case, $\overline{\partial B}$ is an interval of positive rank in $\tilde{\Sigma} B$ and hence is Gorenstein*. 
For example, Gorenstein* posets  are near-Gorenstein* \cite{EKDecompositionTheoremCDIndex}*{Definition~2.2}. 
	As another example, the face poset of a polyhedral subdivision of a polytope is near-Gorenstein* 
(see  Example~\ref{ex:CWball}). 


	\subsection{Background on polytopes}\label{ss:polytopes} We recall some background on polytopes. We refer the reader to \cite{DRSTriangulations10} and \cite{ZieglerLectures}  for more details. All vector spaces in this paper are finite-dimensional.

	Let $n$ be a nonnegative integer and let $P \subset V$ be an $(n - 1)$-dimensional polytope in a real vector space $V$. 
	The \emph{face lattice} $\face(P)$ is the poset of faces of $P$ ordered by inclusion. Here  the empty face is considered a face of $P$ of   dimension $-1$. Then $\face(P)$ is a lattice and is Eulerian of rank $n$ with natural rank function $\rho_{\face(P)}(F) = \dim F + 1$. For example, 
	if $\Delta^{n - 1}$ denotes the simplex with $n$ vertices, then $\face(\Delta^{n - 1}) = B_n$.

	Given a face $F$ of $P$, the \emph{quotient polytope} $P/F$ is a polytope (well-defined up to projective transformation) such that $\face(P/F) = \face(P)_{\ge F}$ \cite{BMIntersectionHomologyKalai}. 
	For example, if $F$ is a vertex and $L = \{ v \in V : h(v) = 0 \}$ is an affine hyperplane for some affine function $h$ such that $L \cap P = F$ and $P \subset \{ v \in V : h(v) \ge  0 \}$,   then $P/F$ is the intersection of $P$ with  $\{ v \in V : h(v) = \epsilon \}$ for any $0 < \epsilon \ll  1$. 
	Let  $P' \subset V'$ be  a polytope in  a real vector space $V'$. 
	Let $0_V$ and $0_{V'}$ denote the origin in $V$ and $V'$ respectively.
	The \emph{free join} $P \star P'$ is the convex hull of 
	$P \times \{ 0_{V'} \} \times \{ 0 \}$ and 
	$\{ 0_V \} \times P'  \times \{ 1 \}$ in $V \oplus V' \oplus \R$.
	The \emph{pyramid} $\Pyr(P)$ is the free join of $P$ and $P' = \{ 0 \} \subset \R$.  With the notation of Section~\ref{ss:posets},
	$\face(P \star P') = \face(P) \times \face(P')$ and $\face(\Pyr(P)) = \Pyr(\face(P))$. The \emph{apex} of $\Pyr(P)$ is the vertex 
	$(0_V,0,1)$. 
	The \emph{product} $P \times P'$  is $\{ (z,z') \in V \oplus V' : z \in P, z' \in P' \}$, and the \emph{prism} $\Prism(P)$ is the product of $P$ and $P' = [-1,1] \subset \R$. Then $\face(P \times P') = \face(P) \diamond \face(P')$ and  $\face(\Prism(P)) = \Prism(\face(P))$. 
	The \emph{free sum} (also known as the \emph{direct sum}) $P \oplus P'$  of $P$ and $P'$ is the convex hull of $P \times \{ 0_{V'} \}$ and 
	$\{ 0_V \} \times P'$ in $V \oplus V'$.  Assume that $0_V$ lies in the relative interior of $P$,  and $0_{V'}$ lies in the relative interior of $P'$.  Then $\face(P \oplus P') = \face(P) \diamond^\ast \face(P')$. In this case, the
	\emph{bipyramid} $\Bipyr(P)$ of $P$ is the free sum of $P$ and $P' = [-1,1] \subset \R$, 
	and  $\face(\Bipyr(P)) = \Bipyr(\face(P))$. The \emph{apices} of $\Bipyr(P)$ are $(0_V,1)$ and $(0_V,-1)$.

	A \emph{polyhedral subdivision} $\cS$ of $P$ is a finite collection of polytopes whose union is $P$, such that if $Q \in \cS$, then any face of $Q$ lies in $\cS$, and if $Q' \in \cS$, then $Q \cap Q'$ is a common face of $Q$ and $Q'$. The \emph{face poset} $\face(\cS)$ of $\cS$ is the poset of elements of $\cS$ ordered by inclusion. Then $\face(\cS)$ is a lower Eulerian poset with natural rank function $\rho_{\face(P)}(F) = \dim F + 1$. If $\cS'$ is a polyhedral subdivision of $P$, then $\cS'$ is a \emph{refinement} of $\cS$ if every polytope in $\cS'$ is contained in a polytope of $\cS$.  In that case, we have an induced function
	\[
	\sigma: \face(\cS') \to \face(\cS),
	\]
	where $\sigma(F')$ is the smallest element of $\cS$ containing $F'$. 
	For example, the trivial subdivision of $P$ is the subdivision consisting of the faces of $P$, with facet poset $\face(P)$.  
	
	Let $A \subset P$ be a finite subset containing the vertices of $P$. Given a function $\omega: A \to \R$, the \emph{upper convex hull} $\UH(\omega)$ of $\omega$ is the convex hull of 
	$\{ (a,\lambda) : a \in A, \omega(a) \le \lambda \} \subset V \oplus \R$. Let $\pi: V \oplus \R \to V$ denote projection onto the first coordinate. Let $\cS(\omega)$ be the polyhedral subdivision $\{ \pi(F) : F \textrm{ is a bounded face of } \UH(\omega)\}$ of $P$. A polyhedral subdivision $\cS$ of $P$ is \emph{regular} if $\cS = \cS(\omega)$ for some function $\omega$. 
	Let $B \subset P$ be a finite subset and let $A$ be the union of $B$ and the vertices of $P$. Define $\omega_B: A \to \R$ by $\omega(a) = -1$ if $a \in B$ and $\omega(a) = 0$ otherwise. Then $\cS(\omega_B)$ is the \emph{pulling refinement} of $P$ by $B$.

	\subsection{Background on fans}\label{ss:fans}
	
	We recall some background on cones and fans. We refer the reader to \cite{CoxLittleSchenck11} and \cite{FultonIntroductionToricVarieties}  for more details. Recall that all vector spaces in this paper are finite-dimensional.
	All cones in this paper will be polyhedral cones. 
	
	Let $C$ be an $n$-dimensional (polyhedral) cone in a real vector space $V$. The \emph{face lattice} $\face(C)$ is the poset of faces of $C$ ordered by inclusion.  Then $\face(C)$ is a lattice and is Eulerian of rank $n$. 
	We say that $C$ is \emph{pointed} if the minimal face of $C$ is the origin in $V$. 
	If $C$ is pointed, then the natural rank function is given by $\rho_{\face(C)}(F) = \dim F$. 
	Given a polyhedron $Q$ contained in $V$, let $C(Q)$ be the smallest cone containing $Q$.

	Given an $(n - 1)$-dimensional polytope $P$ in $V$, consider the $n$-dimensional pointed cone $C(P \times \{ 1\})$, where  $P \times \{ 1 \}$ lies in $V \oplus \R$. Then $\face(C(P \times \{1\})) = \face(P)$. 
	Conversely, if $C$ is a pointed cone and $H$ is an affine hyperplance 
	that intersects all the rays of $C$ at a single nonzero point, then $P = C \cap H$ is a polytope and $\face(C) = \face(P)$.
	If $C = \{ 0 \}$, then $P$ is the empty polytope. 
	We conclude that we may view face lattices of pointed cones as  face lattices of polytopes and vice versa. 
	
	\begin{example}\label{ex:polyhedron}
		Let $Q$ be a polyhedron in a real vector space $V$. Then $Q$ can be written as the Minkowski sum of a (not necessarily unique) polytope and a unique cone $C_Q$, called the \emph{recession cone} of $Q$. Then $C(Q \times \{1\}) \subset V \oplus \R$ is a pointed cone with faces equal to $\{ C(F \times \{ 1 \}):  F \textrm{ face of } Q \}$, together with the nonzero faces of $C_Q \times \{ 0 \}$. 
	\end{example}
	

	
	Recall that a  \emph{fan} $\Sigma$ in a real vector space $V$ is a finite collection of pointed cones such that 
	\begin{enumerate}
		\item if $C \in \Sigma$ and $F$ is a face of $C$, then $F \in \Sigma$, and,
		\item if $C,C' \in \Sigma$, then $C \cap C'$ is a common face of $C$ and $C'$. 
	\end{enumerate}
	The \emph{support} $|\Sigma|$ of $\Sigma$ is the union of the cones of $\Sigma$ in $V$.  
	We say that $\Sigma$ is \emph{complete} if $|\Sigma| = V$. 
	The \emph{face poset} $\face(\Sigma)$  is the poset of cones in $\Sigma$ ordered by inclusion.  Then $\face(\Sigma)$ is lower Eulerian with natural rank function $\rho_{\face(\Sigma)}(C) = \dim C$. 

	Let $\Sigma'$ be a fan in a real vector space $V'$. Let $\phi: V' \to V$ be a linear map. Then $\phi$ defines a \emph{morphism of fans} 
	from $\Sigma'$ to $\Sigma$ 
	if for every cone $C'$ in $\Sigma'$ there is a cone $C$ in $\Sigma$ such that $\phi(C') \subset C$. 
	In that case, we have an induced function
	\[
	\sigma: \face(\Sigma') \to \face(\Sigma),
	\]
	where $\sigma(C')$ is the smallest element of $\Sigma$ containing $\phi(C')$.
	A morphism of fans is \emph{proper} if $\phi^{-1}(|\Sigma|) = |\Sigma'|$, and is \emph{surjective} if $\phi$ is surjective.
	A proper morphism of fans is \emph{projective} if there exists
	$p: |\Sigma'| \to \R$ that is 	piecewise linear  with respect to $\Sigma'$, and  
	such that the restriction of $p$ to $\phi^{-1}(C)$ is strictly convex with respect to the restriction $\Sigma'|_{\phi^{-1}(C)}$ for every maximal cone $C$ in $\Sigma$.

	For example, if $\Sigma'$ is complete and $V = \Sigma = \{ 0 \}$, then the unique linear map $\phi: V' \to V$ is a  proper, surjective morphism of fans.
	In that case, we say that $\Sigma'$ is \emph{projective} if the morphism of fans induced by $\phi$ is projective.    
	If $V' = V$ and the identity map on $V$ induces a proper morphism of fans from $\Sigma'$ to $\Sigma$, then we say that $\Sigma'$ is a \emph{refinement} of $\Sigma$. Note that a refinement of fans is proper and surjective. 
	If $\cS$ is a polyhedral subdivision of a polytope $P$, let $\Sigma(\cS)$ be the fan with cones $\{ C(F \times \{ 1\}) : F \in \cS \}$. If $\cS'$ is a polyhedral subdivision refining $\cS$, then $\Sigma(\cS')$ is a refinement of $\Sigma(\cS)$. If $\cS$ is the trivial subdivision of $P$, then $\cS'$ is a regular polyhedral subdivision if and only if the corresponding refinement $\Sigma(\cS') \to C(P \times \{ 1\})$ is projective.


\section{Subdivisions of lower Eulerian posets}\label{sec:subdivisions}

In this section, we state our main results. In Section~\ref{ss:basic} we recall the definition and basis properties of strong formal subdivisions, and prove some lemmas to be used later. In Section~\ref{ss:bijection} we state Theorem~\ref{thm:mainsimplified}, and in Section~\ref{ss:functorial} we state Theorem~\ref{thm:main}. 


\subsection{Basic properties of strong formal subdivisions}\label{ss:basic}

We now recall the definition and describe some properties of strong formal subdivisions. We refer the reader to \cite{KatzStapledon16}*{Section~3} and \cite{DKTPosetSubdivisions}*{Section~2} for more details. 

Suppose that $X$ and $Y$ are ranked posets with rank functions $\rho_X$ and $\rho_Y$ respectively. Recall from Section~\ref{ss:posets} that a function $\sigma: X \to Y$ is order-preserving if $x \le x'$ in $X$ implies that $\sigma(x) \le \sigma(x')$ in $Y$, and is rank-increasing if $\rho_X(x) \le \rho_Y(\sigma(x))$ for all $x$ in $X$.


\begin{definition}\cite[Definition~3.16]{KatzStapledon16}\label{def:stronglysurjective}
	Let  $\sigma: X \to Y$ be an order-preserving, rank-increasing function between 
	ranked posets $X$ and $Y$  with rank functions $\rho_X$ and $\rho_Y$ respectively. Then $\sigma$ is  \emph{strongly surjective} if $\sigma$ is surjective and for all $x \in X$ and $y \in Y$ with $\sigma(x) \leq y$, there exists $x' \in X$ such that $x \le x'$, $\rho_X(x') = \rho_Y(y)$, and $\sigma(x') = y$. In that case, the \emph{rank} of $\sigma$ is $\rank(\sigma) = \rank(X) - \rank(Y)$. 
\end{definition}

Consider the setup of Definition~\ref{def:stronglysurjective}. 
By Lemma~\ref{lem:liftchains} below, any maximal chain in $Y$ is the image under $\sigma$ of a maximal chain in $X$.  In particular, $\rank(\sigma) = \rank(X) - \rank(Y)$ is a nonnegative integer. 

\begin{lemma}\label{lem:liftchains}
	Let $\sigma: X \to Y$ be a strongly surjective function between ranked posets $X$ and $Y$ with rank functions $\rho_X$ and $\rho_Y$ respectively. 
	Then every maximal chain $y_0 < y_1 < \cdots < y_s$ in $Y$ is the image under $\sigma$ of a maximal chain $x_0 < x_1 < \cdots <  x_r$ in $X$ such that  $r - s = \rho_Y(y_0) - \rho_X(x_0)$. 
	
	
\end{lemma}
\begin{proof}
	Since $\sigma$ is surjective, there exists $x_0$ in $X$ such that $\sigma(x_0) = y_0$. If $x \le x_0$, then $\sigma(x) \le \sigma(x_0) = y_0$. Since $y_0$ is minimal in $Y$, it follows that $\sigma(x) = y_0$. Hence we may assume that $x_0$ is a minimal element in $X$.  
	Strong surjectivity implies that there exists a  chain $x_0 < x_1 < \cdots < x_m$ such that $\sigma(x_0) = \sigma(x_1) = \cdots = \sigma(x_m) = y_0$, $\rho_X(x_m) = \rho_Y(y_0)$, and $m = \rho_X(x_m) - \rho_X(x_0) = \rho_Y(y_0) - \rho_X(x_0)$.

	Strong surjectivity then allows us to construct a chain $x_m < \cdots < x_{m + s}$, where $s = \rho_Y(y_s) - \rho_Y(y_0)$, $\sigma(x_{m + i}) = y_i$, and $\rho_X(x_{m + i}) = \rho_Y(y_i) = \rho_Y(y_0) + i$ for $0 \le i \le s$. Let $r = m + s$. 
	
	Note that $x_{r}$ is a maximal element of $X$. 
	Indeed if $x_{r} \le x$, then $\sigma(x_{r}) = y_s \le \sigma(x)$, and the maximality of $y_s$ implies that $\sigma(x) = y_s$. 
	On the other hand, $\rho_X(x_{r}) \le \rho_X(x) \le \rho_Y(\sigma(x)) = \rho_Y(y_s) = \rho_X(x_{r})$, and hence 
	$\rho_X(x_{r}) = \rho_X(x)$ and $x_{r} = x$. 
	
	Then the concatenation $x_0 < x_1 < \cdots <  x_r$ is our desired maximal chain. 
	

\end{proof}

\begin{remark}\label{rem:zerotozero}
	Let  $\sigma: X \to Y$ be an order-preserving,  surjective function between  posets $X$ and $Y$.  Suppose that $X$ has a unique minimal element $\hat{0}_X$. Then	$\sigma(\hat{0}_X) = \hat{0}_Y$ is the unique minimal element of $Y$. Indeed, for any $y \in Y$, there exists $x \in X$ such that $\sigma(x) = y$, and then $\hat{0}_X \le x$ implies $\sigma(\hat{0}_X) \le \sigma(x) = y$. 
\end{remark}

\begin{example}\label{ex:lowertolower}
	Let $\sigma: X \to Y$ be a strongly surjective function between ranked posets $X$ and $Y$ with rank functions $\rho_X$ and $\rho_Y$ respectively. 
	%
	%
	Then Lemma~\ref{lem:liftchains} implies that $\rank(\sigma) = \rank(X) - \rank(Y) = \rho_Y(\hat{0}_Y) - \rho_X(\hat{0}_X)$ (c.f. \cite{KatzStapledon16}*{Definition~3.17}). 
	This implies that $\rho_X$ determines $\rho_Y$. Indeed, $\rho_X$ is the natural rank function for $X$ shifted by $\rho_X(\hat{0}_X)$, and $\rho_Y$ is the natural rank function for $Y$ shifted by $\rho_X(\hat{0}_X) + \rank(\sigma)$. 
	
	%
\end{example}

We now arrive at our main definition.

\begin{definition}\cite[Definition~3.17]{KatzStapledon16}\label{def:sfs}
	Let  $\sigma: X \to Y$ be an order-preserving, rank-increasing, strongly surjective function between  
	lower Eulerian posets $X$ and $Y$  with rank functions $\rho_X$ and $\rho_Y$ respectively.
	Then $\sigma$ is a \emph{strong formal subdivision} if 
	for all $x \in X$ and $y \in Y$ such that $\sigma(x) \le y$, we have
	\begin{equation}\label{eq:strongsubdivisionequality}
		\sum_{ \substack{x \le x' \in X \\ \sigma(x') = y} } (-1)^{\rho_Y(y) - \rho_X(x')} = 1.
	\end{equation}

\end{definition}

By \cite[Lemma~3.18]{KatzStapledon16}, condition \eqref{eq:strongsubdivisionequality} in Definition~\ref{def:sfs} may be replaced by the condition that  for all $x \in X$ and $y \in Y$ such that $\sigma(x) \le y$, we have
\begin{equation}\label{eq:strongsubdivision}
	\sum_{ \substack{x \le x' \in X \\ \sigma(x') \le y} } (-1)^{\rho_Y(y) - \rho_X(x')} = \begin{cases}
		1 &\textrm{ if } \sigma(x) = y, \\
		0 &\textrm{ otherwise. }
	\end{cases}
\end{equation}



\begin{remark}\label{rem:rankfunctionunique}
	In Definition~\ref{def:sfs}, 
	the rank functions $\rho_X$ and $\rho_Y$ are uniquely determined up to a common shift. 
	Indeed, if the conditions of  Definition~\ref{def:sfs} are satisfied for some choice of $\rho_X$ and $\rho_Y$, then they are satisfied for any common shift $\rho_X[m]$ and $\rho_Y[m]$. Conversely, 
	by Example~\ref{ex:lowertolower}, if $\rho_X$ is the natural rank function shifted by $m$ then 
	$\rho_Y$ is the natural rank function for $Y$ shifted by $m + \rank(\sigma)$. 
\end{remark}


\begin{example}\label{ex:polytope}\cite{KatzStapledon16}*{Lemma~3.25}
	Recall from Section~\ref{ss:polytopes} that given 
	polyhedral subdivisions $\cS'$ and $\cS$ of a polytope $P$ such that $\cS'$ is a refinement of $\cS$, there is an induced function between the corresponding  face posets
	$\sigma: \face(\cS') \to \face(\cS),$
	where $\sigma(F')$ is the smallest element of $\cS$ containing $F'$. 
	Let $\rho_{\face(\cS')}$ and $\rho_{\face(\cS)}$ be the natural rank functions for $\face(\cS')$ and $\face(\cS)$ respectively. 
	Then $\sigma$ is a strong formal subdivision of rank $0$. See also Example~\ref{ex:CWsubdivision}. 
	
\end{example}

\begin{example}\label{ex:proper}
	Recall from Section~\ref{ss:fans} that given a linear map $\phi: V' \to V$ inducing a proper, surjective morphism between fans $\Sigma'$ and $\Sigma$, there is induced function between the corresponding  face posets
	$\sigma: \face(\Sigma') \to \face(\Sigma),$
	where $\sigma(C')$ is the smallest element of $\Sigma$ containing $\phi(C')$. Let $r = \dim \ker(\phi)$. 
	Let $\rho_{\face(\Sigma')}$ be the natural rank function, and let 
	$\rho_{\face(\Sigma)}$ be the natural rank function shifted by $r$, i.e., $\rho_{\face(\Sigma)}(C) = \dim C + r = \dim \phi^{-1}(C)$.  
	Then  $\sigma$ is  a strong formal subdivision of rank  $r$ by Remark~\ref{rem:CWregularisformal} and Lemma~\ref{lem:CWproper} below.

\end{example}

\begin{example}\label{ex:identitypre}\cite{KatzStapledon16}*{Example~3.19}
	Let $B$ be a lower Eulerian poset with rank function $\rho_B$. Let $X = Y = B$ and let $\sigma: X \to Y$ be the identity map $\id_B$. Then $\id_B$ is  a strong formal subdivision of  rank $0$. This example will be discussed in detail in Example~\ref{ex:identity}. 
\end{example}

%

Our main result Theorem~\ref{thm:main} gives an alternative characterization of strong formal subdivisions. In the following two remarks, we mention two other known characterizations. 


\begin{remark}\label{rem:altcharKLStheory}	
	We briefly recall an alternative characterizations of strong formal subdivisions in terms of  Kazhdan-Lusztig-Stanley theory  \cite{KatzStapledon16}*{Proposition~3.29}. 	 This will be used in the proof of Lemma~\ref{lem:composesfs} below.
	Let $\sigma:  X \to Y$ be an order-preserving, rank-increasing  function between lower Eulerian posets $X$ and $Y$ with rank functions $\rho_X$ and $\rho_Y$ respectively.
	Let $R = \Z[t^{\pm 1/2}]$. Then \cite{KatzStapledon16}*{Definition~2.2} assigns to a lower Eulerian poset $B$ a pair of $R$-modules $\mathcal{A}^B \subset R^B$, and assigns a pushforward map 
	$\sigma_*: R^X \to R^Y$ to $\sigma$.
	%
	Then \cite{KatzStapledon16}*{Proposition~3.29, Corollary 4.6} states that if $\sigma$ is strongly surjective, then $\sigma$ is a strong formal subdivision  if and only if $\sigma_*(\mathcal{A}^{B}) = \mathcal{A}^{B'}$.


\end{remark}

\begin{remark}\label{rem:altcharnearEulerian}
	We also recall the following alternative characterization of strong formal subdivisions 
	\cite{DKTPosetSubdivisions}*{Proposition~4.4, Remark~4.5}. 
	Let $\sigma: X \to Y$ be an order-preserving, rank-increasing, surjective function between lower Eulerian posets $X$ and $Y$ with rank functions $\rho_X$  and $\rho_Y$ respectively.  
	Recall that we write $X_{\le y} = \sigma^{-1}([\hat{0}_Y,y])$ and $X_{< y} = \sigma^{-1}([\hat{0}_Y,y))$ for $y \in  Y$. 
	Then $\sigma$ is a strong formal subdivision
	if and only if 
	for any $y$ in $Y$, $X_{\le y}$ is a lower Eulerian poset of rank $\rho_Y(y) - \rho_X(\hat{0}_X)$, 
	and either 
	\begin{enumerate}
		\item $y = \hat{0}_Y$ and $X_{\le y}$ is the boundary of an Eulerian poset of positive rank, or,
		\item $y \neq \hat{0}_Y$ and $X_{\le y}$ is near-Eulerian with boundary $X_{< y}$. 
	\end{enumerate}
\end{remark}

\begin{remark}\label{rem:KEdef}
	Let $X$ and $Y$ be Gorenstein* posets of the same rank. Let $\rho_X$ and $\rho_Y$ be the natural rank functions for $X$ and $Y$ respectively. in \cite{EKDecompositionTheoremCDIndex}*{Definition~2.6}, Ehrenborg and Karu define a \emph{subdivision} to be a surjective, order-preserving function $\sigma : X \to Y$ such that for all $y$ in $Y$ with $y \neq \hat{0}_Y$, 
	$X_{\le y}$ is near-Gorenstein* of rank $\rho_Y(y)$ with boundary $X_{< y}$.    
	Using Remark~\ref{rem:altcharnearEulerian}, one may verify that  
	a subdivision in this sense is a strong formal subdivision of rank $0$. 
\end{remark}

We recall the following basic properties of strong formal subdivisions. 


\begin{remark}\label{rem:restrictsfs}\cite{KatzStapledon16}*{Remark~3.20, Remark~3.21}. 
	If $\sigma: X \to Y$ is a strong formal subdivision 
	and $I$ is a nonempty lower order ideal of $Y$, then $\sigma$ restricts to a strong formal subdivision $\sigma^{-1}(I) \to I$.
	For any $x \in X$, $\sigma$ restricts to a strong formal subdivision $X_{\ge x} \to Y_{\ge \sigma(x)}$ of rank $\rho_Y(\sigma(x)) - \rho_X(x)$. 
\end{remark}

By \cite{KatzStapledon16}*{Remark~3.30},  the composition of strong formal subdivisions 
is a strong formal subdivision. 
We briefly recall the proof and then extend the ideas to prove a partial converse. 

\begin{lemma}\label{lem:composesfs}
	Let  $\sigma : X \to Y$ and $\tau: Y \to Z$ be order-preserving, rank-increasing  
	functions between lower Eulerian posets $X$, $Y$, $Z$ with   rank functions $\rho_X$, $\rho_Y$, $\rho_Z$  respectively.  
	Assume that $\sigma$ is a strong formal subdivision. 
	Then $\tau$ is a strong formal subdivision  if and only if $\tau \circ \sigma$ is a strong formal subdivision. 
	
	
\end{lemma}
\begin{proof}
	We use the characterization of strong formal subdivisions in Remark~\ref{rem:altcharKLStheory}.
	By \cite{KatzStapledon16}*{Lemma~3.28}, 
	$\tau \circ \sigma$ is order-preserving and rank-increasing,  and  $(\tau \circ \sigma)_* = \tau_* \circ \sigma_*$. 

	We can now recall the proof of \cite{KatzStapledon16}*{Remark~3.30}: 
	if
	$\tau$ is a strong formal subdivision, 
	then it is easy to check that
	$\tau \circ \sigma$ is strongly surjective. Also, $(\tau \circ \sigma)_*(\mathcal{A}^X)  = \tau_*(\mathcal{A}^Y) = \mathcal{A}^Z$. Hence $\tau \circ \sigma$ is a strong formal subdivision. 
	
	We need to prove the converse. Assume that 
	$\tau \circ \sigma$ is a strong formal subdivision. 
	Then 
	$\tau_*(\mathcal{A}_Y) = \tau_*(\sigma_*(\mathcal{A}_X)) = (\tau \circ \sigma)_*(\mathcal{A}_X) = \mathcal{A}_Z$. To show that $\tau$ is a strong formal subdivision,  
	 it remains to show that $\tau$ is strongly surjective. 
	
	Consider $y \in Y$ and $z \in Z$ such that $\tau(y) \le z$. 
	Since $\sigma$ is surjective, there exists $x \in X$ such that $\sigma(x) = y$.
	Since $\tau \circ \sigma$ is strongly surjective, there exists $x' \in X$ such that $x \le x'$, $\rho_X(x') = \rho_Z(z)$, and $\tau(\sigma(x')) = z$.	 Set $y' = \sigma(x') \in Y$.  Then $y = \sigma(x) \le y'$ and $\tau(y') = z$. Since $\sigma$ and $\tau$ are rank-increasing, 
	$\rho_X(x') \le \rho_Y(y') \le \rho_Z(z) = \rho_X(x')$, and we deduce that $\rho_Y(y') = \rho_Z(z)$, completing the proof.
\end{proof}

Let $\sigma : X \to Y$ and $\tau: Y \to Z$ be order-preserving, rank-increasing
functions between lower Eulerian posets  $X$, $Y$, $Z$ with   rank functions $\rho_X$, $\rho_Y$, $\rho_Z$  respectively. In contrast to Lemma~\ref{lem:composesfs}, the  following example  shows that  if
$\tau$ and $\tau \circ \sigma$ are strong formal subdivisions, then 
$\sigma$ is not necessarily a strong formal subdivision.

\begin{example}
	Let $X$ be all proper subsets of a $3$-element set $\{ a,b,c\}$, i.e., $X = \partial B_3$.  Let $Y'$ be all subsets of $\{ a, b \}$ and let $Y$ be obtained from $Y'$ by adjoining an element $\hat{z}$ such that $\{  y \in Y : y < \hat{z} \} = \{ \emptyset, a , b \}$. That is, $Y' = B_2$ and $Y = \partial (\tilde{\Sigma} B_2)$. Consider $X$ and $Y$ with the corresponding natural rank functions. Let $Z = B_0$ be the one element poset. Then Example~\ref{ex:B0} implies that the unique maps of posets from $X$ to $Z$ and from $Y$ to $Z$ are strong formal subdivisions. 
	%
	On the other hand, define $\sigma: X \to Y$ by $\sigma(\emptyset) = \emptyset$, $\sigma(a) = a$, $\sigma(b) = b$, $\sigma(ab) = ab$, $\sigma(c) = b$, $\sigma(ac) = \sigma(bc) = \hat{z}$. Then $\sigma$ is order-preserving, rank-increasing and strongly surjective, but $\sigma^{-1}(\hat{z}) = \{ ac, bc \}$ and $\sigma$ is not a strong formal subdivision.  
\end{example}

We will also use the following simple observation. 


%

\begin{lemma}\label{lem:parity}
	Let $\sigma: X \to Y$ be a  strong formal subdivision
	between lower Eulerian posets $X$ and $Y$ with rank functions $\rho_X$ and $\rho_Y$ respectively. 
	Then 
	$$\sum_{x \in X} (-1)^{\rho_X(x)} = \sum_{y \in Y} (-1)^{\rho_Y(y)}.$$ In particular,  $|X| = |Y|$ mod $2$. 
\end{lemma}
\begin{proof}
	The second statement follows by considering the first statement modulo $2$. 
	Using \eqref{eq:strongsubdivisionequality}, we compute
	\begin{align*}
		\sum_{y \in Y} (-1)^{\rho_Y(y)} &= \sum_{y \in Y} (-1)^{\rho_Y(y)} \sum_{\substack{\hat{0}_X \le x \in X \\ \sigma(x) = y}} (-1)^{\rho_Y(y) - \rho_X(x)} =  \sum_{x \in X} (-1)^{\rho_X(x)}. 
	\end{align*}
\end{proof}

\subsection{A bijection for strong formal subdivisions}\label{ss:bijection}

In this section, we state 
Theorem~\ref{thm:mainsimplified}, which gives an alternative way of viewing strong formal subdivisions of lower Eulerian posets. The proof will be given in Section~\ref{sec:proof}.
A functorial version of this result will appear in Theorem~\ref{thm:main} in Section~\ref{ss:functorial}.  
We continue with the notation of the previous sections.

We recall the definition of the non-Hausdorff mapping cylinder. 


\begin{definition}\label{def:nonHausdorff} \cite{BMSimpleHomotopy}*{Definition~3.6}\label{def:nonHausdorffmc}
	Let $\sigma : X \to Y$ be an order-preserving function between posets. 
	The \emph{non-Hausdorff mapping cylinder} $\Cyl(\sigma)$ is the poset with elements given by the disjoint union 
	of $X$ and $Y$, 
	with $z \le z'$ in $\Cyl(\sigma)$ if and only if one of the following conditions hold:
	\begin{enumerate}
		\item $z,z' \in X$ and $z \le z'$ in $X$,
		\item $z,z' \in Y$ and $z \le z'$ in $Y$,
		\item $z \in X$, $z' \in Y$, and $\sigma(z) \le z'$ in $Y$.
	\end{enumerate}
\end{definition}

In Definition~\ref{def:nonHausdorff}, one easily checks that the  order-preserving property of $\sigma$ implies that $\Cyl(\sigma)$ is a well-defined poset. 

We explain one motivation for calling $\Cyl(\sigma)$ the non-Hausdorff mapping cylinder (see \cite{BarmakAlgebraicTopologyFinite}*{Section~2.8}).
Given a continuous map $f: A \to B$ between topological spaces, the mapping cylinder is the topological space given by the disjoint of  $A \times [0,1]$ and  $B$, modulo the relation $(a,1) \sim f(a)$ for all $a \in A$.  A finite poset $B$ can be given a topology with open sets equal to lower order ideals. Then $B$ is a finite $T_0$-space and Alexandroff  showed in \cite{AlexandroffDiskrete} that this gives a correspondence between finite $T_0$-spaces and finite posets  (see, for example, \cite{BarmakAlgebraicTopologyFinite}*{Section~1.1}). In particular, $B_1 = \{ \hat{0}, \hat{1} \}$ is a non-Hausdorff space, with  $\hat{1}$ the unique closed point; from this perspective, $B_1$ is called the Sierpi\'nski space.  Then the poset $\Cyl(\sigma)$ is isomorphic as a topological space to the quotient of the disjoint union of $\Pyr(X) = X \times B_1$  and $Y$, modulo the relation $(x,\hat{1}) \sim \sigma(x)$ for all $x \in X$. 

%
%
%

We will need the following lemma.

\begin{lemma}\label{lem:usestronglysurjective}
	Let $\sigma: X \to Y$ be a strongly surjective function between ranked posets $X$ and $Y$ with rank functions $\rho_X$ and $\rho_Y$ respectively. Consider some $x \in X$ and $y \in Y$.  Then
	$y$ covers $\sigma(x)$ in $\Cyl(\sigma)$ if and only if $\sigma(x) = y$ and $\rho_X(x) = \rho_Y(y)$. 
\end{lemma}
\begin{proof}
	Suppose that $y$ covers $\sigma(x)$ in $\Cyl(\sigma)$. By strong surjectivity, there exists $x \le x'$ in $X$ such that 
	$\sigma(x') = y$ and $\rho_X(x') = \rho_Y(y)$. Then $x \le x' < y$ in $\Cyl(\sigma)$ implies that  $x = x'$, as desired. 
	
	Conversely, suppose that  $\sigma(x) = y$ and $\rho_X(x) = \rho_Y(y)$. Suppose that $x \le z \le y$ in $\Cyl(\sigma)$. If $z \in Y$, then $y = \sigma(x) \le z \le y$ implies that $z = y$. If $z \in X$, then $\rho_Y(y) = \rho_X(x) \le \rho_X(z) \le \rho_Y(\sigma(z)) \le \rho_Y(y)$ implies that $\rho_X(x) = \rho_X(z)$ and hence $x = z$. We conclude that $y$ covers $\sigma(x)$ in $\Cyl(\sigma)$. 
\end{proof}

To state our results we introduce the following sets and functions.  Later, we will consider the sets below as objects of corresponding categories and the functions below as functors between these categories.  See Definition~\ref{def:functorialbijections}. Recall from Definition~\ref{def:joinadmissible} that 
an element $q$ of a poset $\Gamma$ is join-admissible if 
$z \vee q$ exists for all $z \in \Gamma$.	


\begin{definition}\label{def:bijections}
	Let $\SFS$  be the set of strong formal subdivisions $\sigma: X \to Y$ between  lower Eulerian posets $X$ and $Y$  with rank functions $\rho_X$ and $\rho_Y$ respectively. 	
	
	Let $\JoinIdealLW$ be the set of triples $(\Gamma, \rho_\Gamma, q)$,
	where $\Gamma$ is a lower Eulerian poset  with rank function $\rho_\Gamma$,  and 
	$q$ is a join-admissible element of $\Gamma$. 
	Let $\JoinIdealLW^\circ \subset \JoinIdealLW$ be the subset of all triples $(\Gamma, \rho_\Gamma, q)$ such that $q \neq \hat{0}_\Gamma$. 	
	
	Define a function
	$$\CYL :  \SFS \to \JoinIdealLW^\circ,$$ 
		\[
	\CYL(\sigma: X \to Y) = (\Cyl(\sigma), \rho_{\Cyl(\sigma)},\hat{0}_Y), 
	\]
	where $\Cyl(\sigma)$ is the non-Hausdorff mapping cylinder of $\sigma$, and 
	\begin{equation}\label{eq:rhoCyl}
		\rho_{\Cyl(\sigma)}(z) =  \begin{cases}
			\rho_X(z) &\textrm{ if } z \in X, \\
			\rho_Y(z) + 1 &\textrm{ if } z \in Y. 
		\end{cases}
	\end{equation}
	Define a function $$\MAP:   \JoinIdealLW^\circ \to \SFS,$$ 
		\begin{equation*}
		\MAP(\Gamma,\rho_\Gamma,q) : \Gamma \smallsetminus \Gamma_{\ge q} \to \Gamma_{\ge q},
	\end{equation*}
	\[
	x \mapsto x \vee q,
	\]
	where the rank functions on $\Gamma \smallsetminus \Gamma_{\ge q}$ and $\Gamma_{\ge q}$ are  the corresponding restrictions of $\rho_\Gamma$ shifted by $0$ and $-1$ respectively.

\end{definition}

We are now ready to state our main theorem. Below, the statement includes the fact that the functions $\CYL$ and $\MAP$ in Definition~\ref{def:bijections} are well-defined. 

\begin{theorem}\label{thm:mainsimplified}
	
	The functions $\CYL: \SFS \to  \JoinIdealLW^\circ$  and $\MAP:  \JoinIdealLW^\circ \to \SFS$ are mutually inverse bijections. 
\end{theorem}

The proof will be given in Section~\ref{sec:proof}. 
See Theorem~\ref{thm:main} for a generalization to an isomorphism of categories. See Section~\ref{sec:examples} for numerous examples. 
We next make some observations and deduce some corollaries.

\begin{remark}
	Let $\sigma: X \to Y$ be a strong formal subdivision between lower Eulerian posets with rank functions $\rho_X$ and $\rho_Y$ respectively, corresponding under Theorem~\ref{thm:mainsimplified} to a triple $(\Gamma, \rho_\Gamma, q)$ with $\Gamma = \Cyl(\sigma)$, $\rho_\Gamma$ determined by \eqref{eq:rhoCyl}, and $q = \hat{0}_Y$. 
	We claim that $\rho_\Gamma(\hat{0}_\Gamma,q) = \rank(\sigma) + 1$.
		Indeed,	using Example~\ref{ex:lowertolower}, we have
	\[
	\rho_\Gamma(\hat{0}_\Gamma,q) =  \rho_Y(\hat{0}_Y) - \rho_X(\hat{0}_X) + 1 = \dim X -  \dim Y + 1 = \rank(\sigma) + 1. 
	\]

\end{remark}

A poset is \emph{graded} if every maximal chain of $B$ has the same length. 

\begin{corollary}\label{cor:graded}
		Let $\sigma: X \to Y$ be a  strong formal subdivision between  lower Eulerian posets with rank functions $\rho_X$ and $\rho_Y$ respectively,
 corresponding under Theorem~\ref{thm:mainsimplified} to a triple $(\Gamma, \rho_\Gamma, q)$ with $\Gamma = \Cyl(\sigma)$, $\rho_\Gamma$ determined by \eqref{eq:rhoCyl}, and $q = \hat{0}_Y$. Then $\rank(\Gamma) = \rank(X) + 1$. Moreover, $X$ is graded if and only if $Y$ is graded if and only if 
	$\Gamma$ is graded. 
	
	
	
\end{corollary}
\begin{proof}
	Using Lemma~\ref{lem:usestronglysurjective}, the maximal chains in $\Gamma = \Cyl(\sigma)$ have the following description: consider any $x \in X$ and $y \in Y$ such that $\sigma(x) = y$ and $\rho_X(x) = \rho_Y(y)$. Note that every $y \in Y$ appears in this way. Then concatenate a maximal chain $\hat{0}_X = x_0 < \cdots < x_r = x$ in $[\hat{0}_X,x] \subset X$ with a maximal chain $y = y_0 < \cdots < y_s$ in $Y_{\ge y}$ to obtain a maximal chain in $\Gamma$.  Here $r = \rho_X(x) - \rho_X(\hat{0}_X)$. By Example~\ref{ex:lowertolower},  
	the maximal chain has length
	\begin{equation}\label{eq:r+s+1}
		r + s + 1 =  \rho_X(x) - \rho_X(\hat{0}_X) + \rho_Y(y_s) - \rho_Y(y) + 1 = \rank(X) + 1 + \rho_Y(y_0,y_s)  - \rank(Y).
	\end{equation}
	Choosing $y_s$ such that $\rho_Y(y_0,y_s) = \rank(Y)$ shows that $\rank(\Gamma) = \rank(X) + 1$. 
	
	In the argument above, the maximal chain in $\Gamma$ has length $\rank(\Gamma)$ if and only if $\rho_Y(y_0,y_s) = \rank(Y)$. If $Y$ is graded, then this condition holds for all maximal chains and we deduce that $\Gamma$ is graded. If $\Gamma$ is graded, then choosing $y_0 = \hat{0}_Y$ above, we deduce that $Y$ is graded. 
	
	Assume that $\Gamma$ is graded. Consider a maximal chain $\hat{0}_X = x_0 < \cdots < x_r = x$ in $X$, and let $\sigma(x) = y$. If $y \le y'$ in $Y$, then the maximality of $x$ and the strong surjectivity condition implies that $y' = y$ and $\rho_X(x) = \rho_Y(y)$. Then  $\hat{0}_X = x_0 < \cdots < x_r = x < y$ is a maximal chain in $\Gamma$ of length $\rank(\Gamma) = \rank(X) + 1$. We conclude that $X$ is graded.
	
	Finally, assume that $X$ is graded. By Lemma~\ref{lem:liftchains} and Example~\ref{ex:lowertolower}, 
	every maximal chain in $Y$ has length $s = \rank(X) - \rho_Y(\hat{0}_Y) + \rho_X(\hat{0}_X) = \rank(X) - \rank(\sigma) = \rank(Y)$. We conclude that $Y$ is graded.


\end{proof}

%
%
%

Not every lower Eulerian poset admits a non-minimal join-admissible element. We present one obstruction below.

\begin{corollary}\label{cor:oddequaleven}
	Let $\Gamma$ be a lower Eulerian poset that contains a join-admissible element $q \neq \hat{0}_\Gamma$. 
	Then $\Gamma$ has as many elements of even rank as odd rank. In particular, $|\Gamma|$ is even. 
\end{corollary}
\begin{proof}
	Fix a choice of rank function $\rho_\Gamma$ for $\Gamma$. 
	By Theorem~\ref{thm:mainsimplified}, there exists a strong formal subdivision $\sigma: X \to Y$ 
	between lower Eulerian posets $X$ and $Y$ with rank functions $\rho_X$ and $\rho_Y$ respectively
	such that $\Gamma = \Cyl(\sigma)$ and  
	\begin{equation*}
		\rho_{\Gamma}(z) = \begin{cases}
			\rho_X(z) &\textrm{ if } z \in X, \\
			\rho_Y(z) + 1 &\textrm{ if } z \in Y. 
		\end{cases}
	\end{equation*}
	By Lemma~\ref{lem:parity}, 
	\[
	\sum_{z \in \Gamma} (-1)^{\rho_\Gamma(z)} = \sum_{x \in X} (-1)^{\rho_X(x)} - \sum_{y \in Y} (-1)^{\rho_Y(y)} = 0. 
	\]
\end{proof}

The converse to Corollary~\ref{cor:oddequaleven} is false. 
Observe that any near-Eulerian poset has as many even elements as odd elements. Also, observe that if $q$ is a join-admissible element of a lower Eulerian poset $\Gamma$ then $q$ is a common lower bound for all maximal elements of $\Gamma$.  
Let $P$ be a square and let $e$ be an edge, and consider the near-Eulerian poset $B = \face(P) \smallsetminus \{ e, P \}$. 
Then the maximal elements of $B$ correspond to the three edges of $P$ distinct from $e$. The common intersection of these edges is the empty face, and hence  $B$ contains no join-admissible elements except $\hat{0}_B$.



We will also need the following lemma, which is a reformulation of 
\cite{EKDecompositionTheoremCDIndex}*{Lemma~2.10}. Let $B$ be a ranked poset. 
We say that a poset $B$ is \emph{locally Gorenstein*} if every interval of positive rank is Gorenstein*.
A locally Gorenstein* poset is \emph{lower Gorenstein*} if it contains a unique minimal element. For example, Gorenstein* posets are precisely lower Gorenstein* posets that contain a unique maximal element. 

\begin{lemma}\label{lem:GorensteinstarsubdivEK}
		Let $\sigma: X \to Y$ be a strong formal subdivision between lower Eulerian posets with rank functions $\rho_X$ and $\rho_Y$ respectively,
 corresponding under Theorem~\ref{thm:mainsimplified} to a triple $(\Gamma, \rho_\Gamma, q)$ with $\Gamma = \Cyl(\sigma)$, $\rho_\Gamma$ determined by \eqref{eq:rhoCyl}, and $q = \hat{0}_Y$. If $\Gamma$ is a lower Gorenstein* poset then
	for any $y$ in $Y$, either 
	\begin{enumerate}
		\item $y = \hat{0}_Y$ and $X_{\le y}$ is the boundary of
		a Gorenstein* poset, or,
		\item\label{i:nearGorenstein} $y \neq \hat{0}_Y$ and $X_{\le y}$ is near-Gorenstein* with boundary $X_{< y}$.
	\end{enumerate}
\end{lemma}
\begin{proof}
	If $y = \hat{0}_Y$, then $\overline{X_{\le y}} = [\hat{0}_\Gamma,y]$ is an interval of positive rank in $\Gamma$, and hence is Gorenstein*. 
	If $y \neq \hat{0}_Y$, then  $X_{\le y}$ is near-Gorenstein* by \cite{EKDecompositionTheoremCDIndex}*{Lemma~2.10} (applied to the Gorenstein* poset $[\hat{0}_\Gamma,y]$ with join-admissible element $q$; note that Ehrenborg and Karu assume the poset is a 
	lattice, but the proof only uses that $q$ is join-admissible). 
The boundary of $X_{\le y}$ is  $X_{< y}$ by Remark~\ref{rem:altcharnearEulerian}. 
\end{proof}

\subsection{A functorial bijection for strong formal subdivisions}\label{ss:functorial}

In this section, we state  Theorem~\ref{thm:main}, which is a functorial version of Theorem~\ref{thm:mainsimplified}. 
The proof will be given in Section~\ref{sec:proof}.
We continue with the notation of the previous sections. 

We first discuss some reasons why functoriality is important. When working with posets we often identify isomorphic posets. Here an isomorphism of posets is an order-preserving function with an order-preserving inverse.
We similarly need the notion of 
of an isomorphism between  two strong formal subdivisions, and we want the bijection of Theorem~\ref{thm:mainsimplified} to respect isomorphisms. 
Moreover, in the companion paper \cite{StapledonKLSTheory}, we need the notion of  an automorphism of a strong formal subdivision when we consider group actions. 
These issues are resolved by generalizing Theorem~\ref{thm:mainsimplified}  to an isomorphism of categories.
A corollary of the proof is a method to construct a  strong formal subdivision from a commutative diagram of strong formal subdivisions (see \eqref{eq:CYLcat}), as well as an interesting involution on the strong formal subdivisions appearing in this way (see Remark~\ref{rem:involution}). 



We briefly recall some basic category theory. Recall that a category $C$ consists of objects and morphisms, denoted $\Obj(C)$ and $\Mor(C)$ respectively. 
We require that
morphisms compose in the sense that if $\alpha: a \to b$ and $\beta: b \to c$ are morphisms then there is a well-defined morphism $\beta \circ \alpha: a \to c$. We require that composition is associative, i.e., if $\gamma: c \to d$ is a morphism then $\gamma \circ (\beta \circ \alpha) = (\gamma \circ \beta) \circ \alpha$. Also, for any object $a$ in $C$, there is an identity morphism $\id_a : a \to a$, and we require that $\alpha \circ \id_a = \id_b \circ \alpha = \alpha$ for all choices of $\alpha: a \to b$. A functor $F: C \to D$ between categories takes each object $a$ of $C$ to an object $F(a)$ of $D$, and takes each morphism $\alpha: a \to b$ of $C$ to a morphism $F(\alpha): F(a) \to F(b)$.  We require that $F(\id_a) = \id_{F(a)}$ for all objects $a$ in $C$, and $F(\beta \circ \alpha) = F(\beta) \circ F(\alpha)$ for all morphisms $\alpha: a \to b$ and $\beta: b \to c$ in $C$. 
If $C$ is a category then a subcategory $C'$ of $C$ is a \emph{full subcategory} if all morphisms in $C$ between objects in $C'$ are morphisms in $C'$. Given any collection $A$ of objects in $C$, there is a unique full subcategory of $C$ whose objects are precisely the objects in $A$. 

Given a category $C$, the \emph{arrow category} $\Arr(C)$ is the category with objects given by morphisms in $C$, i.e., $\Obj(\Arr(C)) = \Mor(C)$, and a morphism from $\alpha: a \to b$ to $\beta: c \to d$ given by a commutative square
\[
\begin{tikzcd} a \arrow[r, "\alpha"] \arrow[d, "\phi_1"] & b \arrow[d, "\phi_2"] \\ c \arrow[r, "\beta"] &  d. \end{tikzcd}
\]
Composition of morphisms is given by vertically stacking commutative diagrams on top of each other to obtain a new commutative diagram. The identity morphism on $\alpha: a \to b$ is obtained by setting $\alpha = \beta$, $\phi_1 = \id_a$, and $\phi_2 = \id_b$ above. 

We now define categories $\SFScat$ and $\Joincat$ whose objects are the sets $\SFS$  and $\JoinIdealLW$ respectively from Definition~\ref{def:bijections}. 
	Recall that $\SFS$  is the set of strong formal subdivisions $\sigma: X \to Y$ between  lower Eulerian posets $X$ and $Y$  with rank functions $\rho_X$ and $\rho_Y$ respectively. 	
Recall that  $\JoinIdealLW$ is the set of triples $(\Gamma, \rho_\Gamma, q)$,
where $\Gamma$ is a lower Eulerian poset  with rank function $\rho_\Gamma$,  and 
$q$ is a join-admissible element of $\Gamma$. 
Recall that $\JoinIdealLW^\circ \subset \JoinIdealLW$ is the subset of all triples $(\Gamma, \rho_\Gamma, q)$ such that $q \neq \hat{0}_\Gamma$.

\begin{definition}\label{def:categorify}
	Let $\EulPos$ be the category with elements given by  pairs $(B,\rho_B)$, where $B$ is a lower Eulerian poset and $\rho_B$ is a rank function for $B$, and with
	morphisms given by strong formal subdivisions. 
	We define $\SFScat$  to be the corresponding arrow category $\Arr(\EulPos)$. 
	
	Let $\Joincat$ be the category with objects $\JoinIdealLW$. We describe the morphisms. 
	A morphism  between objects $(\Gamma, \rho_\Gamma, q)$ and $(\Gamma', \rho_{\Gamma'}, q')$ in  $\Joincat$ is 
	a strong formal subdivision 
	$\phi: \Gamma \to \Gamma'$ such that $\phi(q) = q'$ 
	and
	$\phi(z \vee q) = \phi(z) \vee q'$ for all $z \in \Gamma \smallsetminus \Gamma_{\ge q}$. 
	Let $\Joincat^\circ$ be the full subcategory of $\Joincat$ with objects 
	$\JoinIdealLW^\circ$.

	
	
	%
\end{definition}

%

Note that in the definition of the morphisms in $\Joincat$ in Definition~\ref{def:categorify} above, the condition 	$\phi(z \vee q) = \phi(z) \vee q'$ holds for all $z \in \Gamma$. Indeed, if $z \in \Gamma_{\ge q}$, then $\phi(z) \in \Gamma_{\ge q'}'$ since $\phi$ is order-preserving, and hence   	$\phi(z \vee q) = \phi(z) = \phi(z) \vee q'$. 

It follows from Example~\ref{ex:identitypre} and Lemma~\ref{lem:composesfs}  that $\EulPos$ 
is a well-defined category.
To check that $\Joincat$ is a well-defined category, we need to additionally verify that composition of morphisms is well-defined. 
Consider objects $\alpha = (\Gamma, \rho_\Gamma,q)$, $\alpha' = (\Gamma', \rho_{\Gamma'},q')$, and $\alpha'' = (\Gamma'', \rho_{\Gamma''},q'')$ in  $\Joincat$. Consider strong formal subdivisions  $\phi: \Gamma \to \Gamma'$ and  $\phi': \Gamma' \to \Gamma''$ defining morphisms from $\alpha$ to $\alpha'$ and from $\alpha'$ to $\alpha''$ respectively.  We need to check that the strong formal subdivision $\phi' \circ \phi$ defines a morphism from $\alpha$ to $\alpha''$.  We compute
$(\phi' \circ \phi)(q) = \phi'(q') = q''$, and for all $z \in \Gamma$, 
$(\phi' \circ \phi)(z \vee q) = \phi'(\phi(z) \vee q') = \phi'(\phi(z)) \vee q'',$ as desired.

\begin{remark}
	Observe that there is a functor from $\Joincat$ to $\EulPos$ that `forgets' $q$, i.e., sends $(\Gamma,\rho_\Gamma,q)$ to $(\Gamma,\rho_\Gamma)$, and sends a strong formal subdivision $\phi: \Gamma \to \Gamma'$ to itself. 
	Also, there is a functor from  $\EulPos$  to $\Joincat$, that sends $(\Gamma,\rho_\Gamma)$ to $(\Gamma,\rho_\Gamma,\hat{0}_\Gamma)$, and 
	sends a strong formal subdivision $\phi: \Gamma \to \Gamma'$ to itself. This identifies $\EulPos$ with the full subcategory of $\Joincat$ with objects of the form $(\Gamma,\rho_\Gamma,\hat{0}_\Gamma)$. With this identification, we claim that $\Joincat$ is the union of $\EulPos$ and 
	$\Joincat^{\circ}$. 
	
	To establish the claim, consider a morphism  between objects $(\Gamma, \rho_\Gamma, q)$ and $(\Gamma', \rho_{\Gamma'}, q')$ in  $\Joincat$  
	given by a strong formal subdivision 
	$\phi: \Gamma \to \Gamma'$. We need to show that $q = \hat{0}_\Gamma$ if and only if $q' = \hat{0}_{\Gamma'}$. 	If $q = \hat{0}_\Gamma$, then $q' = \phi(q) = \hat{0}_{\Gamma'}$ by Remark~\ref{rem:zerotozero}. Conversely, suppose that $q' = \hat{0}_{\Gamma'}$. 
	By Remark~\ref{rem:altcharnearEulerian}, $\Gamma_{\le q'}$ is the boundary of an Eulerian poset of positive rank. In particular, 
	$|\Gamma_{\le q'}|$ is odd. By Corollary~\ref{cor:oddequaleven},  the only join-admissible element in $\Gamma_{\le q'}$ is $\hat{0}_\Gamma$. 
	If $z \in \Gamma_{\le q'}$, then $\phi(z \vee q) = \phi(z) \vee q' = q'$ and hence $z \vee q \in \Gamma_{\le q'}$. We deduce that $q$ is a join-admissible element of $\Gamma_{\le q'}$, and hence $q = \hat{0}_\Gamma$. 		
\end{remark}


We may view the functions $\CYL$ and $\MAP$ in Definition~\ref{def:bijections} as defining functions between the objects of 
$\SFScat$ and $\Joincat^\circ$. We next extend these to functors.

\begin{definition}\label{def:functorialbijections}
	Define functors 
	\[
	\CYLcat: \SFScat \to \Joincat^\circ,
	\]
	\[
	\MAPcat: \Joincat^\circ \to \SFScat,
	\]
	where the image of an object under $\CYLcat$ or $\MAPcat$ is 
	given by applying $\CYL$ or $\MAP$ respectively (see Definition~\ref{def:bijections}).
	The image of a morphism under $\CYLcat$ is defined by
	\begin{equation}\label{eq:CYLcat}
				\CYLcat \left(\begin{tikzcd} X \arrow[r, "\sigma"] \arrow[d, "\phi_1"] & Y \arrow[d, "\phi_2"] \\ X' \arrow[r, "\sigma'"] &  Y' \end{tikzcd}\right) = (\phi: \Cyl(\sigma) \to \Cyl(\sigma')),
	\end{equation}
		where $\phi$ is the strong formal subdivision defined by
		\[
		\phi(z) = \begin{cases}
				\phi_1(z) &\textrm{ if } z \in X, \\
				\phi_2(z) &\textrm{ if } z \in Y. 
			\end{cases}
		\]
	Consider a morphism in $\Joincat^\circ$ between objects $(\Gamma,\rho_\Gamma,q)$ and $(\Gamma',\rho_{\Gamma'},q')$ given by a strong formal subdivision 
	$\phi: \Gamma \to \Gamma'$.
	Suppose that $\MAP(\Gamma,\rho_\Gamma,q)$ is a strong formal subdivision $\sigma: X \to Y$, and  $\MAP(\Gamma',\rho_{\Gamma'},q')$ is a strong formal subdivision $\sigma': X' \to Y'$.
	Then 
	the image of the morphism under $\MAPcat$ is defined by
	\[
	\MAPcat(\phi: \Gamma \to \Gamma') = 
	\begin{tikzcd} X \arrow[r, "\sigma"] \arrow[d, "\phi_1"] & Y \arrow[d, "\phi_2"] \\ X' \arrow[r, "\sigma'"] &  Y' \end{tikzcd},
	\]
	where $\phi_1(x) = \phi(x) \in X' \subset \Gamma'$ for all $x \in X$, and
	$\phi_2(y) = \phi(y) \in Y' \subset \Gamma'$ for all $y \in Y$. 

\end{definition}

\begin{remark}\label{rem:cylinderphi}
	With the notation of \eqref{eq:CYLcat} in Definition~\ref{def:functorialbijections} above, the non-Hausdorff mapping cylinder $\Cyl(\phi)$ is the poset with elements given by the disjoint union 
	of $X$, $X'$, $Y$, and $Y'$, 
	and $z \le z'$ in $\Cyl(\phi)$ if and only if one of the following conditions hold:
	\begin{enumerate}
		\item $z,z' \in X$ and $z \le z'$ in $X$,
		\item $z,z' \in X'$ and $z \le z'$ in $X'$,
		\item $z,z' \in Y$ and $z \le z'$ in $Y$,
		\item $z,z' \in Y'$ and $z \le z'$ in $Y'$,
		\item $z \in X$, $z' \in Y$, and $\sigma(z) \le z'$ in $Y$,
		\item $z \in X$, $z' \in X'$, and $\phi_1(z) \le z'$ in $X'$,
		\item $z \in X'$, $z' \in Y'$, and $\sigma'(z) \le z'$ in $Y'$,
		\item $z \in Y$, $z' \in Y'$, and $\phi_2(z) \le z'$ in $Y'$,
		\item $z \in X$, $z' \in Y'$, and $(\phi_2 \circ \sigma)(z) = (\sigma' \circ \phi_1)(z) \le z'$ in $Y'$.
	\end{enumerate}
\end{remark}

We are now ready to state our theorem. 
Below, the statement includes the fact that the functors $\CYLcat$ and $\MAPcat$ in Definition~\ref{def:bijections} are well-defined.

\begin{theorem}\label{thm:main}
	The functors $\CYLcat: \SFScat \to  \Joincat^\circ$  and $\MAPcat:  \Joincat^\circ \to \SFScat$ are mutually inverse isomorphisms.
\end{theorem}

The proof will be given in Section~\ref{sec:proof}. 


\begin{remark}\label{rem:involution}
	Note that $\Mor(\SFScat)$ has a natural involution given by reflecting commutative diagrams through the diagonal from top left to bottom right:
	\[
	\begin{tikzcd} X \arrow[r, "\sigma"] \arrow[d, "\phi_1"] & Y \arrow[d, "\phi_2"] \\ X' \arrow[r, "\sigma'"] &  Y' \end{tikzcd} \mapsto 
		\begin{tikzcd} X \arrow[r, "\phi_1"] \arrow[d, "\sigma"] & X' \arrow[d, "\sigma'"] \\ Y \arrow[r, "\phi_2"] &  Y'. \end{tikzcd}
\]
Applying $\CYLcat$ to both sides, Theorem~\ref{thm:main} gives a corresponding involution $\iota$ on $\Mor(\Joincat^\circ)$. It follows from Remark~\ref{rem:cylinderphi} that $\iota$ preserves the non-Hausdorff mapping cylinder in the sense that for any $\phi: \Gamma \to \Gamma'$ in $\Mor(\Joincat^\circ)$, $\Cyl(\phi) = \Cyl(\iota(\phi))$. 
\end{remark}

\section{Examples of main results}\label{sec:examples}

The goal of this section is to study Theorem~\ref{thm:mainsimplified} 
and Theorem~\ref{thm:main} in several examples that will also be used throughout the remainder of the paper. 
In Section~\ref{ss:facelatticepolytope}, we consider the case when $\Gamma$ is the face lattice of a polytope. In Section~\ref{ss:basicexamples}, we consider the pyramid and bipyramid constructions, 
and strong formal subdivisions into either $B_0$ or $B_1$. In Section~\ref{ss:constructingsfs}, we consider some general methods for constructing strong formal subdivisions. 
Below we often identify isomorphic posets for convenience; this is justified by the fact that Theorem~\ref{thm:main} gives a functorial
version of Theorem~\ref{thm:mainsimplified}. 
We continue with the notation of Section~\ref{sec:subdivisions}.

We first recall the bijection of Theorem~\ref{thm:mainsimplified} and set notation to be freely used throughout this section. 
Let $\sigma: X \to Y$ be a strong formal subdivision  between  lower Eulerian posets $X$ and $Y$  with rank functions $\rho_X$ and $\rho_Y$ respectively. 
Consider a  triple $(\Gamma, \rho_\Gamma, q)$,
where $\Gamma$ is a  lower Eulerian poset  with rank function $\rho_\Gamma$,  and $q \neq \hat{0}_\Gamma$ is a join-admissible element of $\Gamma$.
Assume that $\sigma$ corresponds to $(\Gamma, \rho_\Gamma, q)$ under the bijection of Theorem~\ref{thm:mainsimplified}.  That is,  $\Gamma = \Cyl(\sigma)$ is the non-Hausdorff mapping cylinder of $\sigma$, $q = \hat{0}_Y$, $X = \Gamma \smallsetminus \Gamma_{\ge q}$,  $Y = \Gamma_{\ge q}$, and $\sigma(x) = x \vee q$ for all $x \in X$. Also, $\rho_\Gamma$ determines and is determined by $(\rho_X,\rho_Y)$ via \eqref{eq:rhoCyl}. 

\subsection{The face lattice of a polytope}\label{ss:facelatticepolytope}
Let $P$ be a full-dimensional polytope in a real vector space $V$. 
Let $F$ be a nonempty face of $P$, and let $r = \dim F$. 
Consider the triple $(\Gamma,\rho_\Gamma,q)$ in $\JoinIdealLW^\circ$, where 
$\Gamma = \face(P)$ is the face lattice of $P$, $\rho_\Gamma$ is a choice of rank function for $\Gamma$, and $q = F \in \Gamma$.
After possible translation, we may assume that the origin lies in the relative interior of $F$. After possibly shifting the rank function, we may assume that $\rho_\Gamma$ is the natural rank function. 
Recall that for a face $G$ of $P$, $C(G)$ denotes the cone spanned by $G$. 
Consider the fan $\Sigma = \{ C(G) : G \in \face(P), 0 \notin G \}$. Then the support $|\Sigma|$ of $\Sigma$ is the cone $C(P)$ with faces $\{ C(G) : G \in \face(P), 0 \in G \}$. In particular, $C(F)$ is the minimal face of $C(P)$ and is an $r$-dimensional linear subspace of $V$.  Consider the linear map $\phi: V \to V/L$. 
Let $C = \phi(C(P))$. Then $C$ is a full-dimensional pointed cone in $V/L$, and $\face(C)$ is isomorphic to face lattice of the quotient polytope $P/F$. 
Observe that  $\phi$ induces a proper, surjective morphism of fans from $\Sigma$ to $C$.
We claim that the morphism is projective. Indeed, 
define $p: |\Sigma| \to \R$ to be the unique piecewise linear function with value $1$ on the nonzero vertices of $P$. The fact that $P$ is convex implies that $p$ is a strictly convex function, as desired. 
Let $\sigma: \face(\Sigma) \to \face(C)$ be the corresponding strong formal subdivision of rank $r$, where $\face(\Sigma)$ and $\face(C)$ are equipped with the natural rank function and the natural rank function shifted by $r$ respectively (see Example~\ref{ex:proper}). Then $\sigma = \MAP(\Gamma,\rho_\Gamma,q)$ under Theorem~\ref{thm:mainsimplified}.

Conversely, let $\phi: V \to W$ be a surjective linear map of real vector spaces. Let $L = \ker(\phi)$ and let $r = \dim L =  \dim V - \dim W$. Let $C$ be a full-dimensional pointed cone in $W$, and let $\Sigma$ be a fan in $V$ such that $\phi$ induces a projective morphism of fans from $\Sigma$ to $C$. Let $\sigma: \face(\Sigma) \to \face(C)$ be the corresponding strong formal subdivision  of rank $r$, where, after a possible common shift of rank functions, we may assume that $\face(\Sigma)$ and $\face(C)$
are equipped with the natural rank function and the natural rank function shifted by $r$ respectively (see Example~\ref{ex:proper}). 
Observe that $|\Sigma| = \phi^{-1}(C)$ is a full-dimensional cone in $V$ with minimal face $L$. Let $p: |\Sigma| \to \R$ be a 
piecewise linear function $p$ on $\Sigma$ that is strictly convex with respect to $\Sigma$. After possibly adding a constant, we may assume that $p$ takes strictly positive values on nonzero elements in $|\Sigma|$.  
Let $P = p^{-1}([0,1])$. The fact that $p$ is strictly convex implies that $P$ is convex, and the faces of $P$ not containing the origin are in inclusion-preserving bijection with the cones of $\Sigma$. On the other hand, the faces of $P$ containing the origin are in inclusion-preserving bijection with the faces of $|\Sigma| = \phi^{-1}(C)$, which are in bijection with the faces of $C$. Let $F$ be the nonempty face of $P$ containing the origin in its relative interior. 
Let $\Gamma = \face(P)$, let $\rho_\Gamma$ be the natural rank function, and let $q = F \in \Gamma$. Then $\CYL(\sigma) = (\Gamma,\rho_\Gamma,q)$ under Theorem~\ref{thm:mainsimplified}.

We explain the special case when $r = 0$ in a slightly different, but equivalent, way. 
Let $P$ be a  full-dimensonal polytope in a vector space $V$, and let
$v$ be a vertex of $P$. 
Consider the triple $(\Gamma,\rho_\Gamma,q)$ in $\JoinIdealLW^\circ$, where 
$\Gamma = \face(P)$ is the face lattice of $P$, $\rho_\Gamma$ is the natural rank function, and $q = v \in \Gamma$. After possible translation, we may assume that $v$ is the origin in $V$. Then  $\Sigma = \{ C(G) : G \in \face(P), 0 \notin G \}$ is a projective refinement of the pointed cone $C(P)$. Let $H$ be an affine hyperplace in $V$ that intersects all the rays of $\Sigma$ at a unique nonzero point. Let $Q = C(P) \cap H$ and let $\cS = \{ C(G) \cap H : G \in \face(P), 0 \notin G \}$ be the corresponding regular polyhedral subdivision of $Q$. Consider the induced strong formal subdivision $\sigma: \face(\cS) \to \face(Q)$, where $\face(\cS)$ and $\face(Q)$ are equipped with the natural rank functions (see Example~\ref{ex:polytope}). 
Then $\sigma = \MAP(\Gamma,\rho_\Gamma,q)$ under Theorem~\ref{thm:mainsimplified}.

Conversely, let $\cS$ be a regular polyhedral subdivision of a polytope $Q$ in a vector space $V$. 
Consider the induced strong formal subdivision $\sigma: \face(\cS) \to \face(Q)$, where $\face(\cS)$ and $\face(Q)$ are equipped with the natural rank functions (see Example~\ref{ex:polytope}). 
With the notation of Section~\ref{ss:polytopes}, let 
$A \subset Q$ be a finite subset containing the vertices of $Q$, 
and let $\omega: A \to \R$ be a function such that 
$\cS = \cS(\omega)$, i.e., $\cS$ is the projection of the bounded faces of 
the upper convex hull $\UH(\omega) \subset V \oplus \R$.  
Consider the pointed cone $C(\UH(\omega) \times \{ 1 \})$. By Example~\ref{ex:polyhedron}, the faces of $C(\UH(\omega) \times \{ 1 \})$ are the faces of the form $C(Q' \times \{ 1\})$, where $Q'$ is a face of $\UH(\omega)$, together with the ray $R$ through $(0_V,1,0) \in V \oplus \R \oplus \R$, where  $0_V$ denotes the origin in $V$. 
Let $P$ be the intersection of $C(\UH(\omega) \times \{ 1 \})$ with an affine hyperplane $H$ that intersects all rays of  $C(\UH(\omega) \times \{ 1 \})$ at a unique nonzero point. Let $v$ be the intersection of $R$ with $H$. Then $v$ is vertex of $P$. Let $\Gamma = \face(P)$, let $\rho_\Gamma$ be the natural rank function, and let $q = v \in \Gamma$. Then $\CYL(\sigma) = (\Gamma,\rho_\Gamma,q)$ under Theorem~\ref{thm:mainsimplified}.


\begin{example}\label{ex:polygon}
	Let $P$ be a polygon with $s + 3$ vertices. 
	Let $\Gamma = \face(P)$, let $\rho_\Gamma$ be the natural rank function, and let $q$ be a vertex of $P$. 
	Let $Q = [0,1]$. Then $\face(Q) = B_2$. 
	Let $\cS$ be the regular polyhedral subdivision of $Q$ with $s$ interior vertices.  	 Let $\sigma: \face(\cS) \to \face(Q)$ be the corresponding strong formal subdivision of rank $0$, where $\face(\cS)$ and $\face(Q)$ are equipped with the natural rank functions. 	Then  $\CYL(\sigma) = (\Gamma,\rho_\Gamma,q)$ and $\MAP(\Gamma,\rho_\Gamma,q) = \sigma$ under Theorem~\ref{thm:mainsimplified}. 
\end{example}

\begin{example}\label{ex:polytopeB0}
	Let $P$ be a full-dimensional polytope in a vector space $V$ of dimension $r$.   Let $\Gamma = \face(P)$, let $\rho_\Gamma$ be the natural rank function, and let $q = P \in \Gamma$.  After possible translation, assume that the origin lies in the relative interior of $P$. Consider the projective fan $\Sigma = \{ C(F) : F \textrm{ proper face of } P \}$, and the unique projective morphism induced by the linear map $\phi: V \to \{ 0\}$. Let $\sigma: \face(\Sigma) \to \face(\{ 0\})$ be the corresponding strong formal subdivision of rank $r$, where   $\face(\Sigma) = \partial \face(P)$ and $\face(\{ 0\}) = B_0$
	are equipped with the natural rank function and the natural rank function shifted by $r$ respectively. 	Then  $\CYL(\sigma) = (\Gamma,\rho_\Gamma,q)$ and $\MAP(\Gamma,\rho_\Gamma,q) = \sigma$ under Theorem~\ref{thm:mainsimplified}. See Example~\ref{ex:B0} for a generalization. 
\end{example}

\begin{example}\label{ex:polytopeB1}
	Let $P$ be a full-dimensional polytope in a vector space $V$ of dimension $r + 1$.  Let $F$ be a facet of $P$.  Let $\Gamma = \face(P)$, let $\rho_\Gamma$ be the natural rank function, and let $q = F \in \Gamma$.  After possible translation, assume that the origin lies in the relative interior of $F$. Consider the fan $\Sigma = \{ C(G) : G \textrm{ proper face of } P, G \neq F \}$, and the unique projective morphism to the pointed cone $C = \R_{\ge 0} \subset \R$ induced by the linear map $\phi: V \to V/L \cong \R$, where $L$ is the linear subspace spanned by $F$. Let $\sigma: \face(\Sigma) \to \face(C)$ be the corresponding strong formal subdivision of rank $r$, where   $\face(\Sigma)$ and $\face(C) = B_1$
	are equipped with the natural rank function and the natural rank function shifted by $r$ respectively. Here $\face(\Sigma) = \face(P) \smallsetminus \{ F,P\}$ is a near-Eulerian poset with semisuspension $\face(P)$ and boundary $\partial \face(F)$. 	Then  $\CYL(\sigma) = (\Gamma,\rho_\Gamma,q)$ and $\MAP(\Gamma,\rho_\Gamma,q) = \sigma$ under Theorem~\ref{thm:mainsimplified}. See Example~\ref{ex:B1} for a generalization. 
\end{example}

\begin{example}\label{ex:polytopeidentity}
	Let $P$ be a polytope. Let $\Gamma = \face(\Pyr(P))$, let $\rho_\Gamma$ be the natural rank function, and let $q$ be the apex of $\Pyr(P)$. Let $\id_{\face(P)}: \face(P) \to \face(P)$ be the identity function, which is induced by the trivial subdivision of $P$, and assume that $\face(P)$ is equipped with the natural rank function.
	Then  $\CYL(\id_{\face(P)}) = (\Gamma,\rho_\Gamma,q)$ and $\MAP(\Gamma,\rho_\Gamma,q) = \id_{\face(P)}$ under Theorem~\ref{thm:mainsimplified}. See Example~\ref{ex:identity} for a generalization. 
\end{example} 

\begin{example}\label{ex:polytopebipyramid}
	Let $P$ be a polytope.   Let $\Gamma = \face(\Bipyr(P))$, let $\rho_\Gamma$ be the natural rank function, and let $q$ be an apex of $\Bipyr(P)$. Let $\cS$ be the pulling refinement of $P$ at a point in its relative interior, and let $\sigma: \face(\cS) \to \face(P)$ be the corresponding strong formal subdivision, where $\face(\cS)$ and $\face(P)$ are equipped with the natural rank functions. 	Then  $\CYL(\sigma) = (\Gamma,\rho_\Gamma,q)$ and $\MAP(\Gamma,\rho_\Gamma,q) = \sigma$ under Theorem~\ref{thm:mainsimplified}. See Example~\ref{ex:bipyramid} for a generalization. 
\end{example}


\subsection{Basic examples}\label{ss:basicexamples}

	We now consider a few examples involving general posets. 
	The first two examples describe the effect of the pyramid and bipyramid constructions respectively. 
	The first example also describes our results for the identity strong formal subdivision 	(c.f. Example~\ref{ex:identitypre}). 
	We then describe strong formal subdivisions to $B_0$ and $B_1$ respectively. 
	
		\begin{example}[Pyramid]\label{ex:identity}
		Let $B$ be a lower Eulerian poset with rank function $\rho_B$. 	Recall from Section~\ref{ss:posets}  that $\Pyr(B) = B \times B_1$, where $B_1 = \{ \hat{0}, \hat{1}\}$. Consider the natural rank function on $B_1$, and recall that there is an induced rank function $\rho_{B \times B_1}$ for $\Pyr(B)$. 
		
		Let $\Gamma = \Pyr(B)$, $\rho_\Gamma = \rho_{B \times B_1}$, and $q = (\hat{0}_B, \hat{1})$. 	Then $q \neq \hat{0}_\Gamma$ is join-admissible. Let $X = Y = B$ with $\rho_X = \rho_Y = \rho_B$.  If we identify $X$  with $B \times \{ \hat{0} \} \subset \Gamma$, and $Y = B$ with $B \times \{ \hat{1} \} \subset \Gamma$, then the corresponding strong formal subdivision $\sigma: X \to Y$ is the identity function $\id_B$ on $B$. 
		
		See Example~\ref{ex:productidentity} for a generalization. See also Example~\ref{ex:polytopeidentity} for a special case.

	\end{example}

		\begin{example}[Bipyramid]\label{ex:bipyramid}
		Let $B$ be an Eulerian poset of positive rank with the natural rank function $\rho_B$. 
		Recall from Section~\ref{ss:posets}  that 	the bipyramid $\Bipyr(B)$ of $B$ is the unique Eulerian poset 
		such that $\partial \Bipyr(B) = \partial B \times \partial B_2$.
		Let $\rho_{B_2}$ be the natural rank function for $B_2$, and 
		recall that there is an induced rank function $\rho_{\Bipyr(B)}$ for $\Bipyr(B)$.  Fix a choice of maximal element $\hat{z} \in \partial B_2$. 
		
%
%
		
%
		
		Let $\Gamma = \Bipyr(B)$, $\rho_\Gamma = \rho_{\Bipyr(B)}$, and $q = (\hat{0}_B,\hat{z})$. Then $q \neq \hat{0}_\Gamma$ is join-admissible.
		Let $X = \Pyr(\partial B)$ with rank function $\rho_X = \rho_{\partial B \times B_1}$, where $\rho_{\partial B}$ is the restriction of $\rho_B$, and $\rho_{B_1}$ is the natural rank function on $B_1$.  Let $Y = B$ with rank function $\rho_Y = \rho_B$. If we identify $B_1$ with $B_2 \smallsetminus \{ \hat{z}, \hat{1}_{B_2} \}$, 
		identify
		$X$ with $\partial B \times B_1 \subset \Gamma$, and identify $Y$ with $\Gamma \smallsetminus \partial B \times B_1$ so that $\partial Y$ is identified with $\partial B \times \{ \hat{z} \}$, then  the corresponding strong formal subdivision $\sigma: X \to Y$ of rank $0$ is
		\[
		\sigma:  \Pyr(\partial B) \to B,
		\]
		\[
		\sigma(z,z') = \begin{cases}
			z &\textrm{if } z' = \hat{0} \in B_1, \\
			\hat{1}_B &\textrm{if } z' = \hat{1} \in B_1. \\
		\end{cases}
		\]
		
		See Example~\ref{ex:dualdiamondproductidentity} for a generalization. See also Example~\ref{ex:polytopebipyramid} for a special case.

	\end{example}

	We 
	next give a complete description of strong formal subdivisions into the one element poset $B_0$. Recall that if $B$ is an Eulerian poset of positive rank, then its boundary is  $\partial B = B \smallsetminus \{ \hat{1}_B \}$.  
	
	\begin{example}[Strong formal subdivisions to $B_0$]\label{ex:B0}
		Let $B$ be an Eulerian poset of positive rank with rank function $\rho_B$. 	
		
		Let $\Gamma = B$, $\rho_\Gamma = \rho_B$, and  $q = \hat{1}_\Gamma$. 
		Then $q \neq \hat{0}_\Gamma$ is join-admissible. Let  $X = \partial B$ with $\rho_X$  the restriction of $\rho_B$. Let $Y = B_0 = \{ \hat{0} \}$ with rank function $\rho_Y(\hat{0}) = \rho_B(\hat{1}_B) - 1$. If we identify $Y$ with $\{ q \}$, then the corresponding strong formal subdivision $\sigma: X \to Y$ is the unique function $\sigma: \partial B \to B_0$. 
		
		Conversely, every strong formal subdivision from a lower Eulerian poset to $B_0$ has this form. Indeed, if $\sigma: X \to B_0$ is a strong formal subdivision, then $\Gamma = \Cyl(\sigma) = X \cup B_0$ is Eulerian with $\partial \Gamma = X$, and the rank functions are determined by Remark~\ref{rem:rankfunctionunique}.  
		

		
	\end{example}
	
	Recall that if $B$ is a poset, then $\overline{B}$ is the poset obtained from $B$ by adjoining a unique maximal element.  
	
	\begin{corollary}\label{cor:boundaryEulerian}
		Suppose that $\sigma: X \to Y$ is a strong formal subdivision between lower Eulerian posets  with rank functions $\rho_X$ and $\rho_Y$ respectively. Then $\overline{X}$ is Eulerian if and only if $\overline{Y}$ is Eulerian.
	\end{corollary}
	\begin{proof}
		By Remark~\ref{rem:rankfunctionunique},
		after possibly applying a common shift to $\rho_X$ and $\rho_Y$, we may assume that  $\rho_X(\hat{0}_X) = 0$ and $\rho_Y(\hat{0}_Y) = \rank(\sigma) = \rank(X) - \rank(Y)$. Consider $B_0 = \{ \hat{0} \}$ with rank function $\rho_{B_0}(\hat{0}) = \rank(X)$. 
		Let $\tau: Y \to B_0$ be the unique function  from $Y$ to $B_0$. Then $\tau \circ \sigma: X \to B_0$ is the unique function  from $X$ to $B_0$. 
		By Example~\ref{ex:B0}, $\tau \circ \sigma$ is a strong formal subdivision if and only if $\overline{X}$ is Eulerian, and $\tau$ is  a strong formal subdivision if and only if $\overline{Y}$ is Eulerian. 
		By Lemma~\ref{lem:composesfs}, $\tau \circ \sigma$ is a strong formal subdivision if and only if $\tau$ is a strong formal subdivision.
		
	\end{proof}
	
	We next give a complete description of strong formal subdivisions into the two element poset $B_1$. 
	
	\begin{example}[Strong formal subdivisions to $B_1$]\label{ex:B1}
		Let $B$ be a near-Eulerian poset with rank function $\rho_B$. Recall from Section~\ref{ss:posets} the corresponding semisuspension $\tilde{\Sigma} B = B \cup   \{ \hat{z}, \hat{1}_{\Gamma}\}$ and boundary $\partial B = \{ z \in \tilde{\Sigma} B : z < \hat{z} \}$ with induced rank functions  $\rho_{\tilde{\Sigma} B}$ and $\rho_{\partial B}$ respectively. Here 
		$\rho_{\tilde{\Sigma} B}(\hat{1}_{\Gamma}) = \rho_{\tilde{\Sigma} B}(\hat{z}) + 1 = \rho_{B}(\hat{0}_B) + \rank(B) + 1$.

		Let $\Gamma = \tilde{\Sigma} B$, $\rho_\Gamma = \rho_{\tilde{\Sigma} B}$, and $q = \hat{z}$. 	Then $q \neq \hat{0}_\Gamma$ is join-admissible. Let $X = B$ with $\rho_X = \rho_B$, and let $Y = B_1 = \{ \hat{0}, \hat{1} \}$ with rank function $\rho_{B_1}(\hat{1}) = \rho_{B_1}(\hat{0}) + 1 = \rho_{B}(\hat{0}_B) + \rank(B)$. If we identify $Y$ with $\Gamma_{\ge \hat{z}}$, then  the corresponding strong formal subdivision $\sigma: X \to Y$ is 
		\[
		\sigma: B \to B_1,
		\]
		\[
		\sigma(z) = \begin{cases}
			\hat{0} &\textrm{if } z \in \partial B, \\
			\hat{1} &\textrm{otherwise. }
		\end{cases}
		\]
		
		Conversely, every strong formal subdivision from a lower Eulerian poset to $B_1$ has this form. Indeed, if $\sigma: X \to B_1$ is a strong formal subdivision, then $\Gamma = \Cyl(\sigma) = X \cup B_1$ is Eulerian and $X$ is near-Eulerian with  $\tilde{\Sigma} X = \Gamma$ and $\partial X = \sigma^{-1}(\hat{0})$. The rank functions are determined by Remark~\ref{rem:rankfunctionunique}.

	\end{example}

	In the corollary below,  the implication that $Y$ being near-Eulerian implies that $X$ is near-Eulerian was proved in the case when $\rank(\sigma) = 0$ in \cite{DKTPosetSubdivisions}*{Proposition~4.9}. 
	
	\begin{corollary}\label{cor:nearEulerian}
		Suppose that $\sigma: X \to Y$ is a strong formal subdivision between lower Eulerian posets   with rank functions $\rho_X$ and $\rho_Y$ respectively. Then $X$ is near-Eulerian if and only if $Y$ is near-Eulerian. Moreover, when these conditions hold, $\sigma^{-1}(\partial Y) = \partial X$. 
	\end{corollary}
	\begin{proof}
		By Remark~\ref{rem:rankfunctionunique},
		after possibly applying a common shift to $\rho_X$ and $\rho_Y$, we may assume that  $\rho_X(\hat{0}_X) = 0$ and $\rho_Y(\hat{0}_Y) = \rank(\sigma) = \rank(X) - \rank(Y)$. 
		Consider $B_1 = \{ \hat{0}, \hat{1} \}$ with rank function $\rho_{B_1}$ defined by 
		$\rho_{B_1}(\hat{1}) = \rho_{B_1}(\hat{0}) + 1 = \rank(X)$.
		
		By Example~\ref{ex:B1}, $X$ is near-Eulerian if and only there exists a strong formal subdivision from $X$ to $B_1$. Similarly, $Y$ is near-Eulerian if and only there exists a strong formal subdivision from $Y$ to $B_1$. If  $\tau: Y \to B_1$ is a strong formal subdivision, then Lemma~\ref{lem:composesfs} implies that $\tau \circ \sigma : X \to B_1$ is a strong formal subdivision.

		Conversely, suppose that there exists a strong formal subdivision $\nu: X \to B_1$. 
		Let $\partial Y$ be  the lower order ideal of $Y$ generated by all elements $z \in Y$ 
		such that $Y_{\ge z} = B_1$.   Assume that we can show that $\sigma^{-1}(\partial Y) = \partial X$.  By Lemma~\ref{lem:parity}, since $|X|$ is even, it follows that $|Y|$ is even, and, in particular, $\rank(Y) \ge \rank(B_1) = 1$. Define $\tau: Y \to B_1$ by $\tau(y) = \hat{0}$ if $y \in \partial Y$, and $\tau(y) = \hat{1}$ otherwise. Then $\tau$ is order-preserving and rank-increasing and $\tau \circ \sigma = \nu$. 
		By Lemma~\ref{lem:composesfs}, $\tau$ is a strong formal subdivision, as desired.
		
		It remains to show that $\sigma^{-1}(\partial Y) = \partial X$.
		We first show that $\partial X \subset \sigma^{-1}(\partial Y)$. 
		Let $x \in \partial X$. We want to show that $\sigma(x) \in \partial Y$. 
		Since $\sigma$ is order preserving and $\partial Y$ is a lower order ideal, we reduce to the case when $X_{\ge x} = B_1$.  By Remark~\ref{rem:restrictsfs},
		$\sigma$ restricts to a strong formal subdivision 
		$X_{\ge x} = B_1 \to Y_{\ge \sigma(x)}$. By Lemma~\ref{lem:parity}, $|Y_{\ge \sigma(x)}|$ is even, and since $\sigma$ is surjective, $|Y_{\ge \sigma(x)}| \le 2$. Hence $Y_{\ge \sigma(x)} = B_1$, and $\sigma(x) \in \partial Y$.  
		
		Conversely, we  show that $\sigma^{-1}(\partial Y)  \subset  \partial X$. Consider an element $x \in X$ such that $\sigma(x) \in \partial Y$.
		Then there exists $y \in Y$ such that $\sigma(x) \le y$ and $Y_{\ge y} = B_1$. By strong surjectivity of $\sigma$, there exists $x' \in X$ with $x \le x'$, $\rho_X(x') = \rank(X) - 1$, and $\sigma(x') = y$. 
		By Remark~\ref{rem:restrictsfs}, $\sigma$ restricts to a strong formal subdivision 
		$\tilde{\sigma}: X_{\ge x'}  \to Y_{\ge y} = B_1$ of rank $0$. In particular, $\rank(X_{\ge x'}) = 1$. 
		If $\tilde{y}$ is the unique maximal element of $Y_{\ge y}$, then 
		$\tilde{\sigma}^{-1}(\tilde{y})$ consists of the maximal elements of 
		$X_{\ge x'}$. By \eqref{eq:strongsubdivisionequality}, $|\tilde{\sigma}^{-1}(\tilde{y})| = 1$, and hence $X_{\ge x'} = B_1$ and $x' \in \partial X$. 
		Since $\partial X$ is a lower order ideal and $x \le x'$, we deduce that $x \in \partial X$.	  
	\end{proof}
	
	We finish with some more examples of Theorem~\ref{thm:main}. 
			Suppose that $\sigma: X \to Y$ is a strong formal subdivision between lower Eulerian posets  with rank functions $\rho_X$ and $\rho_Y$ respectively, corresponding to a triple $(\Gamma, \rho_\Gamma,q)$ under Theorem~\ref{thm:mainsimplified}. 
	
	Consider the case when  $\overline{X}$ and $\overline{Y}$ are Eulerian.  Let $\phi_1$ and $\phi_2$ be the unique strong formal subdivisions from $X$ and $Y$ respectively to $B_0$ (see Example~\ref{ex:B0}). Then, using Example~\ref{ex:B0},
			\[
	\CYLcat \left(\begin{tikzcd} X \arrow[r, "\sigma"] \arrow[d, "\phi_1"] & Y \arrow[d, "\phi_2"] \\ B_0 \arrow[r, "\id_{B_0}"] &  B_0 \end{tikzcd}\right) = (\phi: \Gamma \to B_1), \: \: \:
		\CYLcat \left(\begin{tikzcd} X \arrow[r, "\phi_1"] \arrow[d, "\sigma"] & B_0 \arrow[d, "\id_{B_0}"] \\ Y \arrow[r, "\phi_2"] &  B_0 \end{tikzcd}\right) = (\tilde{\sigma}: \overline{X} \to \overline{Y}),
	\]
		where $\tilde{\sigma}$ may be viewed as an extension of $\sigma$. 
		By  Remark~\ref{rem:involution} and Example~\ref{ex:B1}, 
			$\Gamma$ is near-Eulerian with boundary $X$ and semisuspension 
			$\tilde{\Sigma} \Gamma = \Cyl(\phi) = \Cyl(\tilde{\sigma})$. 
			
	Consider the case when $X$ and $Y$ are near-Eulerian.  Let 
	$\phi_1$ and $\phi_2$ be the unique strong formal subdivisions from $X$ and $Y$ respectively to $B_1$ (see Example~\ref{ex:B1}). Then, using Example~\ref{ex:B1},
	\[
	\CYLcat \left(\begin{tikzcd} X \arrow[r, "\sigma"] \arrow[d, "\phi_1"] & Y \arrow[d, "\phi_2"] \\ B_1 \arrow[r, "\id_{B_1}"] &  B_1 \end{tikzcd}\right) = (\phi: \Gamma \to B_2), \: \: \: 
	\CYLcat \left(\begin{tikzcd} X \arrow[r, "\phi_1"] \arrow[d, "\sigma"] & B_1 \arrow[d, "\id_{B_1}"] \\ Y \arrow[r, "\phi_2"] &  B_1 \end{tikzcd}\right) = (\tilde{\Sigma} \sigma: \tilde{\Sigma} X \to \tilde{\Sigma} Y),
	\]
	where $\tilde{\Sigma} \sigma$ may be viewed as an extension of $\sigma$. 
	By Remark~\ref{rem:involution}, $\Cyl(\phi) = \Cyl(\tilde{\Sigma} \sigma)$. 

	\subsection{Constructing strong formal subdivisions}\label{ss:constructingsfs}
	In this section, we consider some general methods for constructing strong formal subdivisions. 
	We first consider the effect of some constructions on posets  applied to elements of $\JoinIdealLW$. These can produce strong formal subdivisions using the identification of $\JoinIdealLW^\circ$ with $\SFS$ in Theorem~\ref{thm:mainsimplified}.  
	 We use the notation from Section~\ref{ss:posets} below. In particular, recall the direct product, dual diamond product, diamond product, and star product constructions. 
	
	\begin{lemma}\label{lem:constructions}
		Consider $(\Gamma,\rho_{\Gamma},q),(\Gamma',\rho_{\Gamma'},q') \in \JoinIdealLW$. Then 
		\begin{enumerate}
			\item\label{i:directproduct} (Direct product)
			If $(q,q') \neq (\hat{0}_\Gamma,\hat{0}_{\Gamma'})$, then 
			$(\Gamma \times \Gamma',\rho_{\Gamma \times \Gamma'},(q,q')) \in \JoinIdealLW^\circ$. 
			
			\item\label{i:dualdiamondproduct} (Dual diamond product)
			Suppose that $\Gamma,\Gamma'$ are Eulerian of positive rank. 
			If $(q,q') \neq (\hat{0}_\Gamma,\hat{0}_{\Gamma'})$, $q \neq \hat{1}_\Gamma$, and $q' \neq \hat{1}_{\Gamma'}$, then 
			$(\Gamma \diamond^\ast \Gamma', \rho_{\Gamma \diamond^\ast \Gamma'}, (q,q'))  \in \JoinIdealLW^\circ$. 
			
			\item\label{i:diamondproduct} (Diamond product) 
			Suppose that $\Gamma,\Gamma'$ are Eulerian of positive rank.
			If $q \neq \hat{0}_\Gamma$ and $q' \neq \hat{0}_{\Gamma'}$, then 
			$(\Gamma \diamond \Gamma', \rho_{\Gamma \diamond \Gamma'}, (q,q'))  \in \JoinIdealLW^\circ$. 
			
			\item\label{i:starproduct} (Star product)
			Suppose that $\Gamma$ is Eulerian of positive rank. If $q' \neq \hat{0}_{\Gamma'}$, then  	$(\Gamma \ast \Gamma', \rho_{\Gamma \ast \Gamma'}, q')  \in \JoinIdealLW^\circ$.

		\end{enumerate}
	\end{lemma}
	\begin{proof}
	 First consider \eqref{i:directproduct}.  For any $(z,z') \in \Gamma \times \Gamma'$, we have $(z,z') \vee (q,q') = (z \vee q, z' \vee q')$, and the result holds. Then \eqref{i:dualdiamondproduct} follows similarly. Explicitly, we have $\hat{1}_{\Gamma \diamond^\ast \Gamma'} \vee (q,q') = \hat{1}_{\Gamma \diamond^\ast \Gamma'}$ and  for any $(z,z') \in \partial \Gamma \times \partial \Gamma'$, we have 
	 $(z,z') \vee (q,q') = (z \vee q, z' \vee q') \in \partial \Gamma \times \partial \Gamma'$ if $z \vee q \in \partial \Gamma$ and $z' \vee q' \in \partial \Gamma'$, and   $(z,z') \vee (q,q') = \hat{1}_{\Gamma \diamond^\ast \Gamma'}$ otherwise. Also, \eqref{i:diamondproduct} follows similarly. Explicitly, we have $\hat{0}_{\Gamma \diamond \Gamma'} \vee (q,q') = (q,q')$, and for any $(z,z') \in (\Gamma \smallsetminus \{ \hat{0}_{\Gamma}\}) \times (\Gamma' \smallsetminus \{ \hat{0}_{\Gamma'}\})$, we have 
	 $(z,z') \vee (q,q') = (z \vee q, z' \vee q')
	 \in (\Gamma \smallsetminus \{ \hat{0}_{\Gamma}\}) \times (\Gamma' \smallsetminus \{ \hat{0}_{\Gamma'}\})$. Finally, consider \eqref{i:starproduct}. If $z \in \partial \Gamma$, then $z \vee q' = q'$. If $z \in \Gamma' \smallsetminus \{ \hat{0}_{\Gamma'} \}$, then $z \vee q' \in \Gamma' \smallsetminus \{ \hat{0}_{\Gamma'} \}$. 
	 
	\end{proof}
	
	\begin{example}\label{ex:productidentity}
		Suppose that $B$ is a lower Eulerian poset with rank function $\rho_B$.  Suppose that $(\Gamma,\rho_\Gamma,q) \in \JoinIdealLW^\circ$ corresponds to a strong formal subdivision $\sigma: X \to Y$. Consider the element $(B,\rho_B,\hat{0}_B) \in \JoinIdealLW$. 
		By \eqref{i:directproduct} in Lemma~\ref{lem:constructions}, $(B \times \Gamma, \rho_{B \times \Gamma},(\hat{0}_B,q)) \in \JoinIdealLW^\circ$ with corresponding strong formal subdivision
		\[
		\id_B \times \sigma: B \times X \to B \times Y,
		\]
		\[
		(z,x) \mapsto (z,\sigma(x)). 
		\]
		For example, setting $\Gamma = B_1$, $q = \hat{1}_{B_1}$, and $\sigma = \id_{B_0}$, recovers
		 Example~\ref{ex:identity}.  
	\end{example}

	\begin{example}\label{ex:productsubdivisions}
		Consider strong formal subdivisions $\sigma: X \to Y$  and $\sigma': X' \to Y'$.
		 between lower Eulerian posets $X,X',Y,Y'$ with rank functions $\rho_{X},\rho_{X'},\rho_{Y},\rho_{Y'}$ respectively.  
		 Recall that $X \times X'$ and $Y \times Y'$ are lower Eulerian posets with rank functions $\rho_{X \times X'}$ and $\rho_{Y \times Y'}$ respectively.  
		We claim that $\sigma \times \sigma': X \times X' \to Y \times Y'$ is a strong formal subdivision. Indeed, this follows from Lemma~\ref{lem:composesfs} and Example~\ref{ex:productidentity}, since $\sigma \times \sigma' = (\sigma \times \id_{Y'}) \circ (\id_X \times \sigma') = (\id_Y \times \sigma') \circ (\sigma \times \id_{X'})$. 
		
		Moreover, suppose that $(\Gamma,\rho_{\Gamma},q)$ and $(\Gamma',\rho_{\Gamma'},q')$ are the elements of $\JoinIdealLW^\circ$ corresponding to $\sigma$ and $\sigma'$ respectively. Then $(\Cyl(\sigma \times \sigma'),\rho_{\Cyl(\sigma \times \sigma')},(q,q'))$ corresponds to $\sigma \times \sigma'$. 
		 Observe that $\Cyl(\sigma \times \sigma') \subsetneq \Gamma \times \Gamma'$, but $\rho_{\Cyl(\sigma \times \sigma')}$ is not the restriction of $\rho_{\Gamma \times \Gamma'}$. More precisely,
		 \[
		 \rho_{\Cyl(\sigma \times \sigma')}(z,z') = \begin{cases}
		 	\rho_{X \times X'}(z,z') = \rho_{\Gamma \times \Gamma'}(z,z') &\textrm{ if } (z,z') \in X \times X', \\
		 			 	\rho_{Y \times Y'}[1](z,z') = \rho_{\Gamma \times \Gamma'}[-1](z,z') &\textrm{ if } (z,z') \in Y \times Y'.
		 \end{cases}
		 \]
		We also have
			\[
		\CYLcat \left(\begin{tikzcd} X \times X' \arrow[r, "\sigma \times \id_{X'}"] \arrow[d, "\id_X \times \sigma'"] & Y \times X'  \arrow[d, "\id_Y \times \sigma'"] \\ X \times Y' \arrow[r, "\sigma \times \id_{Y'}"] &  Y \times Y' \end{tikzcd}\right) = (\id_\Gamma \times \sigma': \Gamma \times X' \to \Gamma \times Y'),
		\]		
					\[
		\CYLcat \left(\begin{tikzcd} X \times X' \arrow[r, "\id_X \times \sigma'"] \arrow[d, "\sigma \times \id_{X'}"] & X \times Y'  \arrow[d, "\sigma \times \id_{Y'}"] \\ Y \times X' \arrow[r, "\id_Y \times \sigma'"] &  Y \times Y' \end{tikzcd}\right) = (\sigma \times \id_{\Gamma'} : X \times \Gamma'  \to Y \times \Gamma'),
		\]		 
		where $\Cyl(\id_\Gamma \times \sigma') = \Cyl(\sigma \times \id_{\Gamma'}) = \Gamma \times \Gamma'$ (c.f. Remark~\ref{rem:involution}). 

	\end{example}

		\begin{example}\label{ex:dualdiamondproductidentity}
		Suppose that $B$ is an Eulerian poset of positive rank with rank function $\rho_B$.  Suppose that $(\Gamma,\rho_\Gamma,q) \in \JoinIdealLW^\circ$ corresponds to a strong formal subdivision $\sigma: X \to Y$.
		Assume that $\Gamma$ is an Eulerian poset of rank at least $2$, and 
		$q \notin \{ \hat{0}_\Gamma, \hat{1}_\Gamma \}$. 
		Consider the element $(B,\rho_B,\hat{0}_B) \in \JoinIdealLW$. 
		By \eqref{i:dualdiamondproduct} in Lemma~\ref{lem:constructions}, 
		$(B \diamond^\ast \Gamma, \rho_{B \diamond^\ast \Gamma}, (\hat{0}_\Gamma,q))  \in \JoinIdealLW^\circ$ with corresponding strong formal subdivision
		\[
		   \partial B \times X  \to B \diamond^\ast Y,
		\]
		\[
		(z,x) \mapsto 
		\begin{cases}
			(z,\sigma(x)) &\textrm{ if } \sigma(x) \neq \hat{1}_Y, \\
			\hat{1}_{B \diamond^\ast Y} &\textrm{ if } \sigma(x) = \hat{1}_Y. 
		\end{cases}
		\]
		For example, setting $\Gamma = B_2$, $X = Y = B_1$, and $\sigma = \id_{B_1}$, recovers Example~\ref{ex:bipyramid}.

	\end{example}
	
		\begin{example}\label{ex:starproductidentity}
		Suppose that $B$ is an Eulerian poset of positive rank with rank function $\rho_B$.  Suppose that $(\Gamma,\rho_\Gamma,q) \in \JoinIdealLW^\circ$ corresponds to a strong formal subdivision $\sigma: X \to Y$.
		Consider the element $(B,\rho_B,\hat{0}_B) \in \JoinIdealLW$. 
		By \eqref{i:starproduct} in Lemma~\ref{lem:constructions}, 
		$(B \ast \Gamma, \rho_{B \ast \Gamma}, q)  \in \JoinIdealLW^\circ$ with corresponding strong formal subdivision
		\[
		B \ast X \to Y,
		\]
		\[
		z \mapsto \begin{cases}
			q &\textrm{ if } z \in \partial B, \\
			\sigma(z) &\textrm{ if } z \in X \smallsetminus \{ \hat{0}_X \}.
		\end{cases}
		\]
		For example, setting $\Gamma = B_1$, $X = Y = B_0$, and $\sigma = \id_{B_0}$, recovers the construction in Example~\ref{ex:B0}. 
		\end{example}

\section{Proof of main results}\label{sec:proof}

In this section, we prove our main results Theorem~\ref{thm:mainsimplified} and Theorem~\ref{thm:main}. 
We will use the notation of Section~\ref{sec:subdivisions}.

Recall from Definition~\ref{def:bijections} that 
	 $\SFS$  is the set of strong formal subdivisions $\sigma: X \to Y$ between  lower Eulerian posets $X$ and $Y$  with rank functions $\rho_X$ and $\rho_Y$ respectively. 
Recall that  $\JoinIdealLW$ is the set of triples $(\Gamma, \rho_\Gamma, q)$,
where $\Gamma$ is a lower Eulerian poset  with rank function $\rho_\Gamma$,  and 
$q$ is a join-admissible element of $\Gamma$. 
Recall that $\JoinIdealLW^\circ \subset \JoinIdealLW$ is the subset of all triples $(\Gamma, \rho_\Gamma, q)$ such that $q \neq \hat{0}_\Gamma$. 	
Recall that
$\CYL :  \SFS \to \JoinIdealLW^\circ$ 
is defined 
\[
\CYL(\sigma: X \to Y) = (\Cyl(\sigma), \rho_{\Cyl(\sigma)},\hat{0}_Y), 
\]
where $\Cyl(\sigma)$ is the non-Hausdorff mapping cylinder of $\sigma$, and 
\begin{equation*}
	\rho_{\Cyl(\sigma)}(z) =  \begin{cases}
		\rho_X(z) &\textrm{ if } z \in X, \\
		\rho_Y(z) + 1 &\textrm{ if } z \in Y. 
	\end{cases}
\end{equation*}
Recall that $\MAP:   \JoinIdealLW^\circ \to \SFS$ is defined by 
\begin{equation*}
	\MAP(\Gamma,\rho_\Gamma,q) : \Gamma \smallsetminus \Gamma_{\ge q} \to \Gamma_{\ge q},
\end{equation*}
\[
x \mapsto x \vee q,
\]
where the rank functions on $\Gamma \smallsetminus \Gamma_{\ge q}$ and $\Gamma_{\ge q}$ are  the corresponding restrictions of $\rho_\Gamma$ shifted by $0$ and $-1$ respectively.


We need to show that $\CYL$ and $\MAP$ are well-defined, and that they are mutually inverse bijections. In what follows, we will freely use the notation above.

\begin{lemma}
	The function $\CYL :  \SFS \to \JoinIdealLW^\circ$ is well-defined. 
\end{lemma}
\begin{proof}
	Consider a strong formal subdivision $\sigma: X \to Y$ between  lower Eulerian posets $X$ and $Y$  with rank functions $\rho_X$ and $\rho_Y$ respectively. 
	We need to show that $(\Cyl(\sigma), \rho_{\Cyl(\sigma)},\hat{0}_Y)$ is a well-defined element in $\JoinIdealLW^\circ$. 
	
	We first show that $\rho_{\Cyl(\sigma)}$ is a rank function for $\Cyl(\sigma)$. Suppose that $z'$ covers $z$ in $\Cyl(\sigma)$. 
	We need to show that $\rho_{\Cyl(\sigma)}(z') = \rho_{\Cyl(\sigma)}(z) + 1$. If $z,z' \in X$ or $z,z' \in Y$, then this follows since $\rho_X$ and $\rho_Y$ are rank functions. Hence we may assume that $z \in X$ and $z' \in Y$. By Lemma~\ref{lem:usestronglysurjective}, 
	$\rho_X(z) = \rho_Y(z')$, and we compute
	$\rho_{\Cyl(\sigma)}(z') = \rho_Y(z') + 1 = \rho_X(z) + 1 =  \rho_{\Cyl(\sigma)}(z) + 1$, as desired.

	We next show that $\Cyl(\sigma)$ is locally Eulerian. Given elements $z < z'$ in $\Cyl(\sigma)$, we need to show that
	$\sum_{z \le z'' \le z'} (-1)^{\rho_{\Cyl(\sigma)}(z'')} = 0$. 
	If $z,z' \in X$ or $z,z' \in Y$, then this holds since $X$ and $Y$ are locally Eulerian. 
	Hence we may assume that $z \in X$, $z' \in Y$, and $\sigma(z) \le z'$ in $Y$. Using the fact that $Y$ is locally Eulerian, we compute
	\begin{align*}
		\sum_{z \le z'' \le z'} (-1)^{\rho_{\Cyl(\sigma)}(z'')} &= 
		\sum_{\substack{z \le x \in X \\ \sigma(x) \le z'}} (-1)^{\rho_X(x)} +  \sum_{ \sigma(z) \le y \le z' \in Y } (-1)^{\rho_Y(y) + 1} 
		\\
		&= 
		\sum_{\substack{z \le x \in X \\ \sigma(x) \le z'}} (-1)^{\rho_X(x)}   -  \begin{cases}
			(-1)^{\rho_Y(z')} &\textrm{ if } \sigma(z) = z', \\
			0 &\textrm{ otherwise. }
		\end{cases}
		\\
	\end{align*}
	The right hand side of the above expression is zero by \eqref{eq:strongsubdivision}, as desired. 	
	
	Since $\sigma$ is surjective, the minimal elements of $\Cyl(\sigma)$ are the minimal elements of $X$. Hence $\hat{0}_X$ is the unique minimal element of $\Cyl(\sigma)$, and $\Cyl(\sigma)$ is lower Eulerian. 
	
	Finally, using Definition~\ref{def:nonHausdorffmc}, we have $\hat{0}_Y \neq \hat{0}_\Gamma$ and $x \vee \hat{0}_Y = \sigma(x) \in \Cyl(\sigma)$ for all $x \in X$. Also,  $y \vee \hat{0}_Y = y$ for all $y \in Y$. We conclude that $\hat{0}_Y$ is a join-admissible element of $\Cyl(\sigma)$.
	

\end{proof}

Our next goal is to
show that $\MAP$ is well-defined. 
We introduce some temporary notation. Fix an element $(\Gamma,\rho_\Gamma,q)$ in $\JoinIdealLW^\circ$. Let $\hat{X} = \Gamma \smallsetminus \Gamma_{\ge q}$ and $\hat{Y} = \Gamma_{\ge q}$. 
Let $\rho_{\hat{X}} = \rho_\Gamma|_X$ be the restriction of $\rho_\Gamma$ to $\hat{X}$. 
Let  $\rho_{\hat{Y}} = \rho_\Gamma|_{\hat{Y}}[-1]$ be the shift by $-1$ of the restriction of $\rho_\Gamma$ to $\hat{Y}$, i.e. $\rho_{\hat{Y}}(y) =  \rho_\Gamma(y) - 1$ for all $y \in \hat{Y}$.
Let $\hat{\sigma} = \MAP(\Gamma,\rho_\Gamma,I)$, i.e., 
$\hat{\sigma}: \hat{X} \to \hat{Y}$ is given by $\hat{\sigma}(x) = x \vee q$ for all $x \in \hat{X}$. 


We need to show that $\hat{\sigma}$ is an element of $\SFS$. 
Since $q \neq \hat{0}_\Gamma$, it follows that $\hat{X}$ and $\hat{Y}$ are nonempty. 	Moreover, since 
$\hat{X}$ and $\hat{Y}$  are lower order and upper order ideals  of $\Gamma$ respectively, it follows that $\hat{X}$ and $\hat{Y}$ are lower Eulerian with rank functions $\rho_{\hat{X}}$ and $\rho_{\hat{Y}}$ respectively, and unique minimal elements $\hat{0}_{\hat{X}} = \hat{0}_\Gamma$ and $\hat{0}_{\hat{Y}} = q$ respectively.

By Remark~\ref{rem:joinorderpreserving}, if $z \le z'$ in $\Gamma$,  then $z \vee q \le z' \vee q$, and  we deduce that 
$\hat{\sigma}$ is order-preserving. 
We claim that $\hat{\sigma}$ is rank-increasing.
Indeed, for any $x \in \hat{X}$, we have $x < \hat{\sigma}(x)$ in $\Gamma$ and hence 
$\rho_{\hat{X}}(x) = \rho_{\Gamma}(x) \le \rho_\Gamma(\hat{\sigma}(x)) - 1 = \rho_{\hat{Y}}(\hat{\sigma}(x))$. It remains to show that $\hat{\sigma}$ is strongly surjective, and to check that \eqref{eq:strongsubdivision} holds.	


We first establish strong surjectivity in the following preliminary lemma. In the proof, we use a weaker property than $\Gamma$ being locally Eulerian. Namely, we use the fact that $\Gamma$ is \emph{thin}, i.e.,  every interval of rank $2$ in $\Gamma$ is isomorphic to the Boolean algebra $B_2$. 

\begin{lemma}\label{lem:hatsigma}
	With the notation above, 
	consider $x \in \hat{X}$ and $y \in \hat{Y}$ with $\hat{\sigma}(x) \le y$. 
	Then there exists $x' \in \hat{X}$ such that $x \le x'$,  $\rho_\Gamma(y) - \rho_\Gamma(x') = 1$, and $\hat{\sigma}(x') = y$. In particular, $\hat{\sigma}$ is strongly surjective. 
\end{lemma}
\begin{proof}
	Assume the first statement holds. We have already shown above that $\hat{\sigma}$ is order-preserving and rank-increasing. If we can show that $\hat{\sigma}$ is surjective, then the first statement implies that $\hat{\sigma}$ is strongly surjective. 
	To show that $\hat{\sigma}$ is surjective, 
	consider an element $y \in \hat{Y}$. It is enough to show that there exists $x \in \hat{X}$ such that $\hat{\sigma}(x) \le y$, since  then the first statement implies that there exists $x' \in \hat{X}$ with $\hat{\sigma}(x') = y$.  After possibly replacing $y$ by some $y' \le y$ in $\hat{Y}$, we may assume that $y$ is a minimal element, i.e., $y = \hat{0}_{\hat{Y}} = q$. 
	Consider any element $x$ in $\Gamma$ such that $x < q$, for example, $x = \hat{0}_{\Gamma}$. 
	Then $\hat{\sigma}(x) = x \vee q = q$, as desired.


	
	We now prove the first statement. 
	Consider a maximal chain $x = z_0 < z_1 < \cdots < z_t = \hat{\sigma}(x)$ in the interval $[x,\hat{\sigma}(x)]$ of $\Gamma$. 
	Here $t = \rho_\Gamma(\hat{\sigma}(x)) - \rho_\Gamma(x)$ and $z_i \in \hat{X}$ for $0 \le i < t$. 
	Fix some  $0 \le i < t$. Since $\hat{\sigma}$ is order-preserving, $\hat{\sigma}(x) \le \hat{\sigma}(z_i)$. 	Also, $z_i < \hat{\sigma}(x) \in \Gamma_{\ge q}$ implies that $\hat{\sigma}(z_i) =  z_i \vee q \le \hat{\sigma}(x)$.
	We deduce that $\hat{\sigma}(z_i) = \hat{\sigma}(x)$. 
	Suppose there exists  $x' \in \hat{X}$  with $z_{t - 1} \le x'$, $\rho_\Gamma(y) - \rho_\Gamma(x') = 1$, and $\hat{\sigma}(x') = y$.  Then $x \le z_{t - 1} \le x'$, and the lemma holds. 
	We conclude that, after possibly replacing $x$ by $z_{t - 1}$, we may and will assume that $t = \rho_\Gamma(\hat{\sigma}(x)) - \rho_\Gamma(x) =  1$. 
	
	Consider a maximal chain $\hat{\sigma}(x) = y_0 < y_1 < \cdots < y_s = y$ in the interval $[\hat{\sigma}(x), y]$ of  $\hat{Y}$, where $s = \rho_\Gamma(y) - \rho_\Gamma(\hat{\sigma}(x))$. Suppose that we can construct a chain 
	$x = x_0 < x_1 < \cdots < x_s$ in $\hat{X}$ such that  $\rho_\Gamma(y_i) - \rho_\Gamma(x_i) =  1$  and $\hat{\sigma}(x_i) = y_i$ for $0 \le i \le s$. 
	Then the lemma holds if we set $x' = x_s$. It remains to construct such a chain. 
	
	We proceed by induction. For $0 \le j \le s$, we claim there exists a chain  
	$x = x_0 < x_1 < \cdots < x_j$ in $\hat{X}$ such that
	$\rho_\Gamma(y_i) - \rho_\Gamma(x_i) =  1$  and $\hat{\sigma}(x_i) = y_i$ for $0 \le i \le j$.  The claim holds when $j = 0$ since $\rho_\Gamma(\hat{\sigma}(x)) - \rho_\Gamma(x) =  1$. Suppose the claim holds for some $0 \le j < s$. We will show that it holds for $j + 1$, completing the proof of the claim and the lemma.  Consider the  interval $[x_j,y_{j + 1}]$ in  $\Gamma$. We have $\rho_\Gamma(y_{j + 1}) - \rho_\Gamma(x_j) = 
	(\rho_\Gamma(y_{j + 1}) - \rho_\Gamma(y_j)) + (\rho_\Gamma(y_j) - \rho_\Gamma(x_j)) = 2$. 
	Since $\Gamma$ is thin, $[x_j,y_{j + 1}]$ is isomorphic to $B_2$. 
	Hence $[x_j,y_{j + 1}]$ contains precisely two elements  with rank 
	$\rho_\Gamma(x_j) + 1 = \rho_\Gamma(y_j) = \rho_\Gamma(y_{j + 1}) - 1$ in $\Gamma$, one of which is $y_j$. 	Let $x_{j + 1}$ be the unique rank $\rho_\Gamma(y_j)$ element of $\Gamma$ contained in $[x_j,y_{j + 1}]$ such that $x_{j + 1} \neq y_j$. 
	If $x_{j + 1} \in \Gamma_{\ge q}$, then $\rho_\Gamma(x_{j + 1}) = \rho_\Gamma(x_j) + 1$ implies that $x_{j + 1} = x_j \vee q = \hat{\sigma}(x_j) = y_j$, a contradiction.  We deduce that 
	$x_{j + 1} \in \hat{X}$. Moreover, since $x_{j + 1} \le y_{j + 1}$ in $\Gamma$ and $\rho_\Gamma(y_{j + 1}) = \rho_\Gamma(x_{j + 1}) + 1$, we have $\hat{\sigma}(x_{j + 1}) = x_{j + 1} \vee q = 
	y_{j + 1}$, as desired.
\end{proof}

We now complete the proof that $\MAP$ is well-defined.


\begin{lemma}\label{lem:sfsGammaqissfs}
	The function $\MAP : \JoinIdealLW^\circ  \to \SFS$ is well-defined. 
\end{lemma}
\begin{proof}
	Consider the setup above. By the above discussion and Lemma~\ref{lem:hatsigma}, it remains to show that \eqref{eq:strongsubdivision} holds.
	Consider $x \in \hat{X}$ and $y \in \hat{Y}$ such that $\hat{\sigma}(x) \le y$. Firstly, using the fact that $[x,y] \subset \Gamma$ is an Eulerian poset of positive rank, and then using the fact that $[\hat{\sigma}(x),y] \subset \Gamma$ is an Eulerian poset, we compute
	\begin{align*}
		\sum_{ \substack{x \le x' \in \hat{X} \\ \hat{\sigma}(x') \le y} } (-1)^{\rho_{\hat{Y}}(y) - \rho_{\hat{X}}(x')} &= -\sum_{\substack{z \in \Gamma \\ x \le z \le y \\ z \notin \hat{Y} }} (-1)^{\rho_\Gamma(y) - \rho_\Gamma(z)} 
		\\ 
		&= \sum_{\substack{z \in \Gamma \\ x \le z \le y \\ z \in \hat{Y} }} (-1)^{\rho_\Gamma(y) - \rho_\Gamma(z)} 
		\\
		&= \sum_{\substack{z \in \Gamma \\ \hat{\sigma}(x) \le z \le y }} (-1)^{\rho_\Gamma(y) - \rho_\Gamma(z)} 
		\\
		&=  \begin{cases}
			1 &\textrm{ if } \hat{\sigma}(x) = y, \\
			0 &\textrm{ otherwise. }
		\end{cases}
	\end{align*} 	

\end{proof}

Now that we have established that the functions $\CYL$ and $\MAP$ are well-defined,  we complete the proof of Theorem~\ref{thm:mainsimplified}. 

\begin{proof}[Proof of Theorem~\ref{thm:mainsimplified}]
	We need to show that the functions $\CYL: \SFS \to  \JoinIdealLW^\circ$  and $\MAP:  \JoinIdealLW^\circ \to \SFS$ are mutually inverse bijections.  
	
	
	Let  $\sigma: X \to Y$ be a  strong formal subdivision  between lower Eulerian posets $X$ and $Y$ with rank functions $\rho_X$ and $\rho_Y$ respectively. 
	Then $\CYL(\sigma) = (\Cyl(\sigma), \rho_{\Cyl(\sigma)},\hat{0}_Y)$, where 
	\[
	\rho_{\Cyl(\sigma)}(z) =  \begin{cases}
		\rho_X(z) &\textrm{ if } z \in X, \\
		\rho_Y(z) + 1 &\textrm{ if } z \in Y. 
	\end{cases}
	\]		
	
	and 
	\begin{equation*}
		\MAP(\CYL(\sigma)) : \Cyl(\sigma) \smallsetminus \Cyl(\sigma)_{\ge \hat{0}_Y} \to \Cyl(\sigma)_{\ge \hat{0}_Y},
	\end{equation*}
	\[
	x \mapsto x \vee \hat{0}_Y, 
	\]
	where the rank functions on $\Cyl(\sigma) \smallsetminus \Cyl(\sigma)_{\ge \hat{0}_Y}$ and $\Cyl(\sigma)_{\ge \hat{0}_Y}$ are  the  corresponding restrictions of $\rho_\Gamma$ shifted by $0$ and $-1$ respectively.
	Observe that $X = \Cyl(\sigma) \smallsetminus \Cyl(\sigma)_{\ge \hat{0}_Y}$,  $Y = \Cyl(\sigma)_{\ge \hat{0}_Y}$, and $x \vee \hat{0}_Y = \sigma(x)$ for all $x \in X$. We conclude that $\MAP(\CYL(\sigma)) = \sigma$, as desired.
	
	Conversely, consider a triple $(\Gamma,\rho_\Gamma,q)$ in $\JoinIdealLW^\circ$. 
	Let $X = \Gamma \smallsetminus \Gamma_{\ge q}$, $Y = \Gamma_{\ge q}$, and $\sigma = \MAP(\Gamma,\rho_\Gamma,q)$.
	Then
	\begin{equation*}
		\sigma : X \to Y,
	\end{equation*}
	\[
	x \mapsto x \vee q,
	\]
	where the rank functions $\rho_X$ for $X$ and $\rho_Y$ for $Y$ are  the  corresponding restrictions of $\rho_\Gamma$ shifted by $0$ and $-1$ respectively. 		Then $\CYL(\sigma) = (\Cyl(\sigma), \rho_{\Cyl(\sigma)},\hat{0}_Y)$, where 
	\[
	\rho_{\Cyl(\sigma)}(z) =  \begin{cases}
		\rho_X(z) &\textrm{ if } z \in X, \\
		\rho_Y(z) + 1 &\textrm{ if } z \in Y. 
	\end{cases}
	\]		
	Observe that 
	$\Cyl(\sigma) = \Gamma$ as sets, and,   
	with this identification,  $\rho_{\Cyl(\sigma)} = \rho_\Gamma$, and $\hat{0}_Y = q$. 
	If we can show that $\Cyl(\sigma) = \Gamma$ as posets, then $\CYL(\sigma) = (\Gamma,\rho_\Gamma,q)$, completing the proof.
	
	Consider $z,z' \in \Cyl(\sigma)$. If $z,z'$ both lie in $X$ or both lie in $Y$, then $z \le z' \in \Cyl(\sigma)$ if and only if $z \le z'$ in  $\Gamma$. Suppose that 
	$z \in X$ and $z' \in Y$. Then $z \le z'$  in $\Cyl(\sigma)$ if and only if $\sigma(z) = z \vee q \le z'$ in $\Gamma$. 
	If $z \vee q \le z'$ in $\Gamma$, then $z \le z \vee q \le z'$ in $\Gamma$. On the other hand, by Remark~\ref{rem:joinorderpreserving}, if $z \le z'$ in $\Gamma$ then 
	$z \vee q \le z' \vee q = z'$ in $\Gamma$. We conclude that $z \le z' \in \Cyl(\sigma)$ if and only if $z \le z'$ in $\Gamma$. Hence $\Cyl(\sigma) = \Gamma$ as posets.
	
\end{proof}

We next give the proof of Theorem~\ref{thm:main}.

\begin{proof}[Proof of Theorem~\ref{thm:main}]
	By Lemma~\ref{lem:CYLwelldefined} and Lemma~\ref{lem:MAPwelldefined} below, $\CYLcat$ and $\MAPcat$ are well-defined functors. Then $\CYLcat$ and $\MAPcat$ induce mutually inverse bijections between $\Obj(\CYLcat)$ and $\Obj(\MAPcat)$ by Theorem~\ref{thm:mainsimplified}. Moreover, using Theorem~\ref{thm:mainsimplified} and tracing through the definitions, it follows that  $\CYLcat$ and $\MAPcat$ induce mutually inverse bijections between $\Mor(\CYLcat)$ and $\Mor(\MAPcat)$. In slightly more detail, with the notation of Definition~\ref{def:functorialbijections}, in the definition of $\CYLcat$, $\phi$ is constructed by concatenating $\phi_1$ and $\phi_2$, while in the definition of $\MAPcat$, $\phi_1$ and $\phi_2$ are obtained by restricting $\phi$, and these two operations are inverse to each other.
	
	
	
\end{proof}

\begin{lemma}\label{lem:CYLwelldefined}
	The functor $\CYLcat: \SFScat \to  \Joincat^\circ$ is well-defined.
\end{lemma} 
\begin{proof}
	By Theorem~\ref{thm:mainsimplified}, $\CYLcat$ defines a well-defined function from $\Obj(\SFScat)$ to $\Obj(\Joincat^\circ)$. 
	Suppose we can show that the image of a morphism in $\SFScat$ under $\CYLcat$ is a well-defined morphism in $\Joincat^\circ$. Then by tracing through the definitions in Definition~\ref{def:functorialbijections}, it follows that $\CYLcat$ takes the identity morphism of an object in $\SFScat$ to the identity morphism of the corresponding object in $\Joincat^\circ$, and also that $\CYLcat$ commutes with composition of morphisms. We conclude that $\CYLcat$ is a well-defined functor. 
	
	Consider a morphism in $\SFScat$ given by the commutative diagram
	\[
	\begin{tikzcd} X \arrow[r, "\sigma"] \arrow[d, "\phi_1"] & Y \arrow[d, "\phi_2"] \\ X' \arrow[r, "\sigma'"] &  Y'. \end{tikzcd}
	\]
	Here $\sigma,\sigma',\phi_1,\phi_2$ are strong formal subdivisions.
	Let $\rho_X, \rho_{X'}, \rho_Y, \rho_{Y'}$ be the rank functions corresponding to $X, X', Y, Y'$ respectively. With the notation of Definition~\ref{def:bijections},    $\CYLcat(\sigma) = (\Cyl(\sigma),\rho_{\Cyl(\sigma)},q)$, where 
	$q = \hat{0}_Y \in \Cyl(\sigma)$, and $\CYLcat(\sigma') = (\Cyl(\sigma'),\rho_{\Cyl(\sigma')},q')$, where 
	$q' = \hat{0}_{Y'} \in \Cyl(\sigma')$.  By Definition~\ref{def:functorialbijections}, 
	the image of the above morphism under $\CYLcat$ is 
	\[
	\phi: \Cyl(\sigma) \to \Cyl(\sigma'),
	\]
	\[
	\phi(z) = \begin{cases}
		\phi_1(z) &\textrm{ if } z \in X, \\
		\phi_2(z) &\textrm{ if } z \in Y.
	\end{cases}
	\]

	Our goal is to show that $\phi$ defines a morphism from $\CYLcat(\sigma)$ to $\CYLcat(\sigma')$ in $\Joincat^\circ$. 
	
	We first show that $\phi$ is a strong formal subdivision.  	
	We claim that $\phi$ is order-preserving. Since $\phi_1$ and $\phi_2$ are order-preserving,  it is enough to show that 
	$\phi(x) \le \phi(y)$ in $\Cyl(\sigma')$ for any $x \in X$ and $y \in Y$ such that $\sigma(x) \le y$.  Since $\phi(x) = \phi_1(x) \in X'$ and $\phi(y) = \phi_2(y) \in Y'$, we need to show that $\sigma'(\phi_1(x)) \le \phi_2(y)$. Since $\phi_2$ is order-preserving, we have 
	$\sigma'(\phi_1(x)) =  \phi_2(\sigma(x)) \le \phi_2(y)$, 
	as desired.	
	
	We claim that $\phi$ is rank-increasing. This follows since $\phi_1$ and $\phi_2$ are rank-increasing. Explicitly, using \eqref{eq:rhoCyl},
	$$\rho_{\Cyl(\sigma)}(x) = \rho_X(x) \le \rho_{X'}(\phi_1(x)) = \rho_{\Cyl(\sigma')}(\phi(x)),$$ 
	for any $x \in X$, and, 
	$$\rho_{\Cyl(\sigma)}(y) = \rho_Y(y) + 1 \le \rho_{Y'}(\phi_2(y)) + 1 = \rho_{\Cyl(\sigma')}(\phi(y)),$$ 
	for any $y \in Y$. 
	
	Since $\phi_1$ and $\phi_2$ are surjective, it follows that $\phi$ is surjective. We want to show that $\phi$ is strongly surjective. Since $\phi_1$ and $\phi_2$ are strongly surjective, it is enough to show that given any $x \in X$ and $y' \in Y'$ such that  $\phi(x) \le y' \in \Cyl(\sigma')$, there exists some $y \in Y$ such that $\sigma(x) \le y$, $\phi(y) = y'$, and 
	$\rho_{\Cyl(\sigma)}(y) = \rho_{\Cyl(\sigma')}(y')$. Observe that 
	$\phi_2(\sigma(x)) = \sigma'(\phi_1(x)) = \sigma'(\phi(x))  \le y'$. Since $\phi_2$ is strongly surjective, there exists $\sigma(x) \le y$ such that $\rho_Y(y) = \rho_{Y'}(y')$ and $\phi_2(y) = \phi(y) = y'$. 
	By \eqref{eq:rhoCyl}, $\rho_{\Cyl(\sigma)}(y) = \rho_Y(y) + 1 = \rho_{Y'}(y') + 1 = \rho_{\Cyl(\sigma')}(y')$, as desired.
	
	It follows from \eqref{eq:rhoCyl} and the fact that $\phi_1$ and $\phi_2$ satisfy \eqref{eq:strongsubdivisionequality}, that $\phi$ also satisfies \eqref{eq:strongsubdivisionequality}. Explicitly, it is enough to show that \eqref{eq:strongsubdivisionequality} holds for any $x \in X$ and $y' \in Y'$ such that  $\phi(x) \le y'$. Since $\phi^{-1}(y') \subset Y$, we compute
	\begin{align*}
		\sum_{ \substack{x \le z \in \Cyl(\sigma) \\ \phi(z) = y'} } (-1)^{\rho_{\Cyl(\sigma')}(y') - \rho_{\Cyl(\sigma)}(z) } &= 
		\sum_{ \substack{\sigma(x) \le y \in Y \\ \phi_2(y) = y'} } (-1)^{\rho_{\Cyl(\sigma')}(y') - \rho_{\Cyl(\sigma)}(y) } \\
		&= 
		\sum_{ \substack{\sigma(x) \le y \in Y \\ \phi_2(y) = y'} } (-1)^{\rho_{Y'}(y') - \rho_{Y}(y) }.
	\end{align*}
	Since $\phi_2$ satisfies \eqref{eq:strongsubdivisionequality}, the final summation above is $1$, as desired. We conclude that $\phi$ is a strong formal subdivision. 
	
	Using Remark~\ref{rem:zerotozero} applied to $\phi_2$, we compute $$\phi(q) = \phi(\hat{0}_Y) = \phi_2(\hat{0}_Y) = \hat{0}_{Y'} = q'.$$
	Finally, for all $x \in X = \Cyl(\sigma) \smallsetminus \Cyl(\sigma)_{\ge q}$, we compute
	$$\phi(x \vee q) = \phi(\sigma(x)) = \phi_2(\sigma(x)) = \sigma'(\phi_1(x)) = \phi_1(x) \vee q' = \phi(x) \vee q'.$$ 
	This completes the proof that $\phi$ determines a well-defined morphism in $\Joincat^\circ$. 
\end{proof}

\begin{lemma}\label{lem:MAPwelldefined}
	The functor $\MAPcat: \Joincat^\circ \to  \SFScat$ is well-defined.
\end{lemma}  
\begin{proof}
	By Theorem~\ref{thm:mainsimplified}, $\MAPcat$ defines a well-defined function from $\Obj(\Joincat^\circ)$ to $\Obj(\SFScat)$. 
	Suppose we can show that the image of a morphism in $\Joincat^\circ$ under $\MAPcat$ is a well-defined morphism in $\SFScat$. Then by tracing through the definitions in Definition~\ref{def:functorialbijections}, it follows that $\MAPcat$ takes the identity morphism of an object in $\Joincat^\circ$ to the identity morphism of the corresponding object in $\SFScat$, and also that $\MAPcat$ commutes with composition of morphisms. We conclude that $\MAPcat$ is a well-defined functor. 
	
	Consider a morphism  in $\Joincat^\circ$ between objects $(\Gamma,\rho_\Gamma,q)$ and $(\Gamma',\rho_{\Gamma'},q')$ defined by a strong formal subdivision $\phi: \Gamma \to \Gamma'$. 
	Let $X = \Gamma \smallsetminus \Gamma_{\ge q}$, $Y = \Gamma_{\ge q}$, $X' = \Gamma' \smallsetminus \Gamma_{\ge q'}'$, and $Y' = \Gamma_{\ge q'}'$. 	
	By Definition~\ref{def:functorialbijections}, 
	the image of the morphism under $\MAPcat$ is defined by
	%
	\[
	\MAPcat(\phi: \Gamma \to \Gamma') = 
	\begin{tikzcd} X \arrow[r, "\boldsymbol{\cdot} \vee q"] \arrow[d, "\phi_1"] & Y \arrow[d, "\phi_2"] \\ X' \arrow[r, "\boldsymbol{\cdot} \vee q'"] &  Y' \end{tikzcd},
	\]
	where $\phi_1$ and $\phi_2$ are restrictions of $\phi$. 
	
	Assume for the moment that 
	$\phi^{-1}(Y') = Y$. 
	Then the diagram above is well-defined, and 
	using Definition~\ref{def:functorialbijections}, we may check that it is a commutative diagram. 
	Indeed, for any $x \in X$, we have 
	\[
	\phi_1(x) \vee q' = \phi(x) \vee q' = \phi(x \vee q) = \phi_2(x \vee q).  
	\]
	Since $X'$ is a lower order ideal of $\Gamma'$, $\phi^{-1}(X') = X$,  and the rank functions on $X$ and $X'$ are the corresponding restrictions of $\rho_{\Gamma}$ and $\rho_{\Gamma'}$ respectively,  Remark~\ref{rem:restrictsfs} implies that the restriction  $\phi_1: X \to X'$ of $\phi$ is a strong formal subdivision. 
	Since $\phi(q) = q'$, Remark~\ref{rem:restrictsfs} implies that $\phi$ restricts to a strong formal subdivision $\Gamma_{\ge q} \to \Gamma_{\ge q'}$.  
	The rank functions on $Y$ and $Y'$ are the corresponding restrictions of $\rho_{\Gamma}$ and $\rho_{\Gamma'}$ respectively shifted by $-1$, and we deduce that 	$\phi_2$ is a strong formal subdivision using Remark~\ref{rem:rankfunctionunique}.
	
	It remains to show that  $\phi^{-1}(Y') = Y$. Since $\phi(q) = q'$ and $\phi$ is order-preserving, we have $\phi(Y) \subset Y'$ and hence 
	$Y \subset \phi^{-1}(Y')$. For the converse, suppose there exists $x \in X$ such that $\phi(x) \in Y'$. We need to deduce a contradiction. 
	If $x \le x' \in X$, then since $\phi$ is order-preserving, $\phi(x) \le \phi(x')$ in $\Gamma'$. Since $Y'$ is an upper order ideal of $\Gamma'$, we deduce that $\phi(x') \in Y'$. Hence, we may assume that $x$ is a maximal element of $X$. 	By Remark~\ref{rem:restrictsfs}, $\phi$ restricts to a strong formal subdivision $\Gamma_{\ge x} \to \Gamma_{\ge y'}'$. By maximality of $x$, $\Gamma_{> x} \subset Y$, and we deduce that 
	$\Gamma_{> x} = \Gamma_{\ge x \vee q}$ is lower Eulerian. If there exists  $x \vee q < y$ in $\Gamma$, then $|[x \vee q,y]|$ is even,  
	and $|[x,y]| = |[x \vee q,y]| + 1$ is odd, contradicting the Eulerian hypothesis. We deduce that $\Gamma_{\ge x} = \{ x , x \vee q \} = B_1$ and $|\Gamma_{\ge x}|$ is even.
	Since the restriction $\Gamma_{\ge x} \to \Gamma_{\ge y'}'$ is surjective and $\phi(x) = \phi(x) \vee q' = \phi(x \vee q) = y'$,  we deduce that $\Gamma_{\ge y'}' = \{ y' \}$ and $|\Gamma_{\ge y'}'|$ is odd. 
	This contradicts Lemma~\ref{lem:parity}. 
	
\end{proof}

\section{The $cd$-index of an Eulerian poset}\label{sec:cdindex}

In this section, we prove an application of Theorem~\ref{thm:mainsimplified} for $cd$-indices of Eulerian posets.

\subsection{Background on $cd$-indices}\label{ss:cdindexbackground}
We first recall some background on $cd$-indices and refer the reader to 
\cite{BayerCDIndexSurvey} and \cite{StanleyFlagfVectors}  for more details.
Let $B$ be an Eulerian poset of rank $n + 1$, for some nonnegative integer $n$. Let $\rho_B$ be the natural rank function. 
Given a subset  $S \subset [n] = \{ 1,\ldots,n \}$, define a monomial $w_S = w_1 \cdots w_n$  in noncommuting variables $a,b$ by setting 
$w_i = a - b$ if $i \notin S$, and  $w_i = b$ if $i \in S$, for $1 \le i \le n$. 
Let $f_S$ be the number of maximal chains of the $S$-rank selected poset $B_S = \{ z \in B : \rho_B(z) \in S \} \cup \{ \hat{0}, \hat{1} \}$. The 
\emph{$ab$-polynomial} is $\Psi(B) = \Psi(B;a,b) = \sum_{S \subset [n]} f_S w_S$. 
The \emph{$cd$-index} is the unique polynomial $\Phi(B) = \Phi(B;c,d)$ in noncommuting variables $c$ and $d$ such that $\Phi(B;a + b,ab + ba) = \Psi(B;a,b)$. The fact that the $cd$-index exists is due to Fine (see \cite{BKNewIndexPolytopes}). Observe that the $cd$-index is homogeneous of degree $n$, where we set $\deg(c) = 1$ and $\deg(d) = 2$. 

Recall that a \emph{derivation} of a ring $A$ is a $\Z$-linear map $d: A \to A$ such that $d(ab) = a d(b) + d(a)b$ for all $a,b \in A$. In particular, $d(1) = 0$ and $d$ is determined by its values on a set of generators for $A$. In \cite{ERCoproductscdindex}, Ehrenborg and Readdy introduced the following derivations of the ring $\Z\langle c,d \rangle$: 
\begin{enumerate}
	\item  	Let $G$ be the derivation determined by $G(c) = d$ and $G(d) = cd$,
	\item  	Let $G'$ be the derivation determined by $G'(c) = d$ and $G'(d) = dc$,
	\item  	Let $D = G + G'$. 
\end{enumerate}
Let $B$ be an Eulerian poset of positive rank. They showed that 
\[
D(\Phi(B)) = \sum_{\hat{0}_B < z < \hat{1}_B}  \Phi([\hat{0}_B,z]) d \Phi([z,\hat{1}_B]).
\]

\begin{example}\label{ex:derivation}
	Let $B$ be an Eulerian poset of positive rank. 
	Ehrenborg and Readdy \cite{ERCoproductscdindex}*{Theorem~4.4, Lemma~5.1, Theorem~5.2} proved that 
	\[
	\Phi(\Pyr(B)) = \Phi(B)c + G(\Phi(B)) = c\Phi(B)  + G'(\Phi(B)) =  \frac{1}{2} \left[ \Phi(B) c +  c \Phi(B) +  D(\Phi(B)) \right].
	\]
	In particular, since $\Pyr(B_{n + 1}) = B_{n + 2}$, this allows one to recursively compute
	$\{ \Phi(B_{n + 1}) : n \ge 0 \}$.  For example, 
	\[
	\Phi(B_{n + 1}) = \begin{cases}
		1 &\textrm{if } n = 0, \\
		c &\textrm{if } n = 1, \\
		c^2 + d &\textrm{if } n = 2, \\
		c^3 + 2cd + 2dc	 &\textrm{if } n = 3, \\
		c^4 + 3 d c^2 + 5cd c + 3c^2d + 4d^2 		 &\textrm{if } n = 4. \\
	\end{cases}
	\]
\end{example}

\begin{example}\label{ex:cddual}
	Let $B$ be an Eulerian poset of positive rank. 
	Recall from Section~\ref{ss:posets} that the dual poset $B^*$ is the poset with the same elements as $B$ and with all orderings reversed. 
	Then $\Psi(B^*;a,b)$ is obtained from $\Psi(B;a,b)$ by reversing all monomials in $a$ and $b$, and $\Phi(B^*;c,d)$ is obtained from $\Phi(B;c,d)$ by reversing all monomials in $c$ and $d$.
	For example, $B_{n + 1}^* = B_{n + 1}$ and hence $\Phi(B_{n + 1})$ is invariant under reversing monomials (c.f. Example~\ref{ex:derivation}). 
\end{example}

\begin{example}\label{ex:ERPrismPyr}
	Let $B$ be an Eulerian poset of positive rank. 
	Ehrenborg and Readdy  		\cite{ERCoproductscdindex}*{Proposition~4.2} proved that
	\[
	\Phi(\Prism(B)) = \Phi(B)c +   D(\Phi(B)).
	\]
	By Example~\ref{ex:cddual}, this is equivalent to the following formula \cite{ERCoproductscdindex}*{Corollary~4.7}
	\[
	\Phi(\Bipyr(B)) = \Phi(\Prism(B^*)^*) =  c \Phi(B) + D(\Phi(B)).
	\]
	By Example~\ref{ex:derivation}, 
	we can reexpress these formulas as 
	$\Phi(\Prism(B)) = 2 G'(\Phi(B)) + c\Phi(B)$ and 
	$\Phi(\Bipyr(B)) = 2 G(\Phi(B)) + \Phi(B)c$ respectively. 
\end{example}



\begin{example}\label{ex:cdGorenstein*}
	Let $B$ be a Gorenstein* poset over a field. Then Karu \cite{KaruCDIndex} proved that the coefficients of the $cd$-index $\Phi(B;c,d)$ are nonnegative integers. 
\end{example}

\begin{example}\label{ex:cdstarproduct}
	Let $B$ and $B'$ be Eulerian posets of positive rank. Recall from Section~\ref{ss:posets} that the star product  $B \ast B'$ is an Eulerian poset of rank $\rank(B) + \rank(B') - 1$.  For example,   $B_1 \ast B' = B'$, $B \ast B_1 = B$,  and $B \ast B_2 = \tilde{\Sigma} B$. 
	Stanley showed that the $cd$-index is multiplicative with respect to the star product, i.e., $\Phi(B \ast B') = \Phi(B)\Phi(B')$ \cite{StanleyFlagfVectors}. 
	In particular, using Example~\ref{ex:derivation}, $\Phi(\tilde{\Sigma} B) = \Phi(B)\Phi(B_2) = \Phi(B)c$. 
\end{example}

Let $B$ be a near-Eulerian poset. 
Recall that we may associate to $B$ two distinct Eulerian posets of positive rank; the semisuspension $\tilde{\Sigma} B$ and $\overline{\partial B}$. Recall that  $\partial B$ is the boundary of $B$ and $\overline{\partial B}$ is obtained from $\partial B$ by adjoining a maximal element. 
The \emph{local $cd$-index} is 
\begin{equation}\label{eq:localcdindex}
\ell^{\Phi}(B) = \ell^{\Phi}(B;c,d) := \Phi(\tilde{\Sigma} B) -\Phi(\overline{\partial B})c.
\end{equation}
This was first defined by Karu for near-Gorenstein* posets in \cite{KaruCDIndex}, and later considered for near-Eulerian posets in \cite[Section~3.1]{DKTPosetSubdivisions}. 
Since $\rank(\tilde{\Sigma} B) = \rank(B) + 1 = \rank(\overline{\partial B}) + 1$, the local $cd$-index is homogeneous of degree $\rank(B)$. 

\begin{example}\cite{DKTPosetSubdivisions}*{Proposition~3.7}
	Let $B$ be Eulerian of positive rank. By Example~\ref{ex:cdstarproduct}, $\ell^{\Phi}(B) = 0$. 
\end{example}

\begin{example}\label{ex:partialPyr}
	Let $B$ be an Eulerian poset of positive rank. Then $\Pyr(\partial B)$ is near-Eulerian with semisuspension $\Pyr(B)$ and $\partial \Pyr(\partial B) = \partial B$. By Example~\ref{ex:derivation},  $\ell^{\Phi}(\Pyr(\partial B)) = \Phi(\Pyr(B)) -\Phi(B)c = G(\Phi(B))$, where $G$ is the derivation of the ring $\Z\langle c,d \rangle$ determined by $G(c) = d$ and $G(d) = cd$. For example, using Example~\ref{ex:derivation}, 
	\[
	\ell^{\Phi}(\Pyr(\partial B_{n + 1})) = \begin{cases}
		0 &\textrm{if } n = 0, \\
		d &\textrm{if } n = 1, \\
		2cd + dc &\textrm{if } n = 2, \\
		dc^2  + 3cdc + 3c^2d + 4d^2 &\textrm{if } n = 3. \\
	\end{cases}
	\]
\end{example}

\begin{example}\label{ex:cdlocalnonnegative}
	Let $B$ be a near-Gorenstein* poset over a field. Then Ehrenborg and Karu \cite{EKDecompositionTheoremCDIndex}*{Theorem~5.6} proved that the coefficients of the local $cd$-index $\ell^\Phi(B)$ are nonnegative integers. 
\end{example}

\subsection{Subdivisions and $cd$-indices}
In this section, we state and prove our application to $cd$-indices, as well as giving some remarks and examples.


Recall from Remark~\ref{rem:altcharnearEulerian} that if $\sigma: X \to Y$ is a strong formal subdivision between lower Eulerian posets, then for all $y \in Y$, we write $X_{\le y} = \sigma^{-1}([\hat{0}_Y,y])$ and $X_{< y} = \sigma^{-1}([\hat{0}_Y,y))$, and 
either 
\begin{enumerate}
	\item $y = \hat{0}_Y$ and $X_{\le y}$ is the boundary of an Eulerian poset of positive rank, or,
	\item $y \neq \hat{0}_Y$ and $X_{\le y}$ is near-Eulerian with boundary $X_{< y}$. 
\end{enumerate}
The following proposition shows that if $Y$ is Eulerian, then the $cd$-index of the non-Hausdorff mapping cylinder $\Cyl(\sigma)$ can be expressed in terms of the $cd$-indices and local $cd$-indices of invariants associated to $\sigma$. 

\begin{proposition}\label{prop:cdsubdivision}
	Let $\Gamma$ be an Eulerian poset of positive rank with rank function $\rho_\Gamma$. Let  $q$ be a  join-admissible element of $\Gamma$ such that $q \notin \{ \hat{0}_\Gamma, \hat{1}_\Gamma \}$. Let $\sigma: X \to Y$ be the corresponding strong formal subdivision under Theorem~\ref{thm:mainsimplified}, i.e., 
	$X = \Gamma \smallsetminus \Gamma_{\ge q}$, $Y = \Gamma_{\ge q}$, $\sigma(x) = x \vee q$ for all $x \in X$, and $\Gamma = \Cyl(\sigma)$. Then 
	\begin{equation}\label{eq:cdformula}
		\Phi(\Gamma) = \ell^{\Phi}(X) + \frac{1}{2} \left[ \Phi(\overline{\partial X}) c + 
		\Phi(\overline{X_{ \le \hat{0}_Y}}) c \Phi(Y) +  \sum_{\hat{0}_Y < y < \hat{1}_Y} \left(\ell^{\Phi}(X_{\le y})c + \Phi(\overline{X_{< y}}) d \right)\Phi([y,\hat{1}_Y]) \right].
	\end{equation}
\end{proposition}  

In the statement above, observe that $X_{ < y} = \partial X_{\le y}$ for $y \neq \hat{0}_Y$, and $\overline{X_{ \le \hat{0}_Y}} = [\hat{0}_\Gamma,q]$. 
The proof involves careful tracking of the chains in $\Gamma$. 
Before giving the proof, we make some remarks and give some examples.

\begin{remark}\label{rem:recursion}
	In Proposition~\ref{prop:cdsubdivision}, observe that all posets appearing in the right hand side of \eqref{eq:cdformula} have rank strictly less than $\rank(\Gamma)$, and contain strictly fewer elements than $\Gamma$. If we write
	$\ell^{\Phi}(X_{\le y}) = \Phi(\tilde{\Sigma} X_{\le y}) - \Phi(\overline{X_{< y}})c$, then $\rank(\overline{X_{< y}}) <  \rank(\Gamma)$,  and $\rank(\tilde{\Sigma} X_{\le y}) \le  \rank(\Gamma)$ with equality if and only if $y = \hat{1}_Y$. 
	Also, $|\overline{X_{< y}}| < |\Gamma|$, and $|\tilde{\Sigma} X_{\le y}| = |X_{\le y}| + 2 \le |\Gamma| = |X| + |Y|$, with equality if and only if $y = \hat{1}_Y$ and $|Y| = 2$. 
	
	We conclude that if we  further assume that $\rho_\Gamma(q,\hat{1}_\Gamma) > 1$ in the statement of Proposition~\ref{prop:cdsubdivision}, then the right hand side of \eqref{eq:cdformula} can be expressed in terms of the $cd$-indices of Eulerian posets of positive rank that contain strictly fewer elements than $\Gamma$. If $\rho_\Gamma(q,\hat{1}_\Gamma) = 1$, then $Y = B_1$ and we are in the situation of Example~\ref{ex:B1}. In that case, 
	\eqref{eq:cdformula} holds by definition, since $\Gamma = \tilde{\Sigma}X$ and the right hand side expands to $\ell^{\Phi}(X) + \frac{1}{2} (\Phi(\overline{\partial X}) c + 
	\Phi(\overline{\partial X}) c \Phi(B_1)) = \Phi(\tilde{\Sigma}X)$. 
\end{remark}


\begin{example}\label{ex:Gorensteinstarnonnegative}
	Assume in the statement of Proposition~\ref{prop:cdsubdivision} that $\Gamma$ is a Gorenstein* poset. 
	We claim that all terms in \eqref{eq:cdformula} have nonnegative integer coefficients. 
	
	We now prove the claim.
	The posets $\{ [y,\hat{1}_Y] : y \neq \hat{1}_Y \}$ are intervals of positive rank in a Gorenstein* poset, and hence are Gorenstein*.
	Recall from Section~\ref{ss:posets} that if $B$ is the boundary of a near-Gorenstein poset, then $\overline{B}$ is Gorenstein*. For any $y \in Y$ with $y \neq \hat{0}_B$,
	it follows from 
	Lemma~\ref{lem:GorensteinstarsubdivEK}  that 
	$X_{\le y}$ is near-Gorenstein* with boundary $X_{< y}$. Hence the posets
	$\{ \overline{X_{< y}} : y \neq \hat{0}_Y \}$ are Gorenstein*. 
	Lemma~\ref{lem:GorensteinstarsubdivEK}  also implies that  $\overline{X_{ \le \hat{0}_Y}}$ is Gorenstein*. Then  Example~\ref{ex:cdGorenstein*} implies that the coefficients of the following polynomials are nonnegative: $\Phi(\Gamma)$, 
	$\Phi(\overline{X_{ \le \hat{0}_Y}})$, $\{ \Phi(\overline{X_{< y}}) : y \neq \hat{0}_Y \}$, and $\{ \Phi([y,\hat{1}_Y]) : y \neq \hat{1}_Y \}$.
	Moreover, Example~\ref{ex:cdlocalnonnegative} implies that 
	the coefficients of $\{ \ell^{\Phi}(X_{\le y}) :  y \neq \hat{0}_Y \}$ are nonnegative.
\end{example}

\begin{example}\label{ex:identitycdindex}
	Let $B$ be an Eulerian poset of positive rank with rank function $\rho_B$. Recall from  Example~\ref{ex:identity} that the identity map $\id_B: B \to B$ is a strong formal subdivision with non-Hausdorff mapping cylinder $\Gamma = \Cyl(\id_B) =  \Pyr(B)$. 
%
	In this case, Proposition~\ref{prop:cdsubdivision} says that
	\begin{equation*}
		\Phi(\Pyr(B)) =   \frac{1}{2} \left[ \Phi(B) c +  c \Phi(B) +  \sum_{\hat{0}_B < z < \hat{1}_B}  \Phi([\hat{0}_B,z]) d \Phi([z,\hat{1}_B]) \right] = \frac{1}{2} \left[ \Phi(B) c +  c \Phi(B) +  D(\Phi(B)) \right].
	\end{equation*}
	Hence we recover the formula of Ehrenborg and Readdy from Example~\ref{ex:derivation}. 
\end{example}

\begin{example}\label{ex:cdbipyramid}
	Recall the setup of Example~\ref{ex:bipyramid}. 
	That is, let $B$ be an Eulerian poset of positive rank with the natural rank function $\rho_B$, and
	consider the strong formal subdivision of rank $0$ given by
	\[
	\sigma:  \Pyr(\partial B) \to B,
	\]
	\[
	\sigma(z,z') = \begin{cases}
		z &\textrm{if } z' = \hat{0} \in B_1, \\
		\hat{1}_B &\textrm{if } z' = \hat{1} \in B_1. \\
	\end{cases}
	\]
	Recall that $\Gamma = \Cyl(\sigma) = \Bipyr(B) = B \diamond^\ast B_2 = \overline{\partial B \times \partial B_2}$. Recall from Example~\ref{ex:partialPyr} that $\ell^{\Phi}(\Pyr(\partial B)) = \Phi(\Pyr(B)) -\Phi(B)c = G(\Phi(B))$, where $G$ is the derivation of the ring $\Z\langle c,d \rangle$ determined by $G(c) = d$ and $G(d) = cd$.
	In this case, using Example~\ref{ex:derivation},  Proposition~\ref{prop:cdsubdivision} says that
	\begin{align*}
		\Phi(\Bipyr(B)) &= \Phi(\Pyr(B)) - \Phi(B)c + \frac{1}{2} \left[ \Phi(B) c + 
		c \Phi(B) 
		+  \sum_{\hat{0}_B < z < \hat{1}_B}  \Phi([\hat{0}_B,z]) d \Phi([z,\hat{1}_B])
		\right] \\
		&= 2\Phi(\Pyr(B)) - \Phi(B)c \\ 
		&= c \Phi(B) +  D(\Phi(B)). \\
	\end{align*}	
	Hence we recover the formula of Ehrenborg and Readdy for $\Phi(\Bipyr(B))$, and hence the formula for $\Phi(\Prism(B))$, from Example~\ref{ex:ERPrismPyr}.  
	
	
\end{example}

		\begin{example}
		Let $B$ be an Eulerian poset of positive rank with rank function $\rho_B$.  Suppose that $(\Gamma,\rho_\Gamma,q) \in \JoinIdealLW^\circ$ corresponds to a strong formal subdivision $\sigma: X \to Y$. Assume that $\Gamma$ is Eulerian and $q \neq \hat{1}_\Gamma$. Recall from Example~\ref{ex:starproductidentity} that 
			$(B \ast \Gamma, \rho_{B \ast \Gamma}, q)  \in \JoinIdealLW^\circ$ with corresponding strong formal subdivision
		\[
		B \ast X \to Y,
		\]
		\[
		z \mapsto \begin{cases}
			q &\textrm{ if } z \in \partial B, \\
			\sigma(z) &\textrm{ if } z \in X \smallsetminus \{ \hat{0}_X \}.
		\end{cases}
		\]
		Consider the effect of replacing $\sigma$ by the above strong formal subdivision on the  formula \eqref{eq:cdformula}. 
		This leaves the intervals $\{ [y,\hat{1}_Y] : y \in \partial Y \}$ unchanged, and 
		replaces all other posets $Z$ by $B \ast Z$, i.e., $Z \in \{ \Gamma,  \overline{X_{ \le \hat{0}_Y}},   \{ X_{\le y} : \hat{0}_Y \neq y \in Y \}, \{ \overline{X_{< y}} : \hat{0}_Y \neq y \in Y \}  \}$. 
		By Example~\ref{ex:cdstarproduct}, the effect is to  multiply all terms in the formula \eqref{eq:cdformula} by $\Phi(B)$ on the left. 
\end{example}

We end this section with the proof of Proposition~\ref{prop:cdsubdivision}. 

\begin{proof}[Proof of Proposition~\ref{prop:cdsubdivision}]
	We first consider some preliminaries.
	Let $B$ be an Eulerian poset of rank $n + 1$ and let $\rho_B$ be the natural rank function. Recall that given a subset  $S \subset [n]$, $w_S = w_1 \cdots w_n$, where 
	$w_i = a - b$ if $i \notin S$, and  $w_i = b$ if $i \in S$, for $1 \le i \le n$.  Recall that $f_S$ is the number of maximal chains of the $S$-rank selected poset $B_S = \{ z \in B : \rho_B(z) \in S \} \cup \{ \hat{0}_B, \hat{1}_B \}$, and $\Psi(B) = \sum_{S \subset [n]} f_S w_S$.
	
	Suppose that $B'$ is a near-Eulerian poset with semisuspension $B$, i.e., 
	$B = B' \cup \{ \hat{z}, \hat{1}_B \}$ for some maximal element $\hat{z} \in \partial B$, $\tilde{\Sigma} B' = B$, and $\overline{\partial B'} = [\hat{0}_B,\hat{z}]$. Let $f_S'$ be the number of maximal chains of $B_S$ such that all elements of the chain lie in $B' \cup \{ \hat{1}_B \}$, and let 
	$\Psi(B') = \sum_{S \subset [n]} f_S' w_S$. It follows that 
	\[
	\Psi(B') = \Psi(\tilde{\Sigma} B') - \Psi(\overline{\partial B'})b.
	\]
	Also, define $\ell^{\Psi}(B') = \ell^{\Phi}(B';a + b, ab + ba)$, or, equivalently, 
	$\ell^{\Psi}(B') = \Psi(\tilde{\Sigma} B') -\Psi(\overline{\partial B'})(a + b)$.
	Then 
	\begin{equation}\label{eq:psiB'}
		\Psi(B') = \ell^{\Psi}(B') + \Psi(\overline{\partial B'})a.
	\end{equation}

	
	We now proceed to the proof.
	Our goal is to compute $\Psi(\Gamma)$ in two different ways, which we will then average to produce the result. 
	Recall that $X_{\le y}$ is near-Eulerian with boundary $X_{< y}$ and $\rank(X_{\le y}) = \rank(X_{< y}) + 1 = \rho_\Gamma(\hat{0}_\Gamma,y) - 1$, for all $y \in Y$ with $y \neq \hat{0}_Y$. Also, $X_{\le \hat{0}_Y}$ has rank equal to $\rank(\sigma)$
	 and is the boundary of an Eulerian poset $\overline{X_{\le \hat{0}_Y}}$. 
	Consider a chain $z = \{\hat{0}_\Gamma = z_0 < z_1 < \cdots < z_s < z_{s + 1} = \hat{1}_\Gamma  \}$ in $\Gamma$ for some $s \ge 0$. 	We will examine the contribution of $z$ to $\Psi(\Gamma)$. 
	Let $0 \le j \le s$ be the largest index such that $z_j \in X$. 

	On the one hand, one may index chains $z$ according to the element $y = z_{j + 1} \in Y$. 
	From this perspective,
	$z$ is a 
	concatenation 
	of a chain with initial term 
	$\hat{0}_X$ in $X_{\le y} = \{ x \in X : \sigma(x) \le y \}$,  and a chain in $[y,\hat{1}_Y]$ with initial term $y$ and final term $\hat{1}_Y$, for
	some $y \in Y$. Conversely, a concatenation of such chains gives a chain $z$ as above. Considering the cases when $y = \hat{1}_Y$, $y = \hat{0}_Y$, and $\hat{0}_Y < y < \hat{1}_Y$ separately, and then using \eqref{eq:psiB'},	we deduce that 
	\begin{align*}
		\Psi(\Gamma) &= \Psi(X) + \Psi(\overline{X_{\le \hat{0}_Y}})b\Psi(Y) +
		\sum_{\hat{0}_Y < y < \hat{1}_Y} \Psi(X_{\le y})b\Psi([y,\hat{1}_Y]) \\
		&= \ell^{\Psi}(X) + \Psi(\overline{\partial X})a +  \Psi(\overline{X_{\le \hat{0}_Y}})b\Psi(Y) +
		\sum_{\hat{0}_Y < y < \hat{1}_Y} (\ell^{\Psi}(X_{\le y}) + \Psi(\overline{X_{< y}})a)b\Psi([y,\hat{1}_Y]).
	\end{align*}
	
	
	On the other hand, one may index chains $z$ according to the element $x = z_{j} \in X$.
	From this perspective,
	$z$ is a 
	concatenation 
	of a chain  with initial term 
	$\hat{0}_X$ and final term $x$ in $X$, and a chain in $[\sigma(x),\hat{1}_Y]$ with final term $\hat{1}_Y$, for
	some $x \in X$.  Conversely, a concatenation of such chains gives a chain $z$ as above.

	Let $y = \sigma(x)$. We sum the contributions of the chains $z$ for different choices of $y$. 
	Suppose that 	$y  = \hat{1}_Y$, or, equivalently, $x \in X \smallsetminus \partial X$.  These chains give a contribution of
	$\Psi(X) - \Psi(\partial X)(a - b)$ to $\Psi(\Gamma)$.
	Suppose that 	$y  \neq \hat{1}_Y$. 
	Then chains in $[y,\hat{1}_Y]$ with final term $\hat{1}_Y$ come in pairs, where two chains in a pair differ depending on whether they contain $y$ or not. If 
	$y = \hat{0}_Y$, we get a contribution to $\Psi(\Gamma)$ of  $\Psi(\overline{X_{\le \hat{0}_Y}})a \Psi(Y)$. If $y \neq \hat{0}_Y$, we
	get a contribution to $\Psi(\Gamma)$ of 
	$(\Psi(X_{\le y}) - \Psi(\overline{X_{< y}})(a - b))a\Psi([y,\hat{1}_Y])$.  Putting this together and then using \eqref{eq:psiB'} gives
	\begin{align*}
		\Psi(\Gamma) &= \Psi(X) - \Psi(\partial X)(a - b) +
		\Psi(\overline{X_{\le \hat{0}_Y}})a \Psi(Y) + 	\sum_{\hat{0}_Y < y < \hat{1}_Y}  (\Psi(X_{\le y}) - \Psi(\overline{X_{< y}})(a - b))a\Psi([y,\hat{1}_Y])
		\\
		&= \ell^{\Psi}(X) + \Psi(\overline{\partial X})b +  \Psi(\overline{X_{\le \hat{0}_Y}})a \Psi(Y) + 	\sum_{\hat{0}_Y < y < \hat{1}_Y}  (\ell^{\Psi}(X_{\le y}) + \Psi(\overline{X_{< y}})b)a\Psi([y,\hat{1}_Y])
	\end{align*}
	
	
	Averaging our two expressions for $\Psi(\Gamma)$, and substituting $c = a + b$ and $d = ab + ba$, gives 
	\[
	\Phi(\Gamma) = \ell^{\Phi}(X) +  \frac{1}{2} \left[ 
	\Phi(\overline{\partial X})c + 
	\Phi(\overline{X_{\le \hat{0}_Y}})c \Phi(Y) + 
	\sum_{\hat{0}_Y < y < \hat{1}_Y}  (\ell^{\Phi}(X_{\le y})c + \Phi(\overline{X_{< y}})d)\Phi([y,\hat{1}_Y])
	\right].
	\]
	This is precisely \eqref{eq:cdformula}. 	      
\end{proof}

\section{Subdivisions of $CW$-posets}\label{sec:CWsubdivisions}

The prototypical example of a lower Eulerian poset coming from topology is 
a $CW$-poset, which is the face poset of a regular $CW$-complex. In this section, we develop analogues of our main results for $CW$-posets.  
In what follows, all regular $CW$-complexes are finite. 

\subsection{Background on regular $CW$-complexes and $CW$-posets}\label{ss:regularCW}

We first recall some basic facts about topology and regular $CW$-complexes. We refer the reader to \cite{HatcherAlgebraicTopology} and \cite{StanleySurveyEulerian} for details. 
For a nonnegative integer $n$, let $D^n$ denote the unit ball of dimension $n$, with boundary $\partial D^n$ an $(n - 1)$-dimensional sphere $S^{n - 1}$.
When $n = 0$, $S^{n - 1}$ is the empty set.
A \emph{closed $n$-cell}
is a topological space $D$ that is homeomorphic to $D^n$. A closed $n$-cell $D$ has a well-defined boundary $\partial D$, equal to the image of $S^{n - 1}$ under any homeomorphism of $D$ with $D^n$. We let $\Int D = D \smallsetminus \partial D$ denote the interior of $D$. An \emph{open $n$-cell} is a topological space homeomorphic to $\Int D^n$. 
If $Z$ and $Z'$ are topological spaces, then we write $Z \cong Z'$ if $Z$ and $Z'$ are homeomorphic.  
We will need the following remark. 

\begin{remark}\label{rem:Alexander}
	By Alexander's trick \cite{AlexanderDeformation}, if $D$ and $D'$ are closed $n$-cells, then any homeomorphism from $\partial D$ to $\partial D'$ extends to a homeomorphism from $D$ to $D'$.
\end{remark}

We now recall the definition of a regular $CW$-complex. 

\begin{definition}\label{def:regularCWcomplex}
	A  (finite) nonempty \emph{regular $CW$-complex} $\K$ is a Hausdorff topological space $|\K|$ together with a finite collection of 
	open cells $\{ e_\alpha \}_\alpha$ whose union is $|\K|$ such that 
	the closure $\overline{e}_\alpha$ of $e_\alpha$ in $|\K|$ is a closed cell and the boundary $\partial \overline{e}_\alpha$ is a union of some subcollection of the open cells $\{ e_\alpha \}_\alpha$, necessarily of dimension strictly less than the dimension of $e_\alpha$. 
	The dimension $\dim \K$ of $\K$ is the maximum dimension of a cell. 
	We say that $\K$ is a \emph{regular $CW$-sphere} or \emph{regular $CW$-ball} if $|\K|$ is homeomorphic to $S^{\dim \K}$ or $D^{\dim \K}$ respectively.
\end{definition}

	In what follows, it will be convenient to consider the \emph{empty cell} of $\K$ to be 
a closed cell of dimension $-1$ which is contained in every closed cell.
It  will also be convenient to  consider the empty set to be a regular $CW$-sphere of dimension $-1$ consisting of the empty cell. 

Let $\K$ be a regular $CW$-complex. A function $f: |\K| \to U$ from $|\K|$ to a topological space $U$ 
is continuous if and only if its restriction to each closed cell is continuous.
The \emph{face poset} $\face(\K)$ of $\K$ is the set of closed cells of $\K$, including the empty cell, with partial order given by inclusion. 
The poset $\face(\K)$ is lower Eulerian with natural rank function $\rho_{\face(\K)}(e_\alpha) = \dim e_\alpha + 1$. 
There is a bijection between nonempty lower order ideals $I$ of $\face(\K)$ 
and subcomplexes $\K_I = \{ e_\alpha \}_{\alpha \in I}$ of $\K$, where 
$I = \{ \hat{0}_{\face(\K)} \}$ corresponds to the empty subcomplex. The topology of $|\K|$ is determined by $\face(\K)$; the order complex $\O(\face(\K) \smallsetminus \hat{0}_{\face(\K)})$ is the barycentric subdivision of $\K$ and hence $|\K| \cong |\O(\face(\K) \smallsetminus \hat{0}_{\face(\K)})|$. We recall the definition of a $CW$-poset.

\begin{definition}\label{def:CWposet}
	A poset of the form $\face(\K)$ for some  regular $CW$-complex $\K$ is called a \emph{$CW$-poset}. 
\end{definition}

Note that with our conventions, $B_0$ is a $CW$-poset, equal to the face poset of the empty complex. 	See \cite{BjornerPosets} for an alternative characterization of $CW$-posets.

The reduced Euler characteristic of $|\K|$ is the number of even-dimensional cells in $\K$ minus the number of odd-dimensional cells in $\K$, including the empty cell. 
Suppose that $\mathcal{A}$ is a (possibly empty) subcomplex of $\K$ and $\phi: S^{n - 1} \to |\mathcal{A}|$ is a homeomorphism. Then we can form a new regular $CW$-complex $\K'$ by \emph{attaching an $n$-cell} via $\phi$. That is, 
$|\K'|$ is the disjoint union of $D^n$ and $|\K|$ modulo the relation $\{  z \sim \phi(z) : z \in S^{n - 1} \}$ and open cells of $\K'$ are the images of the  open cells of $\K$ together with the image of the interior of $D^n$.

\begin{example}\label{ex:CWpolytope}
	Let $P$ be a polytope. Then we may view $P$ as a regular $CW$-ball with closed cells equal to the faces of $P$. The boundary $\partial P$ of $P$ is a regular $CW$-sphere. More generally, we may view a polyhedral subdivision $\cS$ of $P$ as a a regular $CW$-ball with closed cells equal to the elements  of $\cS$.
\end{example}

\begin{example}\label{ex:CWfan}
	Let $\Sigma$ be a fan in a real vector space $V$. Let $S \subset V$ be a sphere centered at the origin. Let $\Sigma \cap S$ denote the regular $CW$-complex with closed cells $\{ C \cap S : C \in \Sigma \}$. 
	Then $\face(\Sigma \cap S) = \face(\Sigma)$. 
	Suppose that $|\Sigma|$ is a full-dimensional (convex) cone in $V$ with boundary $\partial |\Sigma|$. We may consider the subfan $\partial \Sigma = \{ C \in \Sigma : C \subset \partial |\Sigma| \}$. 
	Then $\Sigma \cap S$ has dimension $\dim V - 1$. Moreover,  
	$\Sigma \cap S$ is a regular $CW$-sphere if $|\Sigma| = V$, and 
	$\Sigma \cap S$ is a regular $CW$-ball with boundary 	$\partial \Sigma \cap S$ otherwise.
\end{example}

\begin{example}\label{ex:CWEulerian}
	Let $B$ be a $CW$-poset. 
	Assume that $B$ is Eulerian of positive rank. Let $\K$ be a regular $CW$-complex such that $\face(\K) = B$. 
	Since $|\K|$ is the closed cell corresponding to the unique maximal element $\hat{1}_B$ in $B$, we deduce that $\K$ is a regular $CW$-ball. Since the boundary of $|\K|$ is a sphere, we deduce that $B$ is Gorenstein*  (see Section~\ref{ss:posets}). 
	Conversely, if $\K$ is a nonempty regular $CW$-complex with a unique  maximal closed cell, then $\face(\K)$ is Gorenstein*  (and hence Eulerian of positive rank). 
	For example, the face lattice of a polytope is Gorenstein*  (see Example~\ref{ex:CWpolytope}). 
\end{example}

\begin{example}\label{ex:CWsphere}
	Suppose that $\K$ is a regular $CW$-sphere	of dimension $n - 1$. 
	After choosing a homeomorphism $\phi: S^{n - 1} \to |\K|$, we may attach an $n$-cell via $\phi$ to obtain a regular $CW$-ball $\K'$ of dimension $n$ with a unique maximal closed cell. 
	By Example~\ref{ex:CWEulerian},  $\face(\K')$ is Gorenstein*, and 
	we deduce that 
	$\face(\K)$ is the boundary of a Gorenstein* poset  (and hence the boundary of an Eulerian poset of positive rank).  For example, the face poset of a complete fan is the boundary of a Gorenstein* poset  (see Example~\ref{ex:CWfan}). 
\end{example}

\begin{example}\label{ex:CWball}
	Suppose that $\K$ is a regular $CW$-ball of dimension $n$. 	
	The boundary subcomplex $\partial \K$ is the union of all closed cells  of dimension $n - 1$ which are contained in a unique cell of dimension $n$. 
	Then $\partial \K$ is a regular $CW$-sphere of dimension $n - 1$ and $|\partial \K| = \partial |\K|$. After choosing a homeomorphism $\phi: S^{n - 1} \to |\partial \K|$, we may attach an $n$-cell via $\phi$ to obtain a regular $CW$-sphere $\K'$. By Example~\ref{ex:CWsphere}, we may then attach an $(n + 1)$-dimensional cell to obtain a regular $CW$-ball $\K''$ such that $\face(\K'')$ is Gorenstein*. 
	We deduce that $\face(\K)$ is near-Gorenstein*  (and hence near-Eulerian) with semisuspension  $\tilde{\Sigma} \face(\K) = \face(\K'')$ and boundary $\partial ( \face(\K)) = \face(\partial K)$. 
	For example, the face poset of a polyhedral subdivision of a polytope is near-Gorenstein*  (see Example~\ref{ex:CWpolytope}). 

\end{example}

Let $\K$ and $\K'$ be  nonempty regular $CW$-complexes with open cells $\{ e_\alpha \}_{\alpha \in \face(\K)}$ and $\{ e_{\alpha'} \}_{\alpha' \in \face(\K')}$ respectively. 
Recall that the topological join $|\K| \star |\K'|$ is the 
quotient of $|\K| \times |\K'| \times [0,1]$ by the relations $(z,z_1',0) \sim (z,z_2',0)$ and   $(z_1,z',1) \sim (z_2,z',1)$ for all $z,z_1,z_2 \in |\K|$ and $z',z_1',z_2' \in |\K'|$. 
Given nonempty cells $e_\alpha$ and 
$e_{\alpha'}$ in $\K$ and $\K'$ respectively, let $e_{\alpha, \alpha'}$ be the image of $e_\alpha \times e_{\alpha'} \times [0,1]$ in $|\K| \star |\K'|$. If $e_\alpha$ is nonempty and $e_{\alpha'}$ is the empty cell, then let $e_{\alpha, \alpha'}$ be the image of $e_\alpha$ in $|\K| \times |\K'| \times \{ 0 \}$. If $e_\alpha$ is the empty cell and $e_{\alpha'}$ is nonempty, then let $e_{\alpha, \alpha'}$ be the image of $e_{\alpha'}$ in $|\K| \times |\K'| \times \{ 1 \}$. If  $e_\alpha$ and $e_{\alpha'}$ are the empty cells, then let $e_{\alpha, \alpha'}$ be the empty cell. 
Then the \emph{join} $\K \star \K'$ is the regular $CW$-complex with open cells $\{ e_{\alpha, \alpha'} : (\alpha, \alpha') \in \face(\K) \times \face(\K') \}$, $|\K \star \K'| = |\K| \star |\K'|$, and $\face(\K \star \K') = \face(\K) \times \face(\K')$. 
We may extend the definition when $\K'$ or $\K$ is empty. Explicitly, 
if $\K'$ is empty, then define $e_{\alpha, \alpha'} = e_\alpha$ so that $\K \star \K'$ is identified with $\K$. Similarly, if  $\K$ is empty, then define $e_{\alpha, \alpha'} = e_{\alpha'}$ so that $\K \star \K'$ is identified with $\K'$.
%
Given subcomplexes $\mathcal{A}$ and $\mathcal{A}'$ of $\K$ and $\K'$ respectively, we may view $\mathcal{A} \star \mathcal{A}'$ as a subcomplex of $\K \star \K'$. In particular, 
we may view $\K$ and $\K'$ as the subcomplexes $\K \star \emptyset = \{ e_{\alpha, \hat{0}_{\face(\K')}} : \alpha \in \face(\K) \}$ and   $\emptyset \star \K' = \{ e_{\hat{0}_{\face(\K)},\alpha'} : \alpha' \in \face(\K') \}$ of $\K \star \K'$ respectively. 
We write $e_\alpha \star |\K'|$ to denote the union of $\{ e_{ \alpha, \alpha' } : \alpha' \in \K' \}$.

\begin{example}\label{ex:joinpoints}
	If $\K$ is a regular $CW$-complex, then the join of $\K$ with a point has face poset $\Pyr(\face(\K))$. The successive join of $n + 1$ points is the $n$-dimensional simplex $\Delta^n$,  a regular $CW$-ball with $\face(\Delta^n) = B_{n + 1}$.
	
\end{example}

\begin{example}\label{ex:joinballs}
	With the notation of Example~\ref{ex:joinpoints}, consider simplices $\Delta^n$ and $\Delta^{n'}$. 
	Since the topological join operation is associative, 
	$|\Delta^n \star \Delta^{n'}|$ is homeomorphic to $|\Delta^{n + n' + 1}|$. Moreover, the $(n + n')$-dimensional sphere $|\partial  \Delta^{n + n' + 1}|$ is homeomorphic to the union of $(n + n')$-dimensional balls $|\partial \Delta^n \star \Delta^{n'}|$ and $|\Delta^n \star \partial \Delta^{n'}|$ which intersect along an $(n + n' - 1)$-dimensional sphere $|\partial \Delta^n \star  \partial \Delta^{n'}|$. 
	
	If $\K$ is a regular $CW$-ball of dimension $n$, then there is a homeomorphism from  $|\K|$ to  $\Delta^n$ that takes $|\partial K|$ to $|\partial \Delta^n|$. 
	Then $\K \star \Delta^{n'}$ is a regular $CW$-ball of dimension $n + n' + 1$, with boundary the union of the regular $CW$-balls $\partial \K \star \Delta^{n'}$ and $\K \star \partial \Delta^{n'}$ of dimension  $n + n'$, which have common boundary $\partial K \star \partial \Delta^{n'}$. 
\end{example}

\subsection{Strong $CW$-regular subdivisions}

In this section, we define the notion of a strong $CW$-regular subdivision of $CW$-posets and state Theorem~\ref{thm:mainsimplifiedCW} and Theorem~\ref{thm:mainCW} that are analogous to Theorem~\ref{thm:mainsimplified} and Theorem~\ref{thm:main} respectively.

We will use the following notation throughout. 
Let $\sigma: X \to Y$ be an order-preserving, surjective function between $CW$-posets. Let $\K_X$ be a regular $CW$-complex with face poset $X$. 
For any $y \in Y$, recall that 
 $X_{\le y} =  \sigma^{-1}([\hat{0}_Y,y])$ and 
$X_{< y} = \sigma^{-1}([\hat{0}_Y,y))$. We write $\K_{X,\le y}$ and $\K_{X,< y}$ for the  subcomplexes of $\K_X$ corresponding to the lower order ideals $X_{\le y}$ and $X_{< y}$ respectively.

When $\rank(X) = \rank(Y)$, the following definition 
coincides with the notion of a $CW$-regular subdivision in \cite{Stanley92}.

\begin{definition}\label{def:CWsubdivision}
	Let  $\sigma: X \to Y$ be an order-preserving, rank-increasing, surjective function between $CW$-posets $X$ and $Y$  with rank functions $\rho_X$ and $\rho_Y$ respectively. 
	Let 	$\K_X$  be a 
	regular $CW$-poset with  face poset $X$. 
	Then $\sigma$ is a  \emph{strong $CW$-regular subdivision} 
	if for any $y$ in $Y$, 
	$\dim \K_{X,\le y} = \rho_Y(y) - \rho_X(\hat{0}_X) - 1$, and
	either
	\begin{enumerate}
		
		\item $y = \hat{0}_Y$ and $\K_{X,\le y}$ is a regular $CW$-sphere, or, 
		\item $y \neq \hat{0}_Y$ and $\K_{X,\le y}$ is a regular $CW$-ball with boundary $\K_{X,< y}$. 
		%
	\end{enumerate}
\end{definition}

In Definition~\ref{def:CWsubdivision}, since the topology of $|\K_{X,\le y}|$ and $|\K_{X,< y}|$ only depends on the posets 
$X_{\le y}$ and  $X_{< y}$ respectively, the definition is independent of the choice of $\K_X$.

\begin{remark}\label{rem:CWregularisformal}
	Let  $\sigma: X \to Y$ be a strong $CW$-regular subdivision between $CW$-posets $X$ and $Y$  with rank functions $\rho_X$ and $\rho_Y$ respectively. 
	%
	It follows from Definition~\ref{def:CWsubdivision},  Example~\ref{ex:CWsphere}, and Example~\ref{ex:CWball} that 
	for all $y$ in $Y$, $X_{\le y}$ is a lower Gorenstein* poset of rank $\rho_Y(y) - \rho_X(\hat{0}_X)$, and either 
	\begin{enumerate}
		\item $y = \hat{0}_Y$ and $X_{\le y}$ is the boundary of
		a Gorenstein* poset, or,
		\item\label{i:nearGorenstein} $y \neq \hat{0}_Y$ and $X_{\le y}$ is near-Gorenstein* with boundary $X_{< y}$.
	\end{enumerate}
	Then $\sigma$ is a strong formal subdivision  by Remark~\ref{rem:altcharnearEulerian}. 
	
\end{remark}

We define the rank of a strong $CW$-regular subdivision to be its rank as a strong formal subdivision, i.e., with the notation of Definition~\ref{def:CWsubdivision},  $\rank(\sigma) = \rank(X) - \rank(Y) = \rho_Y(\hat{0}_Y) - \rho_X(\hat{0}_X)$.

\begin{remark}\label{rem:restrictsfsCW}
	By Remark~\ref{rem:restrictsfs} and Definition~\ref{def:CWsubdivision},
	if $\sigma: X \to Y$ is a strong $CW$-regular subdivision 
	and $I$ is a  nonempty  lower order ideal of $Y$, then $\sigma$ restricts to a strong $CW$-regular subdivision $\sigma^{-1}(I) \to I$.
\end{remark}

\begin{remark}\label{rem:CWalt}
	Using Example~\ref{ex:joinballs}, we have the following alternative definition of a strong  $CW$-regular subdivision. 	Let  $\sigma: X \to Y$ be an order-preserving, rank-increasing, surjective function between $CW$-posets $X$ and $Y$  with rank functions $\rho_X$ and $\rho_Y$ respectively. 
	Let 	$\K_X$  be a 
	regular $CW$-complex with  face poset $X$, and let $r = \rho_Y(\hat{0}_Y) - \rho_X(\hat{0}_X)$.	Then $\sigma$ is a  \emph{strong $CW$-regular subdivision} 
	if for any $y \in Y$, there exists a homeomorphism $\psi: |\K_{X,\le y}| \to \overline{e_y} \star |\partial \Delta^r|$ that restricts to a homeomorphism 
	$|\K_{X,< y}| \to \partial \overline{e_y} \star |\partial \Delta^r|$ when $y \neq \hat{0}_Y$. 
\end{remark}

The following lemma will be proved in Section~\ref{ss:CWproofs} (c.f. Lemma~\ref{lem:composesfs}). 

\begin{lemma}\label{lem:composeCW}
	Let  $\sigma : X \to Y$ and $\tau: Y \to Z$ be strong formal subdivisions between $CW$-posets $X$, $Y$, $Z$ with   rank functions $\rho_X$, $\rho_Y$, $\rho_Z$  respectively (and hence 	$\tau \circ \sigma$ is a strong formal subdivision by Lemma~\ref{lem:composesfs}). Assume that $\sigma$ is a strong $CW$-regular subdivision. If 
	 $\tau$ is a 
	strong $CW$-regular subdivision,  then $\tau \circ \sigma$ is a strong $CW$-regular subdivision.  Conversely, assume that 
	$\tau \circ \sigma$ is a strong $CW$-regular subdivision. Additionally assume that $\sigma$ has rank $0$. Then $\tau$ is a strong $CW$-regular subdivision. 
\end{lemma}

We next consider some examples of strong $CW$-regular subdivisions.
The first example is analogous to Example~\ref{ex:B0}. 

\begin{example}\label{ex:CWB0}
	Let $B$ be an Eulerian poset of positive rank. 
	Recall from Example~\ref{ex:B0} that, with the appropriate rank functions, the unique function $\sigma: \partial B \to B_0$ is a strong formal subdivision, and all strong formal subdivisions to $B_0$ appear in this way. Moreover, $B$ is the non-Hausdorff mapping cylinder  $\Cyl(\sigma)$. 
	
	Assume further that $B$ is a $CW$-poset. Then 
	it follows from Example~\ref{ex:CWEulerian} 
	that $\sigma: \partial B \to B_0$ is a strong $CW$-regular subdivision. In fact, all strong $CW$-regular subdivisions to $B_0$ appear in this way. 
	Indeed, by Example~\ref{ex:B0}, a strong $CW$-regular subdivision to $B_0$ has the form $\sigma: \partial B \to B_0$ for some Eulerian poset $B$ with positive rank. Definition~\ref{def:CWsubdivision} implies that $\partial B = \face(\K)$ for some regular $CW$-sphere $\K$, and by Example~\ref{ex:CWsphere}, $B$ is a  $CW$-poset. 
	
\end{example}

Below we give an example of a strong formal subdivision between $CW$-posets that is not a strong $CW$-regular subdivision. 

\begin{example}\label{ex:CWB0counter}
	
	In contrast to Example~\ref{ex:CWB0}, let $\K$ be a $CW$-regular complex of dimension $r - 1$ such that $|\K|$ is a homology sphere over $\Q$ but $|\K|$ is not homeomorphic to a sphere. Then $\face(\K)$ is the boundary of a Gorenstein* poset. By Example~\ref{ex:B0},  with the appropriate rank function, 
	the unique function $\sigma: \face(\K) \to B_0$ is a strong formal subdivision of rank $r$ that is not a strong $CW$-regular subdivision. 
\end{example}

The next example is analogous to Example~\ref{ex:B1}. 

\begin{example}\label{ex:CWB1}
	Let $B$ be a near-Eulerian poset. 
	Recall from Example~\ref{ex:B1} that, with the appropriate rank functions, the function  
	$\sigma: B \to B_1 = \{ \hat{0}, \hat{1} \}$ determined by $\sigma^{-1}(\hat{0}) = \partial B$ is  a strong formal subdivision, 
	and all strong formal subdivisions to $B_1$ appear in this way.  
	Moreover, the semisuspension $\tilde{\Sigma} B$ is the non-Hausdorff mapping cylinder  $\Cyl(\sigma)$. 
	
	Assume further that $B = \face(\K)$ for some regular $CW$-ball. 
	Recall from Example~\ref{ex:CWball} that $\partial \K$  is a
	$CW$-regular sphere. It follows that 
	$\sigma$ is a strong $CW$-regular subdivision. 
	In fact, all strong $CW$-regular subdivisions to $B_1$ appear in this way. 
	Indeed,  by Example~\ref{ex:B1}, a strong $CW$-regular subdivision to $B_1$ has the form $\sigma: B \to B_1$ for some near-Eulerian poset $B$ with $\sigma^{-1}(\hat{0}) = \partial B$. Definition~\ref{def:CWsubdivision} implies that $B = \face(\K)$ for some regular $CW$-ball $\K$.
	Moreover, in this case, by Example~\ref{ex:CWball}, $\tilde{\Sigma} B$  is a $CW$-poset. 
	
	
\end{example}

\begin{example}\label{ex:CWidentity}
	Let $B$ be a $CW$-poset with rank function $\rho_B$. 
	Then the identity function $\id_B: B \to B$ is a strong $CW$-regular subdivision (c.f. Example~\ref{ex:identitypre}). 
\end{example}

\begin{lemma}\label{lem:CWproper}
	Consider a linear map $\phi: V' \to V$  inducing a proper, surjective morphism between fans $\Sigma'$ and $\Sigma$ in real vector spaces $V'$ and $V$ respectively.  
	Recall from Example~\ref{ex:proper} that we may consider the induced function between the corresponding  face posets
	$\sigma: \face(\Sigma') \to \face(\Sigma),$
	where $\sigma(C')$ is the smallest element of $\Sigma$ containing $\phi(C')$. Let $r = \dim \ker(\phi)$. Consider $\face(\Sigma')$ and $\face(\Sigma)$ equipped with  the natural rank function and the natural rank function shifted by $r$ respectively.
	Then $\sigma$ is a strong $CW$-regular subdivision of rank $r$. 
\end{lemma}
\begin{proof}
	Let $X = \face(\Sigma')$ and $Y = \face(\Sigma)$. Using Example~\ref{ex:CWfan}, let $\K_X$ and $\K_Y$ be the regular $CW$-complexes obtained by intersecting $\Sigma'$ and $\Sigma$ respectively with spheres centered at the origin in $V'$ and $V$ respectively. 
	We first verify that $\sigma$ is order-preserving; if $C_1' \le C_2'$ in $X$, then $\phi(C_1') \subset \phi(C_2') \subset V$, and hence $\sigma(C_1') \le \sigma(C_2')$ in $Y$. We next verify that 
	$\sigma$ is rank-increasing; if $C'$ in $X$, then 
	$C' \subset \phi^{-1}(\sigma(C'))$ implies that $\rho_X(C') = \dim C' \le \dim \phi^{-1}(\sigma(C')) = \dim \sigma(C') + r = \rho_Y(\sigma(C'))$. 
	
	Fix a pointed cone $C$ in $\Sigma$. Let $V_C$ be the linear span of $C$ in $V$, and let $V_C' = \phi^{-1}(V_C)$. Consider the restriction $\Sigma'|_{\phi^{-1}(C)}$ of $\Sigma'$ to $\phi^{-1}(C)$. Then $\phi$ restricts to a surjection $V_C' \to V_C$ which induces a proper, surjective morphism of fans from $\Sigma'|_{\phi^{-1}(C)}$ to $C$. 
	In particular, if $C'$ is a maximal cone of $\Sigma'|_{\phi^{-1}(C)}$, then $\sigma(C') = C$, and we deduce that $\sigma$ is surjective.
	
	Consider the subcomplex $\K_{X,\le C}$ with $\face(\K_{X,\le C}) = \face(\Sigma'|_{\phi^{-1}(C)})$. 
	By Example~\ref{ex:CWfan}, $\K_{X,\le C}$ has dimension $\dim C + r - 1 = \rho_{Y}(C) - 1$, and either
	\begin{enumerate}
		
		\item 	$\phi^{-1}(C) = V_C'$ and $\K_{X,\le C}$ is a regular $CW$-sphere, or, 
		\item 	$\phi^{-1}(C) \neq V_C'$  and $\K_{X,\le C}$ is a regular $CW$-ball with  $\face(\partial \K_{X,\le C}) = \face(\partial \Sigma'|_{\phi^{-1}(C)})$. 
	\end{enumerate}
	Since $\phi$ is surjective and $C$ is a pointed cone, we have 
	$\phi^{-1}(C) = V_C'$ if and only if $C = V_C$ if and only if $C = \{ 0_V \}$, where $0_V$ denotes the origin in $V$. 
	Moreover, if $C \neq \{ 0_V \}$, then $\partial |\phi^{-1}(C)| = \phi^{-1}(\partial |C|)$, and it follows that
	the subcomplex $\K_{X,< C}$ is determined by $\face(\K_{X,< C}) = \face(\partial \Sigma'|_{\phi^{-1}(C)})$. 
	The result now follows from Definition~\ref{def:CWsubdivision}. 	
\end{proof}

\begin{example}\label{ex:CWsubdivision}
	Recall from Example~\ref{ex:polytope} that  given 
	polyhedral subdivisions $\cS'$ and $\cS$ of a polytope $P$ such that $\cS'$ is a refinement of $\cS$, there is an induced strong formal subdivision of rank $0$ between the corresponding  face posets
	$\sigma: \face(\cS') \to \face(\cS),$
	where $\sigma(F')$ is the smallest element of $\cS$ containing $F'$, and  
	$\face(\cS')$ and $\face(\cS)$  are equipped with the natural rank functions. Recall from Section~\ref{ss:fans} that there is a corresponding refinement of fans 
	that induces $\sigma$. Then Lemma~\ref{lem:CWproper} implies that $\sigma$ is a strong $CW$-regular subdivision of rank $0$. This also follows from the proof of 
	\cite{KatzStapledon16}*{Lemma~3.25}. 
\end{example}

Recall from Definition~\ref{def:bijections} that 
$\SFS$  is the set of strong formal subdivisions $\sigma: X \to Y$ between  lower Eulerian posets $X$ and $Y$  with rank functions $\rho_X$ and $\rho_Y$ respectively, and 
$\JoinIdealLW$ is the set of triples $(\Gamma, \rho_\Gamma, q)$,
where $\Gamma$ is a  lower Eulerian poset  with rank function $\rho_\Gamma$,  and $q$ is a join-admissible element of $\Gamma$.
Recall that $\JoinIdealLW^\circ \subset \JoinIdealLW$ is the subset of all triples $(\Gamma, \rho_\Gamma, q)$ such that $q \neq \hat{0}_\Gamma$. 	
Using Remark~\ref{rem:CWregularisformal}, we define the following subsets.

\begin{definition}\label{def:CWbijections}
	Let $\SFSCW \subset \SFS$  be the set of strong $CW$-regular subdivisions $\sigma: X \to Y$ between  $CW$-posets $X$ and $Y$  with rank functions $\rho_X$ and $\rho_Y$ respectively. 
	Let $\JoinIdealCW \subset \JoinIdealLW$ be the set of triples $(\Gamma, \rho_\Gamma, q)$,
	where $\Gamma$ is a  $CW$-poset  with rank function $\rho_\Gamma$,  and $q$ is a join-admissible element of $\Gamma$ such that $\Gamma_{\ge q}$ is a $CW$-poset.
		Let $\JoinIdealCW^\circ \subset \JoinIdealCW$ be the subset of all triples $(\Gamma, \rho_\Gamma, q)$ such that $q \neq \hat{0}_\Gamma$. 	
\end{definition}

We are now ready to state our analogue of Theorem~\ref{thm:mainsimplified} for $CW$-posets. Recall the functions $\CYL: \SFS \to  \JoinIdealLW^\circ$  and $\MAP:  \JoinIdealLW^\circ \to \SFS$ from Definition~\ref{def:bijections}. Recall that  Theorem~\ref{thm:mainsimplified} states that $\CYL$ and $\MAP$ are 
mutually inverse bijections.

\begin{theorem}\label{thm:mainsimplifiedCW}
	The functions $\CYL: \SFS \to  \JoinIdealLW^\circ$  and $\MAP:  \JoinIdealLW^\circ \to \SFS$ restrict to mutually inverse bijections between $\SFSCW$ and $\JoinIdealCW^\circ$.
\end{theorem}

The proof will be given in Section~\ref{ss:CWproofs}. 
Before stating the corresponding theorem analogous to Theorem~\ref{thm:main}, we consider some examples of Theorem~\ref{thm:mainsimplifiedCW}. 
Firstly, see Example~\ref{ex:CWB0} and Example~\ref{ex:CWB1}  above.

\begin{example}\label{ex:identityCWpyramid}
	Let $B$ be a $CW$-poset with rank function $\rho_B$. Suppose that $B = \face(\K)$ for some regular $CW$-complex $\K$. Then $\Pyr(B) = \face(\K \star \Delta^0)$ is the non-Hausdorff mapping cylinder of the identity function $\id_B: B \to B$ (c.f. Example~\ref{ex:identity}, Example~\ref{ex:CWidentity}). 
	
%
%
\end{example}

\begin{example}
	The topological join construction on regular $CW$-complexes implies that the analogue of \eqref{i:directproduct} in Lemma~\ref{lem:constructions} holds. That is, if $(\Gamma,\rho_{\Gamma},q),(\Gamma',\rho_{\Gamma'},q') \in \JoinIdealCW$ and $(q,q') \neq (\hat{0}_\Gamma,\hat{0}_{\Gamma'})$, then $(\Gamma \times \Gamma',\rho_{\Gamma \times \Gamma'},(q,q')) \in \JoinIdealLW^\circ$. Together with Lemma~\ref{lem:composeCW}, the analogues of Example~\ref{ex:productidentity} and Example~\ref{ex:productsubdivisions} hold. 
In particular, given strong $CW$-regular subdivisions $\sigma: X \to Y$  and $\sigma': X' \to Y'$, the product $\sigma \times \sigma': X \times X' \to Y \times Y'$ is a strong $CW$-regular subdivision.
\end{example}


\begin{example}\label{ex:CWpropermappingcylinder}
	Consider a linear map $\phi: V' \to V$  inducing a proper, surjective morphism between fans $\Sigma'$ and $\Sigma$ in real vector spaces $V'$ and $V$ respectively.  
    Let $r = \dim \ker(\phi)$. 
	Recall from Example~\ref{ex:proper} and Lemma~\ref{lem:CWproper} that we may consider the induced strong $CW$-regular subdivision of rank $r$ between the corresponding  face posets
	$\sigma: \face(\Sigma') \to \face(\Sigma),$
	where $\sigma(C')$ is the smallest element of $\Sigma$ containing $\phi(C')$, and  $\face(\Sigma')$ and $\face(\Sigma)$ are equipped with  the natural rank function and the natural rank function shifted by $r$ respectively.
	
	Let $\K$ be the regular $CW$-complex with closed cells obtained by 
	intersecting the cones $\{ C' \times \{ 0\} : C' \in \Sigma' \}$ and 
	$\{ \phi^{-1}(C) \times \R_{\ge 0} : C \in \Sigma \}$ with a sphere $S$ centered at the origin in $V'$. 
	If $\CYL(\sigma) = (\Gamma,\rho_\Gamma,q)$, then 
	$\Gamma = \face(\K)$ is the non-Hausdorff mapping cylinder $\Cyl(\sigma)$ of $\sigma$, 
	$\rho_\Gamma$ is the natural rank function, and $q = (\ker(\phi) \times \R_{\ge 0}) \cap S \in \face(\K)$. 
	
	Additionally assume that the morphism of fans is projective. That is, 
	there exists
	$p: |\Sigma'| \to \R$ that is 	piecewise linear  with respect to $\Sigma'$, and  
	such that the restriction of $p$ to $\phi^{-1}(C)$ is strictly convex with respect to the restriction $\Sigma'|_{\phi^{-1}(C)}$ for every maximal cone $C$ in $\Sigma$. For each cone $C' \in \Sigma'$, let 
	$\tilde{C}' = \{ (v,p(v)) : v \in C' \} \subset V' \oplus \R$. 
	For each cone $C \in \Sigma$, let 
	$\tilde{C} = \{ (v,\lambda) : v \in \phi^{-1}(C), p(v) \le \lambda \} \subset V' \oplus \R$. Then the union of the cones  
	 $\{ \tilde{C}' : C' \in \Sigma' \}$ and $\{ \tilde{C} : C \in \Sigma \}$ forms a fan $\tilde{\Sigma}$ with  $\face(\tilde{\Sigma}) = \Cyl(\sigma)$. 
\end{example}

\begin{example}\label{ex:CWsubdivisionmappingcylinder}
	Recall from Example~\ref{ex:polytope} and Example~\ref{ex:CWsubdivision} that  given 
	polyhedral subdivisions $\cS'$ and $\cS$ of a polytope $P$ such that $\cS'$ is a refinement of $\cS$, there is an induced strong $CW$-regular subdivision of rank $0$ between the corresponding  face posets
	$\sigma: \face(\cS') \to \face(\cS),$
	where $\sigma(F')$ is the smallest element of $\cS$ containing $F'$, and  
	$\face(\cS')$ and $\face(\cS)$  are equipped with the natural rank functions.
	
	Recall from Section~\ref{ss:polytopes} that if $P$ is contained in a real vector space $V$, then the  pyramid $\Pyr(P)$  of $P$ is the convex hull of $P \times \{ 0 \}$ and $v = (0_V,1)$, where $0_V$ denotes the origin in $V$. For any face $F$ of $\cS$, let $\Conv{F \times \{ 0 \},v}$ denote the convex hull of $F \times \{ 0 \}$ and the apex $v$. For example, when $F$ is the empty face, $\Conv{F \times \{ 0 \},v} = v$. 
	Let $\K$ be the regular $CW$-complex with $|\K| = \Pyr(P)$, and closed cells consisting of $\{ F' \times \{ 0 \} \in F' \in \cS' \}$ and 
	$\{ \Conv{F \times \{ 0 \},v} : F \in \cS \}$. 
	If $\CYL(\sigma) = (\Gamma,\rho_\Gamma,q)$, then 
	$\Gamma = \face(\K)$ is the non-Hausdorff mapping cylinder of $\sigma$, 
	$\rho_\Gamma$ is the natural rank function, and $q = v \in \face(\K)$. 
\end{example}

Recall from Definition~\ref{def:categorify} that $\EulPos$ is the category with elements given by  pairs $(B,\rho_B)$, where $B$ is a lower Eulerian poset and $\rho_B$ is a rank function for $B$, and with
morphisms given by strong formal subdivisions, and $\SFScat$ is the corresponding arrow category $\Arr(\EulPos)$. 
Recall that $\Joincat$ is the category with objects given by elements of $\JoinIdealLW$. 
A morphism  between objects $(\Gamma, \rho_\Gamma, q)$ and $(\Gamma', \rho_{\Gamma'}, q')$ in  $\Joincat$ is given by
a strong formal subdivision 
$\phi: \Gamma \to \Gamma'$ such that $\phi(q) = q'$ 
and
$\phi(z \vee q) = \phi(z) \vee q'$ for all $z \in \Gamma \smallsetminus \Gamma_{\ge q}$.
Recall that $\Joincat^\circ$ is the full subcategory of $\Joincat$ with objects 
$\JoinIdealLW^\circ$. 

We now define categories $\SFSCWcat$ and $\JoinCWcat$ whose objects are the sets $\SFSCW$  and $\JoinIdealCW$ respectively from Definition~\ref{def:CWbijections}. Using Remark~\ref{rem:CWregularisformal}, we may view these as subcategories of $\SFScat$ and $\Joincat$ respectively.

\begin{definition}\label{def:CWcategorify}
	Let $\CWPos$ be the subcategory of $\EulPos$ with elements given by  pairs $(B,\rho_B)$, where $B$ is a $CW$-poset and $\rho_B$ is a rank function for $B$, and with
	morphisms given by strong $CW$-regular subdivisions. We define $\SFSCWcat$ to be the corresponding arrow category $\Arr(\CWPos)$. 
	
	Let $\JoinCWcat$ be the category with objects 
	$\JoinIdealCW$. 
	A morphism  between objects $(\Gamma, \rho_\Gamma, q)$ and $(\Gamma', \rho_{\Gamma'}, q')$ in  $\JoinCWcat$ is given by
	a strong $CW$-regular subdivision 
	$\phi: \Gamma \to \Gamma'$ such that $\phi(q) = q'$,
	$\phi(z \vee q) = \phi(z) \vee q'$ for all $z \in \Gamma \smallsetminus \Gamma_{\ge q}$, and $\phi$ restricts to a strong $CW$-regular subdivision $\Gamma_{\ge q} \to \Gamma_{\ge q'}'$. 	Let $\JoinCWcat^\circ$ be the full subcategory of $\JoinCWcat$ with objects 
	$\JoinIdealCW^\circ$. 

\end{definition}

In Definition~\ref{def:CWcategorify},  $\CWPos$ is well-defined by 
 Lemma~\ref{lem:composeCW} and Example~\ref{ex:CWidentity}. By Example~\ref{ex:CWB0counter}, $\CWPos$  is not the full subcategory of $\EulPos$ with objects given by pairs $(B,\rho_B)$, where $B$ is a $CW$-poset. 

We are now ready to state our analogue of Theorem~\ref{thm:main} for $CW$-posets. Recall the functors $\CYLcat: \SFScat \to \Joincat^\circ$ and $\MAPcat: \Joincat^\circ \to \SFScat$ from Definition~\ref{def:functorialbijections}. Recall that  Theorem~\ref{thm:main} states that $\CYLcat$ and $\MAPcat$ are 
mutually inverse isomorphisms.

\begin{theorem}\label{thm:mainCW}
	The functors $\CYLcat: \SFScat \to  \Joincat^\circ$  and $\MAPcat:  \Joincat^\circ \to \SFScat$ restrict to mutually inverse isomorphisms between $\SFSCWcat$ and $\JoinCWcat^\circ$. 
\end{theorem}

The proof will be given in Section~\ref{ss:CWproofs}.

\subsection{Proof of results on $CW$-posets}\label{ss:CWproofs}
In this section we give the proofs of Lemma~\ref{lem:composeCW}, Theorem~\ref{thm:mainsimplifiedCW}, and Theorem~\ref{thm:mainCW}. 

\begin{proof}[Proof of Theorem~\ref{thm:mainsimplifiedCW}]
	As we explain below, the proof follows immediately from Theorem~\ref{thm:mainsimplified}, together with   Lemma~\ref{lem:CWmappingcylinder} and Lemma~\ref{lem:CWconverse} that we will prove later in the section. 
	
	Recall from Theorem~\ref{thm:mainsimplified} that 
	the functions $\CYL: \SFS \to  \JoinIdealLW^\circ$  and $\MAP:  \JoinIdealLW^\circ \to \SFS$ are mutually inverse bijections.
	We need to show that the subsets $\SFSCW \subset \SFS$ and 
	$\JoinIdealCW^\circ \subset \JoinIdealLW^\circ$ satisfy
	$\CYL(\SFSCW) \subset \JoinIdealCW^\circ$ and $\MAP(\JoinIdealCW^\circ) \subset \SFSCW$. 
	
	We first show that $\CYL(\SFSCW) \subset \JoinIdealCW^\circ$. Consider a  strong $CW$-regular subdivision $\sigma: X \to Y$ between  $CW$-posets $X$ and $Y$  with rank functions $\rho_X$ and $\rho_Y$ respectively. Recall that 
	$\CYL(\sigma: X \to Y) = (\Gamma, \rho_\Gamma, q)$, where 
	$\Gamma = \Cyl(\sigma)$ and $q = \hat{0}_Y$. Then $\Gamma_{\ge q} = Y$ is a $CW$-poset, and $\Gamma$ is a $CW$-poset by Lemma~\ref{lem:CWmappingcylinder} below. By Definition~\ref{def:CWbijections}, $\CYL(\sigma: X \to Y)  \in 
	\JoinIdealCW^\circ$. 
	
	We next show that $\MAP(\JoinIdealCW^\circ) \subset \SFSCW$. Consider a triple $(\Gamma, \rho_\Gamma, q)$ in $\JoinIdealCW^\circ$, where 
	$\Gamma$ is a  $CW$-poset  with rank function $\rho_\Gamma$,  and $q \neq \hat{0}_\Gamma$ is a join-admissible element of $\Gamma$ such that $\Gamma_{\ge q}$ is a $CW$-poset.	
	Consider the corresponding strong formal subdivision
	$\MAP(\Gamma,\rho_\Gamma,q) : \Gamma \smallsetminus \Gamma_{\ge q} \to \Gamma_{\ge q}$. This is a strong $CW$-regular subdivision by Lemma~\ref{lem:CWconverse}. 
	By Definition~\ref{def:CWbijections}, $\MAP(\Gamma,\rho_\Gamma,q) \in 
	\SFSCW$. 
\end{proof}

\begin{proof}[Proof of Theorem~\ref{thm:mainCW}]
	As we explain below, the proof follows immediately from Theorem~\ref{thm:main} and Theorem~\ref{thm:mainsimplifiedCW}, together with  Lemma~\ref{lem:CWmorphisms} that we will prove later in the section. 
	
	Recall from Theorem~\ref{thm:main}  that	the functors $\CYLcat: \SFScat \to  \Joincat^\circ$  and $\MAPcat:  \Joincat^\circ \to \SFScat$ restrict to mutually inverse isomorphisms. We want to show that they restrict to mutually inverse isomorphisms between $\SFSCWcat$ and $\JoinCWcat^\circ$. The statement that they restrict to mutually inverse bijections between $\Obj(\SFSCWcat)$ and $\Obj(\JoinCWcat^\circ)$ is Theorem~\ref{thm:mainsimplifiedCW}. 
	A morphism in $\SFScat$ between objects in $\SFSCWcat$ has the form
	\[
	\begin{tikzcd} X \arrow[r, "\sigma"] \arrow[d, "\phi_1"] & Y \arrow[d, "\phi_2"] \\ X' \arrow[r, "\sigma'"] &  Y', \end{tikzcd}
	\]
	where $\sigma$ and $\sigma'$ are strong $CW$-regular subdivisions and $\phi_1$ and $\phi_2$ are strong formal subdivisions. The corresponding morphism in $\Joincat^\circ$ between objects in $\JoinCWcat^\circ$ is the strong formal subdivision
	\[
	\phi: \Cyl(\sigma) \to \Cyl(\sigma'),
	\]
	\[
	\phi(z) = \begin{cases}
		\phi_1(z) &\textrm{ if } z \in X, \\
		\phi_2(z) &\textrm{ if } z \in Y.
	\end{cases}
	\]
	We need to show that the morphism determined by $\phi_1,\phi_2$ lies in $\SFSCWcat$ if and only if the morphism  determined by $\phi$ lies in $\JoinCWcat^\circ$. 
	By Definition~\ref{def:CWcategorify}, both these conditions have the requirement that $\phi_2$ is a strong $CW$-regular subdivision, and we need to show that $\phi_1$ is a strong $CW$-regular subdivision if and only if $\phi$ is a strong $CW$-regular subdivision. This is the content of Lemma~\ref{lem:CWmorphisms}.	
\end{proof}

\begin{remark}\label{rem:converseCW}
	Suppose that the converse statement in Lemma~\ref{lem:composeCW} holds without the assumption that $\rank(\sigma) = 0$. 
	Then in the proof of Theorem~\ref{thm:mainCW} above, if we assume that $\phi_1$ is a strong $CW$-regular subdivision, then it follows that $\phi_2$ is a strong $CW$-regular subdivision. 
	Indeed, the composition $\phi_2 \circ \sigma =  \sigma' \circ \phi_1$ is a strong $CW$-regular subdivision by Lemma~\ref{lem:composeCW}, and hence $\phi_2$ is a strong $CW$-regular subdivision, again by Lemma~\ref{lem:composeCW}. Then one could simplify Definition~\ref{def:CWcategorify} by not requiring that a morphism  in $\JoinCWcat$ induced by a strong $CW$-regular subdivision $\phi: \Gamma \to \Gamma'$ restricts to a strong $CW$-regular subdivision $\Gamma_{\ge q} \to \Gamma_{\ge q'}'$. 
	
\end{remark}

\begin{lemma}\label{lem:KXKYstar}
	Let  $\sigma: X \to Y$ be a strong $CW$-regular subdivision of rank $r$ between $CW$-posets $X$ and $Y$  with rank functions $\rho_X$ and $\rho_Y$ respectively. 	Let 	$\K_X$ and $\K_Y$ be  
	regular $CW$-complexes with  face posets $X$ and $Y$ respectively, and open cells $\{ e_x : x \in X \}$ and $\{ e_y : y \in X \}$ respectively. Then there exists a homeomorphism
	\[
	\psi: |\K_X| \to |\K_Y \star \partial \Delta^r|,
	\]
	such that $\psi(e_x) \subset e_{\sigma(x)} \star |\partial \Delta^r|$ for all $x \in X$. 
\end{lemma}
\begin{proof}
	Let $\{ e_z : z \in \partial B_{r + 1} \}$ denote the open cells of the regular $CW$-sphere $\partial \Delta^r$ of dimension $r - 1$ (see Example~\ref{ex:joinpoints}). We use the notation of Section~\ref{ss:regularCW}. In particular,  $\K_Y \star \partial \Delta^r$ has open cells $\{ e_{y,z} : (y,z) \in Y \times \partial B_{r + 1} \}$. If $Y = B_0$, then the result follows from Remark~\ref{rem:CWalt}.  
	
	Assume that $Y \neq B_0$. 
	Fix a maximal element $y \neq \hat{0}_Y$ of $Y$, and consider the lower order ideal $I = Y \smallsetminus \{ y \}$. By Remark~\ref{rem:restrictsfsCW},  $\sigma$ restricts to a strong $CW$-regular subdivision $\sigma^{-1}(I) \to I$. 
	By induction on $|Y|$ and using the $Y = B_0$ case above, there exists a 
	homeomorphism $\psi_I: |\K_{\sigma^{-1}(I)}| \to |\K_I \star \partial \Delta^r|$ such that $\psi_I(e_x) \subset e_{\sigma(x)} \star |\partial \Delta^r|$ for all $x \in \sigma^{-1}(I)$.

	We want to extend $\psi_I$ to construct $\psi$.
	By definition, $|\K_X| \smallsetminus |\K_{\sigma^{-1}(I)}| = |\K_{X,\le y}| \smallsetminus |\K_{X,< y}|$, and $|\K_Y \star \partial \Delta^r| \smallsetminus |\K_I \star \partial \Delta^r| =  e_y \star |\partial \Delta^r|$. 
	Recall that the closed cell corresponding to $y$ is $\overline{e_y} = \K_{[\hat{0}_Y, y]}$ with boundary $\partial \overline{e_y} = \K_{[\hat{0}_Y, y)}$. 
	Observe that $\psi_I$ restricts to a homeomorphism $\widehat{\psi}_I: |\K_{X,<y}| \to |\partial \overline{e_y} \star \partial \Delta^r|$. 
	Suppose that we can construct  a homeomorphism $\widehat{\psi}: \vert\K_{X,\le y} \vert \to \vert \overline{e_y} \star \partial \Delta^r\vert$ that extends $\widehat{\psi}_I$. 
	We have the following commutative diagrams where the vertical maps are inclusions. 
	\[
	\begin{tikzcd} \vert\K_{\sigma^{-1}(I)}\vert \arrow[r, "\psi_I"]  & \vert\K_I \star \partial \Delta^r\vert  \\ \vert\K_{X,<y} \vert \arrow[hookrightarrow,u] \arrow[hookrightarrow,d] \arrow[r, "\widehat{\psi}_I"] &  \vert \partial \overline{e_y} \star \partial \Delta^r\vert \arrow[hookrightarrow,u]\arrow[hookrightarrow,d] \\ 
		\vert\K_{X,\le y} \vert \arrow[r,"\widehat{\psi}"] & \vert \overline{e_y} \star \partial \Delta^r\vert. \end{tikzcd}
	\]
	Then $\psi_I$ and $\widehat{\psi}$ induce a bijection 	$\psi: |\K_X| \to |\K_Y \star \partial \Delta^r|$. In fact, since the restrictions of $\psi$ and $\psi^{-1}$ to any closed cell are continuous, we deduce that $\psi$ is a homeomorphism. It follows from our construction that $\psi(e_x) \subset e_{\sigma(x)} \star |\partial \Delta^r|$ for all $x \in X$. 
	
	It remains to construct $\widehat{\psi}$. 
	By Example~\ref{ex:joinballs} and Remark~\ref{rem:CWalt}, 
	$\K_{X,\le y}$ is a regular $CW$-ball with boundary $\K_{X,< y}$ that is homeomorphic to the regular $CW$-ball $\overline{e_y} \star \partial \Delta^r$  with boundary $\partial \overline{e_y} \star \partial \Delta^r$. By  Remark~\ref{rem:Alexander}, Alexander's trick implies that the homeomorphism $\widehat{\psi}_I: |\K_{X,<y}| \to |\partial \overline{e_y} \star \partial \Delta^r|$ extends to a homeomorphism $\widehat{\psi}: \vert\K_{X,\le y} \vert \to \vert \overline{e_y} \star \partial \Delta^r\vert$, as desired.

\end{proof}

\begin{lemma}\label{lem:CWmappingcylinder}
	Let  $\sigma: X \to Y$ be a strong $CW$-regular subdivision of rank $r$ between $CW$-posets $X$ and $Y$  with rank functions $\rho_X$ and $\rho_Y$ respectively. Let $\Gamma = \Cyl(\sigma)$ be the 
	corresponding non-Hausdorff mapping cylinder. Then $\Gamma$ is a $CW$-poset. Moreover, if	$\K_Y$ and $\K_\Gamma$ are  
	regular $CW$-complexes with  face posets $Y$ and $\Gamma$ respectively, then there is a homeomorphism between $|\K_\Gamma|$ and  $|\K_Y \star \Delta^r|$ that restricts to a homeomorphism between $|\K_X|$ and  $|\K_Y \star \partial  \Delta^r|$. 
\end{lemma}
\begin{proof}
	Let 	$\K_X$ and $\K_Y$ be  
	regular $CW$-complexes with  face posets $X$ and $Y$ respectively, and open cells $\{ e_x : x \in X \}$ and $\{ e_y : y \in X \}$ respectively.
	Let $\{ e_z : z \in B_{r + 1} \}$ denote the open cells of the regular $CW$-ball $\Delta^r$ of dimension $r$ (see Example~\ref{ex:joinpoints}). We use the notation of Section~\ref{ss:regularCW}. In particular,  $\K_Y \star \Delta^r$ has open cells $\{ e_{y,z} : (y,z) \in Y \times B_{r + 1} \}$. 
	
	We will define $\K_\Gamma = \{ \hat{e}_z : z \in \Gamma \}$ to be an alternative regular $CW$-structure for $|\K_Y \star  \Delta^r|$.  By Lemma~\ref{lem:KXKYstar}, we may fix a homeomorphism
	\[
	\psi: |\K_X| \to |\K_Y \star \partial \Delta^r|,
	\]
	such that $\psi(e_x) \subset e_{\sigma(x)} \star |\partial \Delta^r|$ for all $x \in X$. For $z \in \Gamma$, we define
	\[
	\hat{e}_z = \begin{cases}
		\psi(e_z) &\textrm{if } z \in X, \\
		e_{z,\hat{1}_{B_r}} &\textrm{if } z \in Y. 
	\end{cases}
	\]
	Observe that $\hat{e}_z$ is an open cell for all $z \in \Gamma$, and 
	$|\K_Y \star \Delta^r|$ is the disjoint union of $\{ \hat{e}_z : z \in \Gamma \}$.  We need to check that the corresponding closed cells $\overline{\hat{e}_z}$ are a union of the open cells $\{ \hat{e}_{z'} : z' \in \Gamma, z' \le z \}$. 	For any $x \in X$, the closure  $\overline{\hat{e}_x} = \overline{\psi(e_x)}$ is 
	the union of $\{ \hat{e}_{x'} = \psi(e_{x'}) : x' \in [\hat{0}_X, x] \subset X \} = \{ \hat{e}_z : z \in \Gamma, z \le x \}$. 	For any $y \in Y$, $\overline{\hat{e}_y} = \overline{e_{y,\hat{1}_{B_r}}}$ is the union of 
	$\{ \hat{e}_{y'} = e_{y',\hat{1}_{B_r}}  : y' \in [\hat{0}_Y,y] \subset Y \}$ and
	$\{ e_{y'} \star |\partial \Delta^r|  : y' \in [\hat{0}_Y,y] \subset Y \}$. Observe that $e_{y'} \star |\partial \Delta^r|$ is the union of $\{ \psi(e_x) : x \in X, \sigma(x) = y' \}$.  We conclude that 
	$\overline{\hat{e}_y}$ is the union of $\{ \hat{e}_z : z \in \Gamma, z \le y \}$, as desired.

\end{proof}

Suppose that $\K_Y$ is a regular $CW$-ball with face poset $Y$. 
Recall from Example~\ref{ex:CWball} that $Y$ is near-Eulerian and $\K_{\partial Y}$ is the boundary of $\K_Y$.

\begin{corollary}\label{cor:CWcyl}
	Let  $\sigma: X \to Y$ be a strong $CW$-regular subdivision of rank $r$ between $CW$-posets $X$ and $Y$  with rank functions $\rho_X$ and $\rho_Y$ respectively. Let $\Gamma = \Cyl(\sigma)$ be the 
	corresponding non-Hausdorff mapping cylinder. Let	$\K_Y$ and $\K_\Gamma$ be  
	regular $CW$-complexes with  face posets $Y$ and $\Gamma$ respectively. 
	Suppose that $\K_Y$ is a regular $CW$-ball. 
	Let 
	$\tilde{\sigma}: \partial X \to \partial Y$ be the restriction of $\sigma$, where $\partial X = \sigma^{-1}(\partial Y)$. 
	Then $\K_\Gamma$ is a regular $CW$-ball of dimension $\rank(X)$ with boundary  the union of the open cells in the regular $CW$-balls $\K_{\Cyl(\tilde{\sigma})}$ and $\K_X$ of dimension $\rank(X) - 1$, which have common boundary $\K_{\partial X}$. 
\end{corollary}
\begin{proof}
	Consider the homeomorphism $\psi: |\K_\Gamma| \to |\K_Y \star \Delta^r|$ constructed in
	Lemma~\ref{lem:CWmappingcylinder}, with $\psi(|\K_X|) = |\K_Y \star \partial \Delta^r|$. 
	By Remark~\ref{rem:restrictsfsCW}, $\tilde{\sigma}: \partial X \to \partial Y$ is a strong $CW$-regular subdivision. By Lemma~\ref{lem:CWmappingcylinder}, there is a 
	homeomorphism from $\psi': |\K_{\Cyl(\tilde{\sigma}})| \to
	|\K_{\partial Y} \star \Delta^r|$ such that 
	$\psi'(|\K_{\partial X}|) = 	|\K_{\partial Y} \star \partial \Delta^r|$. 
	In fact, 
	it follows from the proof of 
	Lemma~\ref{lem:CWmappingcylinder} that $\psi'$ is obtained by restricting   $\psi$.  
	
	By Example~\ref{ex:joinballs}, $\K_Y \star \Delta^r$ is  a regular $CW$-ball of dimension $\rank(X) = \rank(Y) + r$ with boundary  the union of the open cells in $\K_{\partial Y} \star \Delta^r$ and $\K_Y \star \partial \Delta^r$. Moreover, $\K_{\partial Y} \star \Delta^r$ and $\K_Y \star \partial \Delta^r$ are regular $CW$-balls of dimension $\rank(X) - 1$ with common boundary 
	$\K_{\partial Y} \star \partial \Delta^r$. 
	The result now follows by applying $\psi^{-1}$ to the terms above.
	
\end{proof}

\begin{proof}[Proof of Lemma~\ref{lem:composeCW}]
	Let $r = \rank(\sigma)$ and $s = \rank(\tau)$. Consider an element $z \in Z$. Then  $\sigma$ restricts to a strong $CW$-regular subdivision $X_{\le z} = (\tau \circ \sigma)^{-1}([\hat{0}_Z,z]) \to Y_{\le z} = \tau^{-1}([\hat{0}_Z,z])$.
	By Lemma~\ref{lem:KXKYstar}, there exists a  homeomorphism
	$\psi: |\K_{X,\le z}| \to |\K_{Y,\le z} \star \partial \Delta^r|$
	such that $\psi(e_x) \subset e_{\sigma(x)} \star |\partial \Delta^r|$ for all $x \in X$ with $(\tau \circ \sigma)(x) \le z$.   In particular, $\psi$ restricts to a homeomorphism $|\K_{X,< z}| \cong |\K_{Y,< z} \star \partial \Delta^r|$ when $z \neq \hat{0}_Z$.
	
	If $r = 0$, then $\rho_X(\hat{0}_X) = \rho_Y(\hat{0}_Y)$, and  $\psi$ is a homeomorphism from $|\K_{X,\le z}|$ to $|\K_{Y,\le z}|$ that restricts to a  homeomorphism $|\K_{X,< z}| \cong |\K_{Y,< z}|$ when $z \neq \hat{0}_Z$. It then follows from Definition~\ref{def:CWsubdivision} that if 
	$\tau \circ \sigma$ is strong $CW$-regular subdivision, then 
    $\tau$ is a strong $CW$-regular subdivision. 
	
	Assume that $\tau$ is a strong $CW$-regular subdivisions.
Then $\tau$ restricts to a strong $CW$-regular subdivision $ Y_{\le z} = \tau^{-1}([\hat{0}_Z,z]) \to [\hat{0}_Z,z]$. 
	By Remark~\ref{rem:CWalt}, 
	there exists a homeomorphism $\psi': |\K_{Y,\le z}| \to \overline{e_z} \star |\partial \Delta^s|$ that restricts to a homeomorphism 
	$|\K_{Y,< z}| \to \partial \overline{e_z} \star |\partial \Delta^s|$ when $z \neq \hat{0}_Z$.  Then $\psi'$ induces a  homeomorphism
	$|\K_{Y,\le z} \star \partial \Delta^r| \cong \overline{e_z} \star |\partial \Delta^s \star \partial \Delta^r|$ that restrictions to a homeomorphism $|\K_{Y,< z} \star \partial \Delta^r| \cong \partial \overline{e_z} \star |\partial \Delta^s \star \partial \Delta^r|$ if $z \neq \hat{0}_Z$. By Example~\ref{ex:joinballs}, $|\partial \Delta^s \star \partial \Delta^r|$ is a sphere of dimension $r + s - 1$, and hence 
	$|\partial \Delta^s \star \partial \Delta^r| \cong |\partial \Delta^{r + s}|$. Composing with $\psi$, 
	we deduce that 	there exists a homeomorphism $|\K_{X,\le z}| \cong \overline{e_z} \star |\partial \Delta^{r + s}|$ that restricts to a homeomorphism 
	$|\K_{X,< z}| \cong \partial \overline{e_z} \star |\partial \Delta^{r + s}|$ when $z \neq \hat{0}_Z$. The result now follows from Remark~\ref{rem:CWalt}.

\end{proof}

\begin{lemma}\label{lem:CWconverse}
	Let  $\sigma: X \to Y$ be a strong formal subdivision  between lower Eulerian posets $X$ and $Y$  with rank functions $\rho_X$ and $\rho_Y$ respectively. Let $\Gamma = \Cyl(\sigma)$ be the 
	corresponding non-Hausdorff mapping cylinder. 
	Suppose that $Y$ are $\Gamma$ are $CW$-posets. Then $X$ is a $CW$-poset and $\sigma$ is a strong $CW$-regular subdivision. 
\end{lemma}
\begin{proof}
	Let $\K_\Gamma$ be a  
	regular $CW$-complex with  face poset $\Gamma$. 
	Since $X$ is a lower order ideal of $\Gamma$, we may consider the corresponding subcomplex $\K_X$ of $\K_\Gamma$ with face poset $X$. Let $r = \rank(\sigma)$.
	
	Consider an element $y \in Y$. We need to show that  
	$\dim \K_{X,\le y} = r - 1 + \rho_Y(\hat{0}_Y,y)$, and
	either
	\begin{enumerate}
		
		\item $y = \hat{0}_Y$ and $\K_{X,\le y}$ is a regular $CW$-sphere, or, 
		\item $y \neq \hat{0}_Y$ and $\K_{X,\le y}$ is a regular $CW$-ball with boundary $\K_{X,< y}$. 
		%
	\end{enumerate}
	
	By Remark~\ref{rem:restrictsfs}, $\sigma$ restricts to a strong formal subdivision  $\sigma': X_{\le y} \to [\hat{0}_Y,y]$ with  $\Cyl(\sigma') = [\hat{0}_\Gamma,y] \subset \Gamma$. After possibly replacing $\sigma$ with $\sigma'$, we may assume that $\Gamma$ is Eulerian and $y = \hat{1}_\Gamma$.

	Then $\K_\Gamma$ 
	is a  regular $CW$-ball of dimension $r +  \rho_Y(\hat{0}_Y,y)$. If $y = \hat{0}_Y$, then 
	$\K_{X, \le y}$ is the boundary of $\K_\Gamma$, and hence is a regular $CW$-sphere of dimension $r - 1$, as desired. In particular, this establishes the result when $Y = B_0$. 
	
	Assume that $y \neq \hat{0}_Y$. 
	By Remark~\ref{rem:restrictsfs},  $\sigma$ restricts to a strong formal subdivision $\widehat{\sigma}: X_{< y} \to [\hat{0}_Y,y)$. Observe that $\Cyl(\widehat{\sigma})$ is a lower order ideal of $\Gamma$. By induction on $|Y|$ and using the case $Y = B_0$ established above, we may assume that $\widehat{\sigma}$ is a strong $CW$-regular subdivision. By Lemma~\ref{lem:CWmappingcylinder}, 
	there is a homeomorphism 
	$|\K_{\Cyl(\widehat{\sigma})}| \cong |\K_{[\hat{0}_Y,y)} \star \Delta^r|$
	that restricts to a homeomorphism $|\K_{X,< y}| \cong |\K_{[\hat{0}_Y,y)}  \star \partial  \Delta^r|$, where $\K_{[\hat{0}_Y,y)}$ is a regular $CW$-sphere of dimension $\rho_Y(\hat{0}_Y,y) - 2$ with face poset $[\hat{0}_Y,y)$. By Example~\ref{ex:joinballs}, 
	$\K_{[\hat{0}_Y,y)} \star \Delta^r$ is a regular $CW$-ball of dimension 
	$r  - 1 + \rho_Y(\hat{0}_Y,y)$ with boundary $\K_{[\hat{0}_Y,y)} \star \partial \Delta^r$. We conclude that $\K_{\Cyl(\widehat{\sigma})}$ is a regular $CW$-ball of dimension $r - 1 + \rho_Y(\hat{0}_Y,y)$ with boundary 
	$\K_{X,< y}$. 
	
	Observe that $|\partial \K_\Gamma|$ is a sphere of dimension  $r - 1 + \rho_Y(\hat{0}_Y,y)$, and is
	the union of 
	$|\K_{X, \le y}|$ and $|\K_{\Cyl(\widehat{\sigma})}|$. Also, the intersection of 	$|\K_{X, \le y}|$ and $|\K_{\Cyl(\widehat{\sigma})}|$ is $|\K_{X, < y}|$. Hence $|\K_{X, \le y}|$ is obtained from a sphere of dimension $r - 1 + \rho_Y(\hat{0}_Y,y)$ by removing the interior of a $r - 1 + \rho_Y(\hat{0}_Y,y)$-dimensional disk (embedded as a topological submanifold). It follows that $\K_{X, \le y}$ is a regular $CW$-ball of dimension $r - 1 + \rho_Y(\hat{0}_Y,y)$ with boundary $\K_{X, < y}$; for example, this follows from the generalized Schoenflies theorem \cite{MazurEmbeddings,BrownProof}. 
\end{proof}

\begin{lemma}\label{lem:CWmorphisms}
	Let $X, X', Y, Y'$ be $CW$-posets with rank functions $\rho_X, \rho_{X'}, \rho_Y, \rho_{Y'}$ respectively. Consider a commutative diagram
	\[
	\begin{tikzcd} X \arrow[r, "\sigma"] \arrow[d, "\phi_1"] & Y \arrow[d, "\phi_2"] \\ X' \arrow[r, "\sigma'"] &  Y', \end{tikzcd}
	\]
	where $\sigma$, $\sigma'$, and $\phi_2$ are strong $CW$-regular subdivisions, and 
	$\phi_1$ is a strong formal subdivision. 
	Consider the function
	\[
	\phi: \Cyl(\sigma) \to \Cyl(\sigma'),
	\]
	\[
	\phi(z) = \begin{cases}
		\phi_1(z) &\textrm{ if } z \in X, \\
		\phi_2(z) &\textrm{ if } z \in Y.
	\end{cases}
	\]
	Then $\phi_1$ is a strong $CW$-regular subdivision if and only if $\phi$ is a strong $CW$-regular subdivision. 
\end{lemma}
\begin{proof}
	Let $\Gamma = \Cyl(\sigma)$ and $\Gamma' = \Cyl(\sigma')$ with corresponding rank functions $\rho_\Gamma$ and $\rho_{\Gamma'}$ respectively determined by \eqref{eq:rhoCyl}. By Lemma~\ref{lem:CWmappingcylinder}, there are regular $CW$-complexes $\K_\Gamma$ and $\K_{\Gamma'}$ with face posets $\Gamma$ and $\Gamma'$ respectively. Let $\K_X$ and $\K_{X'}$ be the subcomplexes of $\K_\Gamma$ and $\K_{\Gamma'}$ corresponding to $X$ and $X'$ respectively. Let $\K_Y$ and $\K_{Y'}$ be regular $CW$-complexes with face posets $Y$ and $Y'$ respectively. Observe that $\phi$ is a strong formal subdivision by Lemma~\ref{lem:CYLwelldefined}. Let $r = \rank(\sigma)$ and $r_2 = \rank(\phi_2)$. 
	Recall that for any $z' \in \Gamma'$, we write $\Gamma_{\le z'} = \phi^{-1}([\hat{0}_{\Gamma'},z'])$ and $\Gamma_{< z'} = \phi^{-1}([\hat{0}_{\Gamma'},z'))$, with corresponding subcomplexes 
	$\K_{\Gamma, \le z'}$ and $\K_{\Gamma, < z'}$ of $\K_\Gamma$. 
	Similarly, recall that for any $y' \in Y'$, we write $X_{\le y'} = ( \phi_2 \circ \sigma)^{-1}([\hat{0}_{Y'},y'])$ and  $X_{< y'} = ( \phi_2 \circ \sigma)^{-1}([\hat{0}_{Y'},y'))$, with corresponding subcomplexes 
	$\K_{X, \le y'}$ and $\K_{X, < y'}$ of $\K_X$. Also, 
	$Y_{\le y'} = \phi_2^{-1}([\hat{0}_{Y'},y'])$ and 	$Y_{< y'} = \phi_2^{-1}([\hat{0}_{Y'},y'))$, with corresponding subcomplexes 
	$\K_{Y, \le y'}$ and $\K_{Y, < y'}$ of $\K_Y$.

	
	Assume that $\phi$ is a strong $CW$-regular subdivision. Since $X'$ is a lower order ideal of $\Gamma'$ and $\phi^{-1}(X') = X$, 
	$\phi_1$ is a strong $CW$-regular subdivision by Remark~\ref{rem:restrictsfsCW}.

	Conversely,	suppose that $\phi_1$ is a strong $CW$-regular subdivision. 
	Consider an element $z' \in \Gamma'$. We need to show that $\dim \K_{\Gamma,\le z'} = \rho_{\Gamma'}(z') - \rho_{\Gamma}(\hat{0}_\Gamma) - 1 = \rho_{Y'}(z') - \rho_X(\hat{0}_X) = \rho_{Y'}(\hat{0}_{Y'},z') + r + r_2$,
	and	either
	\begin{enumerate}
		
		\item $z' = \hat{0}_{\Gamma'}$ and $\K_{\Gamma,\le z'}$ is a regular $CW$-sphere, or, 
		\item $z' \neq \hat{0}_{\Gamma'}$ and $\K_{\Gamma,\le z'}$ is a regular $CW$-ball with boundary $\K_{\Gamma,< z'}$. 
		%
	\end{enumerate}
	If $z' \in X'$ then $\K_{\Gamma,\le z'} = \K_{X,\le z'}$ and the result follows from the fact that $\phi_1$ is a strong $CW$-regular subdivision.
	Assume that $z' \in Y'$. Note that $z' \neq \hat{0}_\Gamma$. 
	Restricting the above commutative diagram gives a  commutative diagram
	\[
	\begin{tikzcd}  X_{\le z'} \arrow[r, "\sigma"] \arrow[d, "\phi_1"] & Y_{\le z'} \arrow[d, "\phi_2"] \\ X_{\le z'}'  \arrow[r, "\sigma'"] &  {[}\hat{0}_{Y'},z'], \end{tikzcd}
	\]
	where 
	all the maps are strong $CW$-regular subdivisions by Remark~\ref{rem:restrictsfsCW}.  
	Then $\K_{[\hat{0}_{Y'},z']}$ is a regular $CW$-ball of dimension $\rho_{Y'}(\hat{0}_{Y'},z') - 1$ with boundary $\K_{[\hat{0}_{Y'},z')}$. By Corollary~\ref{cor:CWcyl} applied to the restriction of $\phi_2$ above, $\K_{Y,\le z'}$ is a regular $CW$-ball of dimension $\rho_{Y'}(\hat{0}_{Y'},z') + r_2 - 1$ with boundary $\K_{Y,< z}$. By Corollary~\ref{cor:CWcyl} applied to the restriction of $\sigma$ above, $\K_{\Gamma,\le z'}$ is a regular $CW$-ball of dimension $\rho_{Y'}(\hat{0}_{Y'},z') + r + r_2$ with boundary  the union of the open cells in the regular $CW$-balls $\K_{\Cyl(\tilde{\sigma})}$ and $\K_{X, \le z'}$, which have common boundary $\K_{X, < z}$. Here  $\tilde{\sigma}:  X_{< z'} \to Y_{< z'}$
	is the restriction of $\sigma$. The boundary of $\K_{\Gamma,\le z'}$ is the union of the cells of $\K_\Gamma$ indexed by 
	$\{ x \in X : x \le z' \in \Gamma \}$ and $\{ y \in Y : y < z' \in \Gamma \}$. That is,  $\partial \K_{\Gamma,\le z'} =  \K_{\Gamma,< z'}$, as desired.

\end{proof}

\section{Generalizations and further questions}\label{sec:generalizations}

In this section we discuss further questions and possible generalizations. 



\begin{question}
	Are there versions of Theorem~\ref{thm:intromain} for other interesting classes of posets?
\end{question}

For example, Theorem~\ref{thm:intromainCW} is a version of  Theorem~\ref{thm:intromain} for $CW$-posets.
In Section~\ref{ss:locallyEulerian}, 
we will outline how one may replace the lower Eulerian condition by the locally Eulerian condition. 
One could also try to relax the Eulerian condition. For example, one might consider ranked posets that are thin (see the discussion before the proof of Lemma~\ref{lem:hatsigma}). 

Recall that a poset $B$ with a unique minimal element is \emph{lower Gorenstein*} if every interval of positive rank is Gorenstein*. One may ask for a version of 
 Theorem~\ref{thm:main} for lower Gorenstein* posets. A suggestion for the replacement of the notion of a strong formal subdivision is the following definition,  which should be compared to Ehrenborg and Karu's definition of subdivision in Remark~\ref{rem:KEdef}.
 
 \begin{definition}\label{def:Gorensteinstarsubdiv}
 	Let $\sigma: X \to Y$ be an order-preserving, rank-increasing, surjective function between lower Gorenstein* posets $X$ and $Y$ with rank functions $\rho_X$  and $\rho_Y$ respectively.  
 	Then $\sigma$ is a \emph{Gorenstein* subdivision} if
 	for any $y$ in $Y$, $X_{\le y}$ is a lower Gorenstein* poset of rank $\rho_Y(y) - \rho_X(\hat{0}_X)$,  
 	and either 
 	\begin{enumerate}
 		\item $y = \hat{0}_Y$ and $X_{\le y}$ is the boundary of
 		a Gorenstein* poset, or,
 		\item\label{i:nearGorenstein} $y \neq \hat{0}_Y$ and $X_{\le y}$ is near-Gorenstein* with boundary $X_{< y}$.
 	\end{enumerate}
 \end{definition} 
 By Remark~\ref{rem:CWregularisformal}, strong $CW$-regular subdivisions are Gorenstein* subdivisions. 
  With the notation of Definition~\ref{def:Gorensteinstarsubdiv},  Lemma~\ref{lem:GorensteinstarsubdivEK} states that if $\sigma: X \to Y$ is a strong formal subdivision between lower Gorenstein* posets and the non-Hausdorff mapping cylinder $\Cyl(\sigma)$ is lower Gorenstein*, then $\sigma$ is a Gorenstein* subdivision. 
  

%
%
%
%
%

Recall that in Section~\ref{ss:basicexamples} we studied strong formal subdivisions to either $B_0$ or $B_1$. Observe that $B_0$ and $B_1$ have the property that they do not admit strong formal subdivisions to lower Eulerian posets with strictly fewer elements. Recall that $\EulPos$ is the category with elements given by  pairs $(B,\rho_B)$, where $B$ is a lower Eulerian poset and $\rho_B$ is a rank function for $B$, and with
morphisms given by strong formal subdivisions. Lemma~\ref{lem:parity}, Corollary~\ref{cor:boundaryEulerian}, and Corollary~\ref{cor:nearEulerian}  give constraints on the morphisms in $\EulPos$. 

%

\begin{question}
	Which lower Eulerian posets $X$ have the property that they do not admit a strong formal subdivision $\sigma: X \to Y$ for some $Y$ with $|Y| < |X|$ (equivalently, any strong formal subdivision $\sigma: X \to Y$ is an isomorphism of posets)? 
	What can be said about the structure of the category $\EulPos$?
\end{question}

Recall that  Corollary~\ref{cor:oddequaleven} gives an obstruction to a lower Eulerian poset containing a non-minimal join-admissible element. 

\begin{question}
	Consider a lower Eulerian poset $\Gamma$. Are there other obstructions to
		$\Gamma$ containing a non-minimal join-admissible element, or, equivalently,  $\Gamma$ being the non-Hausdorff mapping cylinder of some strong formal subdivision between lower Eulerian posets?
\end{question} 

Recall that Proposition~\ref{prop:cdsubdivision} gives a formula for the $cd$-index of the lower Eulerian poset $\Gamma$. Note that the coefficients on the left hand side of \eqref{eq:cdformula} are integers, while the right hand side of \eqref{eq:cdformula} involves a factor of $\frac{1}{2}$. 

\begin{question}
	Is there a version of Proposition~\ref{prop:cdsubdivision} which does not involve a factor of $\frac{1}{2}$? See the special case 
	of Example~\ref{ex:derivation}. 
\end{question}

The following question was considered in Remark~\ref{rem:converseCW}. 

\begin{question}
		Does the converse statement in Lemma~\ref{lem:composeCW} hold without the assumption that $\sigma$ has rank $0$ (c.f. Lemma~\ref{lem:composesfs})? That is,
		let  $\sigma : X \to Y$ and $\tau: Y \to Z$ be strong formal subdivisions between lower Eulerian posets $X$, $Y$, $Z$ with   rank functions $\rho_X$, $\rho_Y$, $\rho_Z$  respectively.
	Assume that $\sigma$ and  $\tau \circ \sigma$ are strong $CW$-regular subdivisions. Is $\tau$ necessarily  a strong $CW$-regular subdivision? 
\end{question}

As we will explain, taking duals of posets induces an involution on a subset of strong formal subdivisions that is possibly worthy of further investigation. 
Let $\Gamma$ be a lower Eulerian poset. We say that an element $q$ of $\Gamma$ is \emph{meet-admissible} if for any element $z$ in $\Gamma$, 
the meet (or  greatest lower bound) $z \wedge q$ exists. 
Let
$\JoinIdealLW' = \{ (\Gamma,\rho_\Gamma,q) \in \JoinIdealLW : \Gamma \textrm{ is Eulerian}, q \notin \{ \hat{0}_\Gamma, \hat{1}_\Gamma \}, q \textrm{ is meet-admissible}\}$. Given an element $(\Gamma,\rho_\Gamma,q) \in \JoinIdealLW'$, we may consider the dual poset $\Gamma^*$ with induced rank function $\rho_{\Gamma*}$. Then $(\Gamma^*,\rho_{\Gamma^*},q) \in \JoinIdealLW'$. We have an induced involution on $\JoinIdealLW'$ sending 
$(\Gamma,\rho_\Gamma,q)$ to $(\Gamma^*,\rho_{\Gamma^*},q)$. 
By Theorem~\ref{thm:mainsimplified}, we have an induced involution on a subset of $\SFS$. For example, recall that in Section~\ref{ss:facelatticepolytope} we considered the case when $\Gamma$ is the face lattice of a polytope. Then we obtain an involution on all strong formal subdivisions induced by a projective, surjective morphism from a fan to a nonzero pointed cone.    

\subsection{Subdivisions of locally Eulerian posets}\label{ss:locallyEulerian}

We outline how our main results extend to strong formal subdivisions between locally Eulerian posets. Firstly, the definition of strong formal subdivision given in \cite[Definition~3.17]{KatzStapledon16} requires that the posets only be locally Eulerian (in Definition~\ref{def:sfs} simply replace the word `lower' by `locally' to get the definition).  

We first introduce a definition. Let $B$ be a poset, and let $I$ be a nonempty  upper order ideal of $B$. Given an element $z \in B$, let 
$z \vee I$ be the unique minimal element of $\Gamma_{\ge z} \cap I$ if it exists. We say that $I$  is \emph{join-admissible} if $z \vee I$ exists for all $z \in B$. 

As in Definition~\ref{def:bijections}, let $\SFS$ be the set of all strong formal subdivisions $\sigma: X \to Y$ between  locally Eulerian posets $X$ and $Y$  with rank functions $\rho_X$ and $\rho_Y$ respectively. 	
Let $\JoinIdealLW$ be the set of triples $(\Gamma, \rho_\Gamma, I)$,
where $\Gamma$ is a locally Eulerian poset  with rank function $\rho_\Gamma$  and 
$I$ is a join-admissible  upper order ideal of $\Gamma$. 
Let $\JoinIdealLW^\circ \subset \JoinIdealLW$ be the subset of all triples $(\Gamma, \rho_\Gamma, I)$ such that $I$ doesn't contain any minimal elements of $\Gamma$.
Our proof of Theorem~\ref{thm:mainsimplified} in Section~\ref{sec:proof} generalizes to give mutually inverse bijections $\CYL$ and $\MAP$ between
	$\SFS$ and $\JoinIdealLW^\circ$. 
Explicitly, 
\[
\CYL(\sigma: X \to Y) = (\Cyl(\sigma), \rho_{\Cyl(\sigma)},Y), 
\]
where $\Cyl(\sigma)$ is the non-Hausdorff mapping cylinder of $\sigma$, and 
\begin{equation*}
	\rho_{\Cyl(\sigma)}(z) =  \begin{cases}
		\rho_X(z) &\textrm{ if } z \in X, \\
		\rho_Y(z) + 1 &\textrm{ if } z \in Y. 
	\end{cases}
\end{equation*}
Also,
\begin{equation*}
	\MAP(\Gamma,\rho_\Gamma,I) : \Gamma \smallsetminus I \to I,
\end{equation*}
\[
x \mapsto x \vee I,
\]
where the rank functions on $\Gamma \smallsetminus I$ and $I$ are  the  corresponding restrictions of $\rho_\Gamma$ shifted by $0$ and $-1$ respectively.

We explain how to
recover  Theorem~\ref{thm:mainsimplified} from this result. With the notation above, since $\sigma$ is surjective, it  
follows from Definition~\ref{def:nonHausdorffmc}  that the minimal elements of $\Gamma = \Cyl(\sigma)$ are precisely the minimal elements of $X$. In particular, $\Gamma$ is lower Eulerian if and only if $X$ is lower Eulerian. 	By Remark~\ref{rem:zerotozero}, if $X$ is lower Eulerian, then $Y$ is lower Eulerian. In this case, let $q = \hat{0}_Y$ be the unique minimal element of $Y$, so that $I = \Gamma_{\ge q}$. Since $I \neq \Gamma$, we have $q \neq \hat{0}_\Gamma$. 

Finally, we want a functorial version of the above correspondence, generalizing Theorem~\ref{thm:main}. As in Definition~\ref{def:categorify}, 
Let $\EulPos$ be the category with elements given by  pairs $(B,\rho_B)$, where $B$ is a locally Eulerian poset and $\rho_B$ is a rank function for $B$, and with
morphisms given by strong formal subdivisions, and let
 $\SFScat$  to be the corresponding arrow category $\Arr(\EulPos)$.
	Let $\Joincat$ be a category with objects $\JoinIdealLW$. We describe the morphisms. 	A morphism  between objects $(\Gamma, \rho_\Gamma, I)$ and $(\Gamma', \rho_{\Gamma'}, I')$ in  $\Joincat$ is 
	a strong formal subdivision 
	$\phi: \Gamma \to \Gamma'$ such that $\phi^{-1}(I') = I$ 
	and
	$\phi(z \vee I) = \phi(z) \vee I'$ for all $z \in \Gamma \smallsetminus I$. 
	Let $\Joincat^\circ$ be the full subcategory of $\Joincat$ with objects 
	$\JoinIdealLW^\circ$. 
	Our proof of Theorem~\ref{thm:main} in Section~\ref{sec:proof} generalizes to give mutually inverse isomorphisms $\CYLcat$ and $\MAPcat$ between
	$\SFScat$ and $\Joincat^\circ$. 
	Above, if one restricts only to $\Gamma$ that are lower Eulerian in the definition of $\Joincat^\circ$, then this agrees with the definition of $\Joincat^\circ$ in Definition~\ref{def:categorify} by
	 the last paragraph of the proof of Lemma~\ref{lem:MAPwelldefined}.

\bibliography{KLStheory}
\bibliographystyle{amsalpha}

\end{document}